\documentclass[11pt]{article}

\vspace{-1cm}

\title{\Large\bf 
New structure on the quantum alcove model with\\ applications to representation theory and Schubert calculus%
\footnote{Key words and phrases: quantum Yang-Baxter move, quantum Bruhat graph, quantum alcove model, level-zero Demazure module, semi-infinite flag manifold. 
\newline
Mathematics Subject Classification 2010: Primary 05E10; Secondary 14N15, 14M15.}%
}
\author{%
Takafumi Kouno \\
 \small Department of Mathematics, Tokyo Institute of Technology, \\
 \small 2-12-1 Oh-okayama, Meguro-ku, Tokyo 152-8551, Japan \\
 \small (e-mail: {\tt kouno.t.ab@m.titech.ac.jp}) \\[5mm]
Cristian Lenart \\ 
 \small Department of Mathematics and Statistics, State University of New York at Albany, \\ 
 \small Albany, NY 12222, U.S.A. \\ 
 \small (e-mail: {\tt clenart@albany.edu}) \\[5mm]
Satoshi Naito \\ 
 \small Department of Mathematics, Tokyo Institute of Technology, \\
 \small 2-12-1 Oh-okayama, Meguro-ku, Tokyo 152-8551, Japan \\
 \small (e-mail: {\tt naito@math.titech.ac.jp}) \\[5mm]
}
\date{}

\usepackage[truedimen,margin=23truemm]{geometry}
\usepackage{hyperref}

\usepackage{amsmath, amssymb}
\usepackage{amsthm}

\usepackage{bm}

\numberwithin{equation}{section}

\theoremstyle{plain}
\newtheorem{lem}{Lemma}[section]
\newtheorem{prop}[lem]{Proposition}
\newtheorem{thm}[lem]{Theorem}
\newtheorem{cor}[lem]{Corollary}
\newtheorem{conj}[lem]{Conjecture}

\newtheorem{ithm}{Theorem}

\theoremstyle{definition}
\newtheorem{dfn}[lem]{Definition}

\theoremstyle{remark}
\newtheorem{ex}[lem]{Example}
\newtheorem{rem}[lem]{Remark}
\newtheorem{rems}[lem]{Remarks}

\newcommand{\BZ}{\mathbb{Z}}
\newcommand{\BR}{\mathbb{R}}
\newcommand{\BC}{\mathbb{C}}

\newcommand{\sQ}{\mathsf{Q}}
\newcommand{\sR}{\mathsf{R}}
\newcommand{\sS}{\mathsf{S}}
\newcommand{\sT}{\mathsf{T}}

\newcommand{\q}{\mathsf{q}}

\newcommand{\CA}{\mathcal{A}}
\newcommand{\CP}{\mathcal{P}}

\newcommand{\Fg}{\mathfrak{g}}
\newcommand{\Fh}{\mathfrak{h}}

\newcommand{\bk}{\mathbf{k}}
\newcommand{\bp}{\mathbf{p}}
\newcommand{\bq}{\mathbf{q}}
\newcommand{\br}{\mathbf{r}}
\newcommand{\bG}{\mathbf{G}}
\newcommand{\bchi}{\bm{\chi}}
\newcommand{\bpsi}{\bm{\psi}}
\newcommand{\bomega}{\bm{\omega}}

\newcommand{\ve}{\varepsilon}
\newcommand{\vpi}{\varpi}

\newcommand{\vtl}{\vartriangleleft}

\DeclareMathOperator{\QBG}{QBG}
\DeclareMathOperator{\Hom}{Hom}
\DeclareMathOperator{\gch}{gch}
\DeclareMathOperator{\ed}{end}
\DeclareMathOperator{\wt}{wt}
\DeclareMathOperator{\height}{height}
\DeclareMathOperator{\coheight}{coheight}
\DeclareMathOperator{\down}{down}
\DeclareMathOperator{\nega}{neg}
\DeclareMathOperator{\sgn}{sgn}
\DeclareMathOperator{\Par}{Par}
\DeclareMathOperator{\bPar}{\overline{\Par}}

\newcommand{\af}{\mathrm{af}}

\newcommand{\wti}[1]{\widetilde{#1}}
\newcommand{\wh}[1]{\widehat{#1}}
\newcommand{\pair}[2]{\langle #1, #2 \rangle}
\newcommand{\lrpair}[2]{\left\langle #1, #2 \right\rangle}
\newcommand{\bra}[1]{[\![ #1 ]\!]}
\newcommand{\pra}[1]{(\!( #1 )\!)}

\newcommand{\E}[1]{\mathrm{(E #1 )}}

\makeatletter
\renewcommand\section{\@startsection{section}{1}{0pt}
{-3.5ex plus -1ex minus -.2ex}{1.0ex plus .2ex}{\large\bf}}
\renewcommand\subsection{\@startsection{subsection}{1}{0pt}
{2.5ex plus 1ex minus .2ex}{-1em}{\bf}}
\makeatother

\newenvironment{enu}{%
 \begin{enumerate}%
}{\end{enumerate}}

\begin{document}

\maketitle

\vspace{-13mm}

\begin{abstract} The quantum alcove model associated to a dominant weight plays an important 
role in many branches of mathematics, such as combinatorial representation theory, the theory of Macdonald polynomials, and Schubert calculus.
For a dominant weight, it is proved by Lenart-Lubovsky that the quantum alcove model does not depend on the choice of a reduced alcove path, which is a shortest path of alcoves from the fundamental one to its translation by the given dominant weight. This is established through quantum Yang-Baxter moves, which biject the objects of the model associated with two such alcove paths, and can be viewed as a generalization of jeu de taquin slides to arbitrary root systems. 
The purpose of this paper is to give a generalization of quantum Yang-Baxter moves to the quantum alcove model corresponding to an arbitrary weight, which was used to express a general Chevalley formula in the equivariant $K$-group of semi-infinite flag manifolds. The generalized quantum Yang-Baxter moves give rise to a ``sijection'' (bijection between signed sets), and are shown to preserve certain important statistics, including weights and heights. As an application, we prove that the generating function of these statistics does not depend on the choice of a reduced alcove path. Also, we obtain an identity for the graded characters of Demazure submodules of level-zero extremal weight modules over a quantum affine algebra, which can be thought of as a representation-theoretic analogue of the mentioned Chevalley formula. Other applications and some open problems involving ``signed crystals'' are discussed.
\end{abstract}

%===================%
% START SECTION 01 %
%===================%

\section{Introduction.}

The quantum alcove model was introduced in \cite{LL1}. In \cite{LNSSS2} it was proved to be a uniform model for tensor products of single-column Kirillov-Reshetikhin crystals of quantum affine algebras, and its relevance to the theory of Macdonald polynomials was also discussed. Crystals are colored directed graphs encoding the structure of quantum algebra representations when the quantum parameter goes to 0~\cite{Kas}. The quantum alcove model generalizes the alcove model in \cite{LP1}, which has a similar representation theoretic application~\cite{LP2}.

Let $\lambda$ be a dominant weight, and let $\Gamma$ be a reduced $\lambda$-chain (of roots), or equivalently, a shortest path of alcoves from the fundamental one to its translation by $\lambda$ (see  \cite{LP1}).
One associates with $\Gamma$ (viewed as a sequence) a certain family $\CA(\Gamma)$ of subsets of its indices, called admissible subsets.
Here we remark that there are two or more reduced $\lambda$-chains in general, and therefore the model is not uniquely  determined (for the fixed dominant weight $\lambda$). 
However, for any two reduced $\lambda$-chains $\Gamma_1$ and $\Gamma_2$, there exists a bijection between $\CA(\Gamma_1)$ and $\CA(\Gamma_2)$
which preserves the corresponding crystal operators, as well as some important statistics: weights, heights, $\down(\cdot)$, and $\ed(\cdot)$;
the precise definitions of these statistics are given in Section~\ref{subsec:def_quantum_alcove_models}.  
The construction of this bijection was given in \cite{LL2} in terms of so-called quantum Yang-Baxter moves, which are
explicitly described by reduction to the rank 2 root systems. The main idea is the following: given $\Gamma_1$ and $\Gamma_2$ as above, it is known from \cite{LP1} that $\Gamma_2$ is obtained from $\Gamma_1$ by repeated application of a certain procedure called a ``Yang-Baxter transformation'', see Section~\ref{subsec:YB-transformations}; hence it suffices to construct a bijection (i.e., a quantum Yang-Baxter move) between $\CA(\Gamma_1)$ and $\CA(\Gamma_2)$ when $\Gamma_1$ and $\Gamma_2$ are related by a Yang-Baxter transformation. 

The quantum Yang-Baxter moves generalize the Yang-Baxter moves for the alcove model, which were defined and studied in~\cite{L}. It is pointed out in~\cite{LL2} that the quantum Yang-Baxter moves realize the combinatorial $R$-matrix, namely the (unique) affine crystal isomorphism permuting factors in a tensor product of single-column Kirillov-Reshetikhin crystals. It is also explained that these moves can be viewed as a generalization of jeu de taquin slides (for semistandard Young tableaux, relevant to type $A$) to arbitrary root systems. 

For an arbitrary (not necessarily dominant) weight $\lambda$, we also consider a (not necessarily reduced) $\lambda$-chain $\Gamma$ (of not necessarily positive roots). For an arbitrary element $w$ of the finite Weyl group $W$, let $\CA(w, \Gamma)$ denote the collection of $w$-admissible subsets. This generalization is introduced in \cite{LNS} to describe the Chevalley formula for the equivariant $K$-group of semi-infinite flag manifolds, and for the equivariant quantum $K$-theory of flag manifolds $G/B$  (both of arbitrary type), cf. also \cite{KNS, NOS}. 
We also define statistics $\wt(A)$, $\height(A)$, $\down(A)$, $\ed(A)$ for $A \in \CA(w, \Gamma)$ 
in the same way as for $A \in \CA(\Gamma)=\CA(e,\Gamma)$ with $\lambda$ dominant, where $e\in W$ is the identity; 
in addition, we define $n(A) \in \BZ_{\ge 0}$. 

Our main result is the existence of a bijection between $\CA(w, \Gamma_1)$ and $\CA(w, \Gamma_2)$ which preserves the statistics above, where $\Gamma_2$ is obtained from $\Gamma_1$ by a Yang-Baxter transformation.
%
%%%%%%%%%%%%%%%%%%%%
% thm:YB-move_intro %
%%%%%%%%%%%%%%%%%%%%
%
\begin{ithm}[Theorems~\ref{thm:YB-move} and \ref{thm:wt_height_preserving}]\label{thm:YB-move_intro}
Let $\CA(w, \Gamma_{1})$ and $\CA(w, \Gamma_{2})$ be quantum alcove models associated to the same weight 
such that $\Gamma_{2}$ is obtained from $\Gamma_{1}$ by a Yang-Baxter transformation. 
Then, there exist subsets $\CA_{0}(w, \Gamma_{1}) \subset \CA(w, \Gamma_{1})$ and $\CA_{0}(w,\Gamma_{2}) \subset \CA(w, \Gamma_{2})$ which satisfy the following. 
\begin{enu}
\item There exists a ``sign-preserving'' bijection $Y: \CA_{0}(w, \Gamma_{1}) \rightarrow \CA_{0}(w, \Gamma_{2})$,  
which also preserves the statistics $\wt(\cdot)$, $\height(\cdot)$, $\down(\cdot)$, and $\ed(\cdot)$. 

\item If we set $\CA_{0}^{C}(w, \Gamma_{1}) := \CA(w, \Gamma_{1}) \setminus \CA_{0}(w, \Gamma_{1})$ and $\CA_{0}^{C}(w, \Gamma_{2}) := \CA(w, \Gamma_{2}) \setminus \CA_{0}(w, \Gamma_{2})$, 
then there exists a ``sign-reversing'' involution $I_{1}$ (resp., $I_{2}$) on $\CA_{0}^{C}(w, \Gamma_{1})$ (resp., $\CA_{0}^{C}(w, \Gamma_{2})$) 
which preserves the statistics $\wt(\cdot)$, $\height(\cdot)$, $\down(\cdot)$, and $\ed(\cdot)$. 
\end{enu}
\end{ithm}

The map $Y$ in Theorem~\ref{thm:YB-move_intro} can be viewed as a generalization of the bijection described in terms of quantum Yang-Baxter moves when $\lambda$ is a dominant weight.
Although the map $Y$ is not a bijection from the whole of $\CA(w, \Gamma_{1})$ onto the whole of $\CA(w, \Gamma_{2})$, 
there exist nice involutions outside the domain of $Y$ and outside the image of $Y$. 
If we regard $\CA(w, \Gamma_{i})$, $i = 1, 2$, as a signed set equipped with the sign function $A \mapsto (-1)^{n(A)}$, 
then the collection $(I_{1}, I_{2}, Y)$ of maps is a ``sijection'' (i.e., a signed bijection) $\CA(w, \Gamma_{1}) \Rightarrow \CA(w, \Gamma_{2})$ 
which preserves $\wt(\cdot)$, $\height(\cdot)$, $\down(\cdot)$, and $\ed(\cdot)$;
the notion of a sijection was introduced in \cite[Section~2]{FK}. 

Recall that an element of the affine Weyl group $W_{\af}$ can be written $x = w t_{\xi}$,  with $w$ in the finite Weyl group $W$ 
and $\xi$ in the coroot lattice $Q^{\vee}$. For $\CA(w, \Gamma)$, with $\Gamma$ a (not necessarily reduced) $\lambda$-chain for an arbitrary weight $\lambda$, and $x = w t_{\xi}$ in $W_{\af}$, 
we define a generating function $\bG_{\Gamma}(x)$ of the statistics $\wt(\cdot)$, $\ed(\cdot)$, $\height(\cdot)$, and $\down(\cdot)$ as follows: 
\begin{equation*}
\bG_{\Gamma}(x) := \sum_{A \in \CA(w, \Gamma)} (-1)^{n(A)} q^{-\height(A)-\pair{\lambda}{\xi}} e^{\wt(A)} \ed(A)t_{\xi+\down(A)}. 
\end{equation*} 
We also think of $\bG_{\Gamma}$ as a linear function on the group algebra of $W_{\af}$ with the coefficients introduced above. 
In the case that $\lambda$ is a dominant weight and $x = e$, this function is a refinement of the specialization at $t = 0$ of the symmetric Macdonald polynomial $P_{\lambda}(q, t)$, since we know from \cite[Theorem~7.9]{LNSSS2} that 
\begin{equation*}
P_{\lambda}(q, 0) = \sum_{A \in \CA(\Gamma)} q^{\height(A)} e^{\wt(A)}. 
\end{equation*}
There is a similar relationship in the case of nonsymmetric Macdonald polynomials~\cite{LNSSS3}. 

The existence of our generalized quantum Yang-Baxter moves implies the independence of the generating function $\bG_{\Gamma}(x)$, and thus of the quantum alcove model for an arbitrary weight, from the associated chain of roots $\Gamma$. Here we need $\Gamma$ to be ``weakly reduced,'' which means that it does not contain both a simple root and its negative. 
\begin{ithm}[Theorem~\ref{thm:generating_function_reduced}]\label{thm:generating_function_reduced_intro}
Let $\lambda$ be an arbitrary weight, and $x\in W_{\af}$. Given weakly reduced $\lambda$-chains 
$\Gamma_{1}$ and $\Gamma_{2}$, 
we have $\bG_{\Gamma_1}(x) = \bG_{\Gamma_2}(x)$. 
\end{ithm}

We will now discuss several applications of Theorem~\ref{thm:generating_function_reduced_intro} and, implicitly, of the generalized quantum Yang-Baxter moves underlying it.

We give a combinatorial realization of the symmetry of the general Chevalley formula in~\cite{LNS} coming from commutativity in equivariant $K$-theory. Indeed, given arbitrary weights $\mu,\,\nu$, we can successively apply the Chevalley formula for the multiplication by the classes of the line bundles corresponding to them, in either order. The fact that the result is the same is expressed by the following identity, where $\Gamma_1$ is a $\mu$-chain, $\Gamma_2$ is a $\nu$-chain, and $\circ$ indicates composition:
\begin{equation}\label{comm_mu_nu_intro}
\bG_{\Gamma_1}\circ\bG_{\Gamma_2}(x) = \bG_{\Gamma_2}\circ\bG_{\Gamma_1}(x).
\end{equation}
It will be shown that~\eqref{comm_mu_nu_intro} is realized combinatorially via successive applications of the sijection in Theorem~\ref{thm:YB-move_intro}, assuming that the concatenation of $\Gamma_1 $ and $\Gamma_2$ is weakly reduced.

On another hand, we use Theorem~\ref{thm:generating_function_reduced_intro} to obtain an identity for the graded characters of Demazure submodules of level-zero extremal weight modules over a quantum affine algebra, which can be viewed as a representation-theoretic analogue of the general Chevalley formula in \cite{LNS}. 
For a dominant weight $\mu$ and an element $x$ of the affine Weyl group, 
let $V_{x}^{-}(\mu)$ denote the Demazure submodule of the level-zero extremal weight module $V(\mu)$ of extremal weight $\mu$ over the quantum affine algebra. 
For an arbitrary weight $\lambda$, let $\bPar(\lambda)$ denote the set of certain tuples $\bchi$ of partitions bounded by $\lambda$, to which we assign the quantities $|\bchi|$ and $\iota(\bchi)$; 
for the definitions of $\bPar(\lambda)$, $|\bchi|$, and $\iota(\bchi)$, see \eqref{eq:def_Par} and \eqref{eq:def_statistics_par} in Section~\ref{subsec:character_identity}. 

\begin{ithm}[Theorem~\ref{thm:PC-type_formula}]\label{thm:PC-type_formula_intro}
Let $\mu$ be a dominant weight, and $x = w t_{\xi}\in W_{\af}$. 
Take an arbitrary weight $\lambda$ such that $\mu + \lambda$ is dominant, 
and let $\Gamma$ be a reduced $\lambda$-chain. We have 
\begin{equation*}
\begin{split}
& \gch V_{x}^{-}(\mu + \lambda) = \\ 
& \sum_{A \in \CA(w, \Gamma)} \sum_{\bchi \in \bPar(\lambda)} (-1)^{n(A)} q^{-\height(A) - \pair{\lambda}{\xi} - |\bchi|} e^{\wt(A)} \gch V_{\ed(A)t_{\xi + \down(A) + \iota(\bchi)}}^{-}(\mu). 
\end{split}
\end{equation*}
\end{ithm}

The right-hand side of the above formula is proved to be identical to zero if $\mu + \lambda \notin P^{+}$, see Appendix~C. In the case $\mu=0$, this proof provides a combinatorial analog of the vanishing of the {0-th}~cohomology of the semi-infinite flag manifold for line bundles associated to weights that are
not dominant, see~\cite{KNS,NOS} and the details in Appendix C.

Here we should mention that, in \cite{FM}, an identity for generalized Weyl modules similar to the one in Theorem~\ref{thm:PC-type_formula_intro} is obtained in the case that $\lambda$ is a fundamental weight $\varpi_{i}$, $i \in I$; 
a generalized Weyl module can be viewed as the $\q = 1$ limit of a certain finite-dimensional quotient of a Demazure submodule of a level-zero extremal weight module over the quantum affine algebra $U_{\q}(\Fg_{\af})$ associated to the affine Lie algebra $\Fg_{\af}$ (see \cite{N} for an explicit relation between the graded characters of these modules).

The proof of the general Chevalley formula for semi-infinite flag manifolds in \cite{LNS} can be considerably simplified by using Theorem~\ref{thm:generating_function_reduced_intro}, in a similar way to the proof of Theorem~\ref{thm:PC-type_formula_intro}. Alternatively, the general Chevalley formula can be deduced from Theorem~\ref{thm:PC-type_formula_intro} by exactly the same argument as that in~\cite{KNS} and \cite{NOS}.
Conversely, the general Chevalley formula implies Theorem~\ref{thm:PC-type_formula_intro} for $\mu$ sufficiently dominant, but not for an arbitrary dominant $\mu$; in particular, we cannot set $\mu=0$. In this sense, Theorem~\ref{thm:PC-type_formula_intro} and the corresponding vanishing mentioned above are slightly stronger than the general Chevalley formula.

In parallel with the notion of a sijection, we can define a ``signed crystal,'' and then an isomorphism of signed crystals as a sijection commuting with the crystal operators. Using the crystal structure on admissible subsets corresponding to a dominant weight~\cite{LL1,LNSSS2}, as well as the generalized quantum Yang-Baxter moves, we propose a way to obtain a signed crystal structure on $\CA(w,\Gamma)$, where $\Gamma$ is a reduced $\lambda$-chain for an arbitrary weight $\lambda$. We then conjecture that the identities in Theorems~\ref{thm:generating_function_reduced_intro} and \ref{thm:PC-type_formula_intro}, as well as~\eqref{comm_mu_nu_intro}, can be lifted to isomorphisms of the underlying signed crystals. In fact, by phrasing~\eqref{comm_mu_nu_intro} as the invariance of a composite $\bG_{\Gamma_1}\circ\cdots\circ\bG_{\Gamma_p}(x)$ under permuting the maps $\bG_{\Gamma_\cdot}$, we would extend the realization of the combinatorial $R$-matrix in~\cite{LL2}. As in~\cite{LL2}, the key fact in all these conjectures is the commutation of the generalized quantum Yang-Baxter moves with the crystal operators, which is also conjectured.

In conclusion, the generalized quantum Yang-Baxter moves add very useful structure to the quantum alcove model for an arbitrary weight.

This paper is organized as follows. 
In Section~\ref{sec:basic}, we fix our basic notation, and recall the definitions and some properties of the quantum Bruhat graph and the quantum alcove model. 
In Section~\ref{sec:quantum_Yang-Baxter_moves}, we state our main result precisely; its proof is given in Section~\ref{sec:proofs}. 
Finally, we prove the equality between the generating functions associated to two reduced $\lambda$-chains, 
and derive the identity above for the graded characters of (level-zero) Demazure submodules in Section~\ref{sec:generating_function}. The other applications of the generalized quantum Yang-Baxter moves mentioned above are also discussed in Section~\ref{sec:generating_function}. 

\subsection*{Acknowledgements.}
The authors would like to thank Daisuke Sagaki for helpful discussions. 
The authors are also grateful to Shunsuke Tsuchioka for useful advices. 
T.K. was supported in part by JSPS Grant-in-Aid for Scientific Research 20J12058. 
C.L. was supported in part by the NSF grant DMS-1855592.
S.N. was supported in part by JSPS Grant-in-Aid for Scientific Research (B) 16H03920.

%===================%
% START SECTION 02 %
%===================%

%%%%%%%%%%%
% sec:basic %
%%%%%%%%%%%

\section{Preliminaries.}\label{sec:basic}
We fix our basic notation of this paper. 
Also, we recall the definitions and some properties of the quantum Bruhat graph and the quantum alcove model. 

%=========================%
% START SUBSECTION 0201 %
%=========================%

\subsection{Basic notation.}

Throughout this paper, let $\Fg$ be a complex simple Lie algebra or the complex Lie algebra of type $A_{1} \times A_{1}$, 
with Cartan subalgebra $\Fh \subset \Fg$. 
We denote by $\pair{\cdot}{\cdot}$ the canonical pairing of $\Fh$ and $\Fh^{\ast} = \Hom_{\BC}(\Fh, \BC)$. 

Let $\Delta$ denote the root system of $\Fg$, 
with $\Delta^{+} \subset \Delta$ the set of all positive roots. 
Let $I$ be the set of indices of the Dynkin diagram of $\Fg$, 
and let $\alpha_{i}$, $i \in I$, be the simple roots of $\Delta$. 
For $\alpha \in \Delta$, we define $\sgn(\alpha) \in \{1, -1\}$ and $|\alpha| \in \Delta^{+}$ by 
\begin{align*}
\sgn(\alpha) &:= \begin{cases} 1 & \text{if } \alpha \in \Delta^{+}, \\ -1 & \text{if } \alpha \in -\Delta^{+}, \end{cases} \\ 
|\alpha| &:= \sgn(\alpha)\alpha. 
\end{align*}
We set $Q := \sum_{i \in I} \BZ \alpha_{i}$ and 
$Q^{\vee} := \sum_{i \in I} \BZ \alpha_{i}^{\vee}$, 
where $\alpha^{\vee}$ is the coroot of $\alpha \in \Delta$; 
also, we set $Q^{\vee, +} := \sum_{i \in I} \BZ_{\ge 0} \alpha_{i}^{\vee}$. 
Let $W = \langle s_{i} \mid i \in I \rangle$ be the Weyl group of $\Fg$, with length function $\ell: W \rightarrow \BZ_{\geq 0}$ 
and the longest element $w_{\circ} \in W$; here, for $\alpha \in \Delta^{+}$, $s_{\alpha} \in W$ denotes the reflection 
corresponding to $\alpha$, and $s_{i} = s_{\alpha_{i}}$ is the simple reflection for $i \in I$. 

For each $i \in I$, let $\vpi_{i}$ denote the fundamental weight corresponding to $\alpha_{i}$. 
Let $P := \sum_{i \in I} \BZ \vpi_{i}$ be the weight lattice of $\Fg$, with 
$P^{+} := \sum_{i \in I} \BZ_{\ge 0} \vpi_{i}$ the set of dominant weights; 
also, we set $\Fh^{\ast}_{\BR} := P \otimes_{\BZ} \BR$.

%=========================%
% START SUBSECTION 0202 %
%=========================%

\subsection{The quantum Bruhat graph.}
We recall the definition of the \emph{quantum Bruhat graph}, introduced in \cite{BFP}.
We set $\rho := (1/2) \sum_{\alpha \in \Delta^{+}} \alpha$. 
\begin{dfn}[{\cite[Definition~6.1]{BFP}}]
The \emph{quantum Bruhat graph} $\QBG(W)$ is the $\Delta^{+}$-labeled directed graph 
whose vertices are the elements of $W$ and whose edges are of the following form: 
$x \xrightarrow{\alpha} y$, with $x, y \in W$ and $\alpha \in \Delta^{+}$, such that $y = xs_{\alpha}$ 
and either of the following (B) or (Q) holds: 
\begin{itemize}
\item[(B)] $\ell(y) = \ell(x) + 1$; 
\item[(Q)] $\ell(y) = \ell(x) -2\pair{\rho}{\alpha^{\vee}} + 1$. 
\end{itemize}
If (B) (resp., (Q)) holds, then the edge $x \xrightarrow{\alpha} y$ is called a \emph{Bruhat edge} (resp., \emph{quantum edge}). 
\end{dfn}

Let $\bp: w_{0} \xrightarrow{\beta_{1}} w_{1} \xrightarrow{\beta_{2}} \cdots \xrightarrow{\beta_{r}} w_{r}$ be a directed path in $\QBG(W)$. 
We set 
\begin{align*}
\ell(\bp) &:= r, \\ 
\ed(\bp) &:= w_{r}, \\ 
\wt(\bp) &:= \sum_{\substack{k \in \{1, \ldots, r\} \\ w_{k-1} \xrightarrow{\beta_{k}} w_{k} \text{ is a quantum edge}}} \beta_{k}^{\vee}.  
\end{align*}

\begin{dfn}[{\cite[(2.2)]{D}}]
A total order $\vtl$ on $\Delta^{+}$ is a \emph{reflection order} 
if for all $\alpha, \beta \in \Delta^{+}$ such that $\alpha+\beta \in \Delta^{+}$, 
either $\alpha \vtl \alpha+\beta \vtl \beta$ or $\beta \vtl \alpha+\beta \vtl \alpha$ holds. 
\end{dfn}

Let $\vtl$ be a reflection order on $\Delta^{+}$. 
A directed path $\bp$ in $\QBG(W)$ of the form: 
\begin{equation*}
\bp: w_{0} \xrightarrow{\beta_{1}} w_{1} \xrightarrow{\beta_{2}} \cdots \xrightarrow{\beta_{r}} w_{r},
\end{equation*}
with $\beta_{1} \vtl \cdots \vtl \beta_{r}$, is called a \emph{label-increasing} directed path with respect to $\vtl$. 
%
%%%%%%%%%%%%%%%%
% thm:shellability %
%%%%%%%%%%%%%%%%
%
\begin{thm}[{\cite[Theorem~6.4]{BFP}}]\label{thm:shellability}
Let $\vtl$ be a reflection order on $\Delta^{+}$. 
For all $v, w \in W$, there exists a unique label-increasing directed path from $v$ to $w$ in $\QBG(W)$ with respect to $\vtl$. 
Moreover, the unique label-increasing directed path from $v$ to $w$ has the minimum length. 
\end{thm}

The property of $\QBG(W)$ in Theorem~\ref{thm:shellability} is called the \emph{shellability}.
 
For all $v, w \in W$, 
there exists at least one shortest directed path $\bp$ from $v$ to $w$; 
we set 
\begin{equation*}
\ell(v \Rightarrow w) := \ell(\bp), \quad \wt(v \Rightarrow w) := \wt(\bp). 
\end{equation*}
Note that by \cite[Lemma~1\,(2)]{P} or \cite[Proposition~8.1]{LNSSS1}, $\wt(v \Rightarrow w)$ is well-defined. 

We consider a ``generalization'' of label-increasing directed paths. 
Let $\Pi = (\gamma_{1}, \ldots, \gamma_{r})$ be a sequence of roots, i.e., $\gamma_{1}, \ldots, \gamma_{r} \in \Delta$. 
Assume that $\gamma_{1}, \ldots, \gamma_{r}$ are distinct. 
Then we say that a directed path $\bp$ is \emph{$\Pi$-compatible} if $\bp$ is of the following form: 
%
%%%%%%%%%%%%%%%%%%
% eq:Pi-compatible %
%%%%%%%%%%%%%%%%%%
%
\begin{equation}\label{eq:Pi-compatible}
\bp: w_{0} \xrightarrow{|\gamma_{j_{1}}|} w_{1} \xrightarrow{|\gamma_{j_{2}}|} \cdots \xrightarrow{|\gamma_{j_{p}}|} w_{p}, 
\end{equation}
with $1 \le j_{1} < \cdots < j_{p} \le r$. 
For $w \in W$, we denote by $\CP(w, \Pi)$ the set of all $\Pi$-compatible directed paths in $\QBG(W)$ which start at $w$. 
\begin{rem}
If $\{\gamma_{1}, \ldots, \gamma_{r}\} \subset \Delta^{+}$, 
and if there exists a reflection order $\vtl$ on $\Delta^{+}$ such that $\gamma_{1} \vtl \cdots \vtl \gamma_{r}$, 
then a $\Pi$-compatible directed path in $\QBG(W)$ is a label-increasing directed path with respect to $\vtl$. 
\end{rem}

Let $\Pi = (\gamma_{1}, \ldots, \gamma_{r})$ be a sequence of root, with $\gamma_{1}, \ldots, \gamma_{r}$ not necessarily distinct. 
For a directed path $\bp$ of the form: 
\begin{equation*}
\bp: w_{0} \xrightarrow{|\gamma_{j_{1}}|} w_{1} \xrightarrow{|\gamma_{j_{2}}|} \cdots \xrightarrow{|\gamma_{j_{p}}|} w_{p}, 
\end{equation*}
with $1 \le j_{1} < \cdots < j_{p} \le r$, we define $\nega(\bp)$ by 
\begin{equation*}
\nega(\bp) := \# \{ k \in \{1, \ldots, p\} \mid \gamma_{j_{k}} \in -\Delta^{+} \}. 
\end{equation*}

%=========================%
% START SUBSECTION 0203 %
%=========================%

%
%%%%%%%%%%%%%%%%%%%%%%%%%%%%%%%%%
% subsec:def_quantum_alcove_models %
%%%%%%%%%%%%%%%%%%%%%%%%%%%%%%%%%
%
\subsection{The quantum alcove model.}\label{subsec:def_quantum_alcove_models}
We briefly review the quantum alcove model, introduced in \cite{LL1}.

First, we recall from \cite{LP1} the definition of alcove paths. 
For $\alpha \in \Delta$ and $k \in \BZ$, we set $H_{\alpha, k} := \{\nu \in \Fh_{\BR}^{\ast} \mid \pair{\nu}{\alpha^{\vee}} = k\}$; 
$H_{\alpha, k}$ is a hyperplane in $\Fh_{\BR}^{\ast}$. 
Also, for $\alpha \in \Delta$ and $k \in \BZ$, we denote by $s_{\alpha, k}$ the reflection with respect to $H_{\alpha, k}$. 
Note that $s_{\alpha, k}(\nu) = \nu - (\pair{\nu}{\alpha^{\vee}} - k)\alpha$ for $\nu \in \Fh_{\BR}^{\ast}$. 
Each connected component of the space 
\begin{equation*}
\Fh_{\BR}^{\ast} \setminus \bigcup_{\alpha \in \Delta^{+}, k \in \BZ} H_{\alpha, k}
\end{equation*}
is called an \emph{alcove}. 
If two alcoves $A$ and $B$ have a common wall, then we say that $A$ and $B$ are \emph{adjacent}. 
For adjacent alcoves $A$ and $B$, we write $A \xrightarrow{\beta} B$, $\beta \in \Delta$, 
if the common wall of $A$ and $B$ is contained in $H_{\beta, k}$ for some $k \in \BZ$, 
and $\beta$ points in the direction from $A$ to $B$. 

\begin{dfn}[{\cite[Definition~5.2]{LP1}}]
A sequence $(A_{0}, \ldots, A_{r})$ of alcoves is called an \emph{alcove path} if $A_{i-1}$ and $A_{i}$ are adjacent for all $i = 1, \ldots, r$. 
If the length $r$ of an alcove path $(A_{0}, \ldots, A_{r})$ is minimal among all alcove paths from $A_{0}$ to $A_{r}$, 
we say that $(A_{0}, \ldots, A_{r})$ is \emph{reduced}. 
\end{dfn}

Let $\theta \in \Delta^{+}$ denote the highest root of $\Delta$. The \emph{fundamental alcove} $A_{\circ}$ is defined by 
\begin{equation*}
A_{\circ} := \{ \nu \in \Fh_{\BR}^{\ast} \mid \pair{\nu}{\alpha_{i}^{\vee}} > 0 \text{ for all } i \in I \text{ and } \pair{\nu}{\theta^{\vee}} < 1 \}. 
\end{equation*}
Also, for $\lambda \in P$, we define $A_{\lambda}$ by 
\begin{equation*}
A_{\lambda} := A_{\circ} + \lambda = \{ \nu + \lambda \mid \nu \in A_{\circ} \}. 
\end{equation*}

\begin{dfn}[{\cite[Definition~5.4]{LP1}}]
Let $\lambda \in P$. 
A sequence $(\beta_{1}, \ldots, \beta_{r})$ of roots $\beta_{1}, \ldots, \beta_{r} \in \Delta$ 
is called a \emph{$\lambda$-chain} 
if there exists an alcove path $(A_{0}, \ldots, A_{r})$, with $A_{0} = A_{\circ}$ and $A_{r} = A_{-\lambda}$, such that 
\begin{equation*}
A_{\circ} = A_{0} \xrightarrow{-\beta_{1}} A_{1} \xrightarrow{-\beta_{2}} \cdots \xrightarrow{-\beta_{r}} A_{r} = A_{-\lambda}. 
\end{equation*}
\end{dfn}

Now, following \cite[Section~3.2]{LNS}, we review the quantum alcove model. 
\begin{dfn}[{\cite[Definition~17]{LNS}}]
Let $\lambda \in P$, and let $\Gamma = (\beta_{1}, \ldots, \beta_{r})$ be a $\lambda$-chain. 
Fix $w \in W$. A subset $A = \{j_{1} < \cdots < j_{p}\} \subset \{1, \ldots, r\}$ is said to be \emph{$w$-admissible} if 
\begin{equation*}
\bp(A): w = w_{0} \xrightarrow{|\beta_{j_{1}}|} w_{1} \xrightarrow{|\beta_{j_{2}}|} \cdots \xrightarrow{|\beta_{j_{p}}|} w_{p}
\end{equation*}
is a directed path in $\QBG(W)$. 
Let $\CA(w, \Gamma)$ denote the set of all $w$-admissible subsets of $\{1, \ldots, r\}$. 
\end{dfn}
\begin{rem}
The original definition of admissible subsets in \cite{LL1} is only for $w = e \in W$. 
The notion of $w$-admissible subsets for an arbitrary $w \in W$ is introduced in \cite{LNS}. 
\end{rem}

Let $\lambda \in P$, and let $\Gamma = (\beta_{1}, \ldots, \beta_{r})$ be a $\lambda$-chain.
By the definition of $\lambda$-chains, there exists an alcove path $(A_{\circ} = A_{0}, \ldots, A_{r} = A_{-\lambda})$ such that 
\begin{equation*}
A_{\circ} = A_{0} \xrightarrow{-\beta_{1}} A_{1} \xrightarrow{-\beta_{2}} \cdots \xrightarrow{-\beta_{r}} A_{r} = A_{-\lambda}. 
\end{equation*}
For $k = 1, \ldots, r$, we take $l_{k} \in \BZ$ such that $H_{\beta_{k}, -l_{k}}$ contains the common wall of $A_{k-1}$ and $A_{k}$, 
and set $\wti{l_{k}} := \pair{\lambda}{\beta_{k}^{\vee}} - l_{k}$. 

Fix $w \in W$. For $A = \{j_{1} < \cdots < j_{p}\} \in \CA(w, \Gamma)$, we set 
\begin{equation*}
\ed(A) := ws_{|\beta_{j_{1}}|} \cdots s_{|\beta_{j_{p}}|}, \quad \wt(A) := -ws_{\beta_{j_{1}}, -l_{j_{1}}} \cdots s_{\beta_{j_{p}}, -l_{j_{p}}} (-\lambda); 
\end{equation*} 
we call $\wt(A)$ the \emph{weight} of $A$.
Also, we define a subset $A^{-} \subset A$ by 
\begin{equation*}
A^{-} := \left\{ j_{k} \in A \ \middle| \ ws_{|\beta_{j_{1}}|} \cdots s_{|\beta_{j_{k-1}}|} \xrightarrow{|\beta_{j_{k}}|} ws_{|\beta_{j_{1}}|} \cdots s_{|\beta_{j_{k}}|} \text{ is a quantum edge} \right\},  
\end{equation*}
and set 
\begin{equation*}
\down(A) := \sum_{j \in A^{-}} |\beta_{j}|^{\vee}, \quad \height(A) := \sum_{j \in A^{-}} \sgn(\beta_{j})\wti{l_{j}}; 
\end{equation*}
note that $\ed(A) = \ed(\bp(A))$ and $\down(A) = \wt(\bp(A))$. 
In addition, we define $n(A) \in \BZ_{\ge 0}$ by $n(A) := \# \{ j \in A \mid \beta_{j} \in -\Delta^{+} \}$; 
note that $n(A) = \nega(\bp(A))$. 

\begin{rem}\label{rem:coheight}
In \cite[(31)]{LNSSS3}, an additional statistic, called \emph{coheight}, is introduced. 
Assume that $\lambda \in P^{+}$ and $w = e$. 
For $A \in \CA(e, \Gamma)$, $\coheight(A) \in \BZ_{\ge 0}$ is defined by 
\begin{equation}\label{eq:coheight}
\coheight(A) := \sum_{j \in A^{-}} l_{j}. 
\end{equation}
The coheight is used in \cite{LNSSS3} to describe 
the specialization at $t = 0$ of nonsymmetric Macdonald polynomials in terms of the quantum alcove model (see \cite[Theorems~29, 31]{LNSSS3}). 
\end{rem}

%===================%
% START SECTION 03 %
%===================%

%
%%%%%%%%%%%%%%%%%%%%%%%%%%%%%%%%
% sec:quantum_Yang-Baxter_moves %
%%%%%%%%%%%%%%%%%%%%%%%%%%%%%%%%
%
\section{Generalization of quantum Yang-Baxter moves.}\label{sec:quantum_Yang-Baxter_moves}
Quantum Yang-Baxter moves for a dominant weight are introduced in \cite{LL2}. 

%=========================%
% START SUBSECTION 0301 %
%=========================%

%
%%%%%%%%%%%%%%%%%%%%%%%%%%%
% subsec:YB-transformations %
%%%%%%%%%%%%%%%%%%%%%%%%%%%
%
\subsection{Yang-Baxter transformation of \texorpdfstring{$\lambda$}{lambda}-chains.}\label{subsec:YB-transformations}
Let $\lambda \in P$, and let $\Gamma = (\beta_{1}, \ldots, \beta_{r})$ be a $\lambda$-chain (of roots). 
The following procedure (YB) is called the \emph{Yang-Baxter transformation}: 
\begin{itemize}
\item[(YB)] Take a segment $(\beta_{t+1}, \ldots, \beta_{t+q})$ of $\Gamma$ of the form 
\begin{equation*}
(\beta_{t+1}, \ldots, \beta_{t+q}) = (\alpha, s_{\alpha}(\beta), s_{\alpha}s_{\beta}(\alpha), \ldots, s_{\beta}(\alpha), \beta)
\end{equation*}
for some $\alpha, \beta \in \Delta$ with $\pair{\alpha}{\beta^{\vee}} \le 0$, or equivalently $\pair{\beta}{\alpha^{\vee}} \le 0$, and $\alpha \not= -\beta$, 
and set 
\begin{equation*}
\Gamma' := (\beta_{1}, \beta_{2}, \ldots, \beta_{t}, \beta_{t+q}, \beta_{t+q-1}, \ldots, \beta_{t+1}, \beta_{t+q+1}, \beta_{t+q+2}, \ldots, \beta_{r}), 
\end{equation*}
i.e., reverse the segment $(\beta_{t+1}, \ldots, \beta_{t+q})$ of $\Gamma$. 
\end{itemize}

Also, we define a procedure (D), called \emph{deletion},
as follows: 
\begin{itemize}
\item[(D)] Take a segment $(\beta_{t+1}, \beta_{t+2})$ of $\Gamma$ of the form $(\beta_{t+1}, \beta_{t+2}) = (\beta, -\beta)$ for some $\beta \in \Delta$, 
and set $\Gamma' = (\beta_{1}, \ldots, \beta_{t}, \beta_{t+3}, \ldots, \beta_{q})$, i.e., delete the segment $(\beta_{t+1}, \beta_{t+2})$ of $\Gamma$. 
\end{itemize}

It is known that every $\lambda$-chain can be transformed into an arbitrary reduced $\lambda$-chain by repeated application of the procedures (YB) and (D) 
(see \cite[Remark~38]{LNS}, or \cite[Lemma~9.3]{LP1}). 

%=========================%
% START SUBSECTION 0302 %
%=========================%

\subsection{Quantum Yang-Baxter moves.}\label{subsec:main}
Let $\lambda \in P^{+}$ be a dominant weight, and 
let $\Gamma_1$, $\Gamma_2$ be $\lambda$-chains 
such that $\Gamma_{2}$ is obtained from $\Gamma_{1}$ by the Yang-Baxter transformation (YB). 
\emph{Quantum Yang-Baxter moves}, introduced in \cite[Section~3.1]{LL2}, 
give a bijection $\CA(e, \Gamma_{1}) \rightarrow \CA(e, \Gamma_{2})$ 
which preserves weights and heights. 

Our main result is the existence of a generalization of quantum Yang-Baxter moves for an arbitrary (not necessarily dominant) weight $\lambda \in P$ and an arbitrary $w \in W$. 

Let $\lambda \in P$ be an arbitrary weight, 
and let $\Gamma_{1}$ and $\Gamma_{2}$ be $\lambda$-chains such that $\Gamma_{2}$ is obtained from $\Gamma_{1}$ by the Yang-Baxter transformation (YB). 
If we write $\Gamma_{1} = (\beta_{1}, \ldots, \beta_{r})$ and $\Gamma_{2} = (\beta'_{1}, \ldots, \beta'_{r})$, then there exists $1 \le t \le r$ such that 
\begin{itemize}
\item $(\beta_{t+1}, \ldots, \beta_{t+q}) = (\alpha, s_{\alpha}(\beta), s_{\alpha}s_{\beta}(\alpha), \ldots, s_{\beta}(\alpha), \beta)$ for some $q \ge 1$ and some $\alpha, \beta \in \Delta$ such that $\pair{\alpha}{\beta^{\vee}} \le 0$ and $\alpha \not= -\beta$, 
\item $\Gamma_{2} = (\beta'_{1}, \ldots, \beta'_{r}) = (\beta_{1}, \beta_{2}, \ldots, \beta_{t}, \beta_{t+q}, \beta_{t+q-1}, \ldots, \beta_{t+1}, \beta_{t+q+1}, \beta_{t+q+2}, \ldots, \beta_{r})$. 
\end{itemize}

We take the alcove path $(A_{\circ} = A_{0}, \ldots, A_{r} = A_{-\lambda})$ corresponding to $\Gamma_{1}$, 
and take integers $l_{k} \in \BZ$ for $k = 1, \ldots, r$ such that for each $k = 1, \ldots, r$, 
the hyperplane $H_{\beta_{k}, -l_{k}}$ contains the common wall of $A_{k-1}$ and $A_{k}$. 
Also, we take the alcove path $(A_{\circ} = A'_{0}, \ldots, A'_{r} = A_{-\lambda})$ corresponding to $\Gamma_{2}$, 
and we take integers $l'_{k} \in \BZ$ for $k = 1, \ldots, r$ such that for each $k = 1, \ldots, r$, 
the hyperplane $H_{\beta'_{k}, -l'_{k}}$ contains the common wall of $A'_{k-1}$ and $A'_{k}$. 
Then it follows that $A'_{k} = A_{k}$ and $l'_{k} = l_{k}$ for $k = 1, \ldots, t, t+q+1, \ldots, r$, 
and that $l'_{t+p} = l_{t+q+1-p}$ for $p = 1, \ldots, q$. 

Now, we divide $\Gamma_{1}$ into three parts $\Gamma_{1}^{(1)}$, $\Gamma_{1}^{(2)}$, and $\Gamma_{1}^{(3)}$ as follows: 
%
%%%%%%%%%%%%%%%%%%%%
% eq:division_Gamma1 %
%%%%%%%%%%%%%%%%%%%%
%
\begin{equation}\label{eq:division_Gamma1}
\Gamma_{1}^{(1)} := (\beta_{1}, \ldots, \beta_{t}), \ \Gamma_{1}^{(2)} := (\beta_{t+1}, \ldots, \beta_{t+q}), \ \Gamma_{1}^{(3)} := (\beta_{t+q+1}, \ldots, \beta_{r}). 
\end{equation}
Also, we divide $\Gamma_{2}$ into three parts $\Gamma_{2}^{(1)}$, $\Gamma_{2}^{(2)}$, and $\Gamma_{2}^{(3)}$ as follows: 
%
%%%%%%%%%%%%%%%%%%%%
% eq:division_Gamma2 %
%%%%%%%%%%%%%%%%%%%%
%
\begin{equation}\label{eq:division_Gamma2}
\Gamma_{2}^{(1)} := (\beta'_{1}, \ldots, \beta'_{t}), \ \Gamma_{2}^{(2)} := (\beta'_{t+1}, \ldots, \beta'_{t+q}), \ \Gamma_{2}^{(3)} := (\beta'_{t+q+1}, \ldots, \beta'_{r}). 
\end{equation}
Note that $\Gamma_{1}^{(1)} = \Gamma_{2}^{(1)}$ and $\Gamma_{1}^{(3)} = \Gamma_{2}^{(3)}$; 
in addition, $\beta_{t+1}, \ldots, \beta_{t+q}$ are distinct. 
Next, let $w \in W$. For a $w$-admissible subset $A \in \CA(w, \Gamma_{1})$, we define $A^{(1)}$, $A^{(2)}$, and $A^{(3)}$ by 
%
%%%%%%%%%%%%%%
% eq:division_A %
%%%%%%%%%%%%%%
%
\begin{equation}\label{eq:division_A}
A^{(1)} := A \cap \{1, \ldots, t\}, \ A^{(2)} := A \cap \{t+1, \ldots, t+q\}, \ A^{(3)} := A \cap \{t+q+1, \ldots, r\}. 
\end{equation}
Also, for $B \in \CA(w, \Gamma_{2})$, we define $B^{(1)}$, $B^{(2)}$, and $B^{(3)}$ by 
%
%%%%%%%%%%%%%%
% eq:division_B %
%%%%%%%%%%%%%%
%
\begin{equation}\label{eq:division_B}
B^{(1)} := B \cap \{1, \ldots, t\}, \ B^{(2)} := B \cap \{t+1, \ldots, t+q\}, \ B^{(3)} := B \cap \{t+q+1, \ldots, r\}. 
\end{equation}

Unlike the case that $\lambda$ is dominant, there does not exist a bijection between $\CA(w, \Gamma_{1})$ and $\CA(w, \Gamma_{2})$ in general. 
\begin{ex}
Assume that $\Fg$ is of type $A_{2}$. 
We set $\Gamma_{1} := (\alpha_{2}, -\alpha_{1}, -\theta, -\alpha_{1})$ and $\Gamma_{2} := (-\theta, -\alpha_{1}, \alpha_{2}, -\alpha_{1})$, 
where $\theta = \alpha_{1} + \alpha_{2}$. 
Then we see that $\Gamma_{1}$ and $\Gamma_{2}$ are $(-2\varpi_{1}+\varpi_{2})$-chains 
such that $\Gamma_{2}$ is obtained from $\Gamma_{1}$ by a Yang-Baxter transformation (YB). 
Let $w = s_{2}$. By direct calculation, we have 
\begin{align*}
\CA(w, \Gamma_{1}) &= \{ \emptyset, \{1\}, \{2\}, \{3\}, \{4\}, \{1,2\}, \{1,4\}, \{2,4\}, \{3,4\}, \{1,2,3\}, \{1,2,4\}, \{1,2,3,4\} \}, \\[3mm] 
\CA(w, \Gamma_{2}) &= \{ \emptyset, \{1\}, \{2\}, \{3\}, \{4\}, \{1,2\}, \{1,3\}, \{1,4\}, \{2,3\}, \{2,4\}, \{3,4\}, \{1,2,3\}, \{1,2,4\}, \\ 
& \hspace*{7mm} \{1,3,4\}, \{2,3,4\}, \{1,2,3,4\} \}. 
\end{align*}
Hence we have $\# \CA(w, \Gamma_{1}) = 12$, while $\# \CA(w, \Gamma_{2}) = 16$. 
This shows that there does not exist a bijection $\CA(w, \Gamma_{1}) \rightarrow \CA(w, \Gamma_{2})$. 
\end{ex}

Thus, for a generalization of quantum Yang-Baxter moves, we need to consider certain subsets of $\CA(w, \Gamma_{1})$ and $\CA(w, \Gamma_{2})$. 
The following theorem is our main result; the proof is given in the next section. 
%
%%%%%%%%%%%%%%%
% thm:YB-move %
%%%%%%%%%%%%%%%
%
\begin{thm}\label{thm:YB-move}
There exist subsets $\CA_{0}(w, \Gamma_{1}) \subset \CA(w, \Gamma_{1})$ and $\CA_{0}(w,\Gamma_{2}) \subset \CA(w, \Gamma_{2})$ which satisfy the following. 
\begin{enu}
\item There exists a bijection $Y: \CA_{0}(w, \Gamma_{1}) \rightarrow \CA_{0}(w, \Gamma_{2})$ 
such that for all $A \in \CA_{0}(w, \Gamma_{1})$, it holds that 
\begin{itemize}
\item $(Y(A))^{(1)} = A^{(1)}$, $\ed((Y(A))^{(2)}) = \ed(A^{(2)})$, $(Y(A))^{(3)} = A^{(3)}$, 
\item $\down(Y(A)) = \down(A)$, and 
\item $(-1)^{n(Y(A))} = (-1)^{n(A)}$. 
\end{itemize}
\item For $k = 1, 2$, we set $\CA_{0}^{C}(w, \Gamma_{k}) := \CA(w, \Gamma_{k}) \setminus \CA_{0}(w, \Gamma_{k})$ . 
There exists an involution $I_{k}$ on $\CA_{0}^{C}(w, \Gamma_{k})$ 
such that for all $A \in \CA_{0}^{C}(w, \Gamma_{k})$, it holds that 
\begin{itemize}
\item $(I_{k}(A))^{(1)} = A^{(1)}$, $\ed((I_{k}(A))^{(2)}) = \ed(A^{(2)})$, $(I_{k}(A))^{(3)} = A^{(3)}$, 
\item $\down(I_{k}(A)) = \down(A)$, and 
\item $(-1)^{n(I_{k}(A))} = -(-1)^{n(A)}$. 
\end{itemize}
\end{enu}
\end{thm}

\begin{rem}
In order to explain our maps $Y$, $I_{1}$, and $I_{2}$ in Theorem~\ref{thm:YB-move}, 
we have a useful notion, called a \emph{sijection}, introduced in \cite{FK}; 
for the definition of sijections, see \cite[Section~2]{FK}. 
For sets $S$, $T$ equipped with sign functions $S \rightarrow \{ \pm 1\}$, $T \rightarrow \{ \pm 1 \}$, 
a sijection from $S$ to $T$ is the collection $(\iota_{S}, \iota_{T}, \varphi)$ 
of a sign-reversing involution $\iota_{S}$ on a subset $S_{0}$ of $S$, 
a sign-reversing involution $\iota_{T}$ on a subset $T_{0}$ of $T$, 
and a sign-preserving bijection $\varphi$ from $S \setminus S_{0}$ to $T \setminus T_{0}$ (see \cite[p.9]{FK}).
In this terminology, our collection $(I_{1}, I_{2}, Y)$ in Theorem~\ref{thm:YB-move} is a sijection from $\CA(w, \Gamma_{1})$ to $\CA(w, \Gamma_{2})$. 
This sijection can be thought of as a generalization of quantum Yang-Baxter moves. 
\end{rem}

As in the case that $\lambda$ is dominant, we can prove that the maps $Y$, $I_{1}$, and $I_{2}$ preserve weights and heights. 
%
%%%%%%%%%%%%%%%%%%%%%%%%
% thm:wt_height_preserving %
%%%%%%%%%%%%%%%%%%%%%%%%
%
\begin{thm}\label{thm:wt_height_preserving}
\begin{enu}
\item For all $A \in \CA_{0}(w, \Gamma_{1})$, it holds that $\wt(Y(A)) = \wt(A)$ and $\height(Y(A)) = \height(A)$. 
\item Let $k = 1, 2$. For all $A \in \CA_{0}^{C}(w, \Gamma_{k})$, it holds that $\wt(I_{k}(A)) = \wt(A)$ and $\height(I_{k}(A)) = \height(A)$. 
\end{enu}
\end{thm}

%===================%
% START SECTION 04 %
%===================%

%%%%%%%%%%%%
% sec:proofs %
%%%%%%%%%%%%

\section{Proofs of main results.}\label{sec:proofs}
We prove Theorems~\ref{thm:YB-move} and \ref{thm:wt_height_preserving} in this section. 
The proofs are based on a property analogous to the shellability of $\QBG(W)$ 
for the rank 2 root systems. 

%=========================%
% START SUBSECTION 0401 %
%=========================%

\subsection{Quantum Bruhat operators.}
Let $K$ be a field which contains the ring $\BC \bra{Q^{\vee, +}} := \BC \bra{Q_{i} \mid i \in I}$ of formal power series, 
where $Q_{i}$, $i \in I$, are variables, and set $Q^{\xi} := \prod_{i \in I} Q_{i}^{m_{i}}$ for $\xi = \sum_{i \in I}m_{i}\alpha_{i}^{\vee} \in Q^{\vee,+}$. 
For $\gamma \in \Delta^{+}$, following \cite{BFP}, we define the \emph{quantum Bruhat operator} $\sQ_{\gamma}$ on the group algebra $K[W]$ of $W$ by 
\begin{equation*}
\sQ_{\gamma} v := \begin{cases} vs_{\gamma} & \text{ if } v \xrightarrow{\gamma} v s_{\gamma} \text{ is a Bruhat edge,} \\ Q^{\gamma^{\vee}} v s_{\gamma} & \text{ if } v \xrightarrow{\gamma} v s_{\gamma} \text{ is a quantum edge,} \\ 0 & \text{ otherwise.} \end{cases} 
\end{equation*}
We set $\sQ_{-\gamma} := -\sQ_{\gamma}$ for $\gamma \in \Delta^{+}$, and then $\sR_{\gamma} := 1 + \sQ_{\gamma}$ for $\gamma \in \Delta$. 
The operators $\{ \sR_{\gamma} \mid \gamma \in \Delta \}$ satisfy the \emph{Yang-Baxter equation}: 
for $\alpha, \beta \in \Delta$ such that $\pair{\alpha}{\beta^{\vee}} \le 0$ and $\alpha \not= -\beta$, 
it holds that
%
%%%%%%%%%%%%%%%%%%%
% eq:Yang-Baxter_eq %
%%%%%%%%%%%%%%%%%%%
% 
\begin{equation}\label{eq:Yang-Baxter_eq}
\sR_{\alpha} \sR_{s_{\alpha}(\beta)} \sR_{s_{\alpha}s_{\beta}(\alpha)} \cdots \sR_{s_{\beta}(\alpha)} \sR_{\beta} = \sR_{\beta} \sR_{s_{\beta}(\alpha)} \cdots \sR_{s_{\alpha}s_{\beta}(\alpha)} \sR_{s_{\alpha}(\beta)} \sR_{\alpha}; 
\end{equation}
the proof of this equation is the same as that of \cite[Proposition~36]{LNS}. 

We give some properties of quantum Bruhat operators. 
%
%%%%%%%%%%%%%%%%%%%%%%%%%%%
% lem:prod_quantum_Bruhat_op %
%%%%%%%%%%%%%%%%%%%%%%%%%%%
%
\begin{lem}\label{lem:prod_quantum_Bruhat_op}
Let $\Pi = (\beta_{1}, \ldots, \beta_{r})$ be a sequence of roots such that $\beta_{1}, \ldots, \beta_{r}$ are distinct. 
\begin{enu}
\item For $v \in W$, we have 
\begin{equation*}
\sR_{\beta_{r}} \sR_{\beta_{r-1}} \cdots \sR_{\beta_{1}} v = \sum_{\bp \in \CP(v, \Pi)} (-1)^{\nega(\bp)} Q^{\wt(\bp)} \ed(\bp). 
\end{equation*}
\item For $v \in W$, we have 
\begin{equation*}
\sR_{|\beta_{r}|} \sR_{|\beta_{r-1}|} \cdots \sR_{|\beta_{1}|} v = \sum_{\bp \in \CP(v, \Pi)} Q^{\wt(\bp)} \ed(\bp). 
\end{equation*}
\end{enu}
\end{lem}
\begin{proof}
For $J \subset \{1, \ldots, r\}$, we set $\nega(J) := \{j \in J \mid \beta_{j} \in -\Delta^{+}\}$. 
We see that 
%
%%%%%%%%%%%%%%%%%%%%%%%%%%
% eq:prod_quantum_Bruhat_op %
%%%%%%%%%%%%%%%%%%%%%%%%%%
%
\begin{align}
\sR_{\beta_{r}} \sR_{\beta_{r-1}} \cdots \sR_{\beta_{1}} &= (1+\sQ_{\beta_{r}})(1+\sQ_{\beta_{r-1}}) \cdots (1+\sQ_{\beta_{1}}) \nonumber \\ 
&= \sum_{\{j_{1} < \ldots < j_{s}\} \subset \{1, \ldots, r\}} \sQ_{\beta_{j_{s}}}\sQ_{\beta_{j_{s-1}}} \cdots \sQ_{\beta_{j_{1}}} \nonumber \\ 
&= \sum_{\{j_{1} < \ldots < j_{s}\} \subset \{1, \ldots, r\}} (\sgn(\beta_{j_{s}})\sQ_{|\beta_{j_{s}}|})(\sgn(\beta_{j_{s-1}})\sQ_{|\beta_{j_{s-1}}|}) \cdots (\sgn(\beta_{j_{1}})\sQ_{|\beta_{j_{1}}|}) \nonumber \\ 
&= \sum_{J = \{j_{1} < \ldots < j_{s}\} \subset \{1, \ldots, r\}} (-1)^{\nega(J)} \sQ_{|\beta_{j_{s}}|}\sQ_{|\beta_{j_{s-1}}|} \cdots \sQ_{|\beta_{j_{1}}|}. \label{eq:prod_quantum_Bruhat_op}
\end{align}
Similarly, we see that 
%
%%%%%%%%%%%%%%%%%%%%%%%%%%%%%%
% eq:prod_quantum_Bruhat_op_abs %
%%%%%%%%%%%%%%%%%%%%%%%%%%%%%%
%
\begin{equation}
\sR_{|\beta_{r}|} \sR_{|\beta_{r-1}|} \cdots \sR_{|\beta_{1}|} = \sum_{\{j_{1} < \ldots < j_{s}\} \subset \{1, \ldots, r\}} \sQ_{|\beta_{j_{s}}|}\sQ_{|\beta_{j_{s-1}}|} \cdots \sQ_{|\beta_{j_{1}}|}.  \label{eq:prod_quantum_Bruhat_op_abs}
\end{equation}

For $J = \{j_{1}, \ldots, j_{s}\} \subset \{1, \ldots, r\}$, 
if we have the edge $vs_{|\beta_{j_{1}}|} \cdots s_{|\beta_{j_{a-1}}|} \xrightarrow{|\beta_{j_{a}}|} vs_{|\beta_{j_{1}}|} \cdots s_{|\beta_{j_{a}}|}$ in $\QBG(W)$ for all $1 \le a \le s$, 
then we set $\delta(J) := 1$, and define a directed path $\bp(J)$ in $\QBG(W)$ by 
\begin{equation*}
\bp(J): v \xrightarrow{|\beta_{j_{1}}|} vs_{|\beta_{j_{1}}|} \xrightarrow{|\beta_{j_{2}}|} \cdots \xrightarrow{|\beta_{j_{s}}|} vs_{|\beta_{j_{1}}|}\cdots s_{|\beta_{j_{s}}|}; 
\end{equation*}
otherwise, we set $\delta(J) := 0$. 
By the definition of quantum Bruhat operators, we have 
\begin{equation*}
\sQ_{|\beta_{j_{s}}|}\sQ_{|\beta_{j_{s-1}}|} \cdots \sQ_{|\beta_{j_{1}}|} v = \begin{cases} Q^{\wt(\bp(J))} \ed(\bp(J)) & \text{if } \delta(J) = 1, \\ 0 & \text{if } \delta(J) = 0. \end{cases}
\end{equation*}
If $\delta(J) = 1$, then we have $\nega(J) = \nega(\bp(J))$. 
Therefore, by \eqref{eq:prod_quantum_Bruhat_op}, we deduce that 
\begin{align*}
\sR_{\beta_{r}} \sR_{\beta_{r-1}} \cdots \sR_{\beta_{1}} v &= \sum_{J = \{j_{1} < \ldots < j_{s}\} \subset \{1, \ldots, r\}} (-1)^{\nega(J)} \sQ_{|\beta_{j_{s}}|}\sQ_{|\beta_{j_{s-1}}|} \cdots \sQ_{|\beta_{j_{1}}|} v \\ 
&= \sum_{\substack{J \subset \{1, \ldots, r\} \\ \delta(J) = 1}} (-1)^{\nega(\bp(J))} Q^{\wt(\bp(J))} \ed(\bp(J)) \\ 
&= \sum_{\bp \in \CP(v, \Pi)} (-1)^{\nega(\bp)} Q^{\wt(\bp)} \ed(\bp), 
\end{align*}
as desired. This proves part (1) of the lemma. 

Also, we see from \eqref{eq:prod_quantum_Bruhat_op_abs} that 
\begin{align*}
\sR_{|\beta_{r}|} \sR_{|\beta_{r-1}|} \cdots \sR_{|\beta_{1}|} v &= \sum_{\{j_{1} < \ldots < j_{s}\} \subset \{1, \ldots, r\}} \sQ_{|\beta_{j_{s}}|}\sQ_{|\beta_{j_{s-1}}|} \cdots \sQ_{|\beta_{j_{1}}|} v \\ 
&= \sum_{\substack{J \subset \{1, \ldots, r\} \\ \delta(J) = 1}} Q^{\wt(\bp(J))} \ed(\bp(J)) \\ 
&= \sum_{\bp \in \CP(v, \Pi)} Q^{\wt(\bp)} \ed(\bp). 
\end{align*}
This proves part (2) of the lemma. 
\end{proof}

%%%%%%%%%%%
% rem:SkTk %
%%%%%%%%%%%
%
\begin{rem}\label{rem:SkTk}
If we set $\CP(v, \Pi; w, \xi) := \{ \bp \in \CP(v, \Pi) \mid \ed(\bp) = w, \ \wt(\bp) = \xi \}$ for $v, w \in W$ and $\xi \in Q^{\vee, +}$, 
then by Lemma~\ref{lem:prod_quantum_Bruhat_op}\,(1), we deduce that 
%
%%%%%%%%
% eq:Tk %
%%%%%%%%
%
\begin{equation}\label{eq:Tk}
\sR_{\beta_{r}} \cdots \sR_{\beta_{1}} v = \sum_{w \in W} \sum_{\xi \in Q^{\vee, +}} \left( \sum_{\bp \in \CP(v, \Pi; w, \xi)} (-1)^{\nega(\bp)} \right) Q^{\xi} w. 
\end{equation}
Also, if we set $c_{\xi, w}^{v} := \# \CP(v, \Pi; w, \xi)$ for $v, w \in W$ and $\xi \in Q^{\vee, +}$, then we deduce from Lemma~\ref{lem:prod_quantum_Bruhat_op}\,(2) that 
%
%%%%%%%%
% eq:Sk %
%%%%%%%%
%
\begin{equation}\label{eq:Sk}
\sR_{|\beta_{r}|} \sR_{|\beta_{r-1}|} \cdots \sR_{|\beta_{1}|} v = \sum_{w \in W} \sum_{\xi \in Q^{\vee, +}} c_{\xi, w}^{v} Q^{\xi} w. 
\end{equation}
\end{rem}

%=========================%
% START SUBSECTION 0402 %
%=========================%

%%%%%%%%%%%%%%%%%%%%%%%%
% subsec:rank2_shellability %
%%%%%%%%%%%%%%%%%%%%%%%%

\subsection{Key propositions to a generalization of quantum Yang-Baxter moves.}\label{subsec:rank2_shellability}
We prove a certain property of $\QBG(W)$, which plays an important role in the proof of Theorem~\ref{thm:YB-move}. 
Let $\alpha, \beta \in \Delta$ be such that $\pair{\alpha}{\beta^{\vee}} \le 0$ and $\alpha \not= -\beta$. 
We define sequences of roots $\Pi$, $\Pi'$ by 
\begin{align*}
\Pi &= (\gamma_{1}, \ldots, \gamma_{q}) := (\alpha, s_{\alpha}(\beta), s_{\alpha}s_{\beta}(\alpha), \ldots, s_{\beta}(\alpha), \beta), \\ 
\Pi' &= (\gamma'_{1}, \ldots, \gamma'_{q}) := (\beta, s_{\beta}(\alpha), \ldots, s_{\alpha}s_{\beta}(\alpha), s_{\alpha}(\beta), \alpha) = (\gamma_{q}, \ldots, \gamma_{2}, \gamma_{1}); 
\end{align*}
note that $\gamma_{1}, \ldots, \gamma_{q}$ are distinct. 
Also, let $\Delta_{\alpha, \beta}$ be the root subsystem of $\Delta$ generated by $\alpha$ and $\beta$. 
Then $\Delta_{\alpha, \beta}$ is a root system of rank 2. 
More precisely, we see that $\Delta_{\alpha, \beta}$ is isomorphic to the root system of type $A_{1} \times A_{1}$, $A_{2}$, $C_{2}$, or $G_{2}$. 

Assume temporarily that $\Delta$ is not of type $G_{2}$. Then we can prove the following property, 
which can be thought of as a generalization of the shellability of $\QBG(W)$ (Theorem~\ref{thm:shellability}) for the rank 2 root systems. 
%
%%%%%%%%%%%%%%%%%%%%%%
% prop:rank2_shellability %
%%%%%%%%%%%%%%%%%%%%%%
%
\begin{prop}\label{prop:rank2_shellability}
Let $v \in W$, and let $\bp$ be a $\Pi$-compatible directed path in $\QBG(W)$ which starts at $v$, i.e., $\bp \in \CP(v, \Pi)$. 
Then only one of the following occurs. 
\begin{enu}
\item There exists a unique $\bp' \in \CP(v, \Pi) \setminus \{\bp\}$ such that $\ed(\bp') = \ed(\bp)$ and $\wt(\bp') = \wt(\bp)$. 
This $\bp'$ satisfies $(-1)^{\nega(\bp')} = -(-1)^{\nega(\bp)}$. 
Moreover, there does not exist a path $\bq \in \CP(v, \Pi')$ such that $\ed(\bq) = \ed(\bp)$ and $\wt(\bq) = \wt(\bp)$. 
\item There exists a unique $\bp' \in \CP(v, \Pi')$ such that $\ed(\bp') = \ed(\bp)$ and $\wt(\bp') = \wt(\bp)$. 
This $\bp'$ satisfies $(-1)^{\nega(\bp')} = (-1)^{\nega(\bp)}$. 
Moreover, there does not exist a path $\bq \in \CP(v, \Pi) \setminus \{\bp\}$ such that $\ed(\bq) = \ed(\bp)$ and $\wt(\bq) = \wt(\bp)$. 
\end{enu}
\end{prop}

The proof of this proposition can be reduced to the case that $\Delta$ is a root system of rank 2;
in Appendix~\ref{sec:App_B}, we explain how to construct the explicit correspondence $\bp \mapsto \bp'$ through an example. 
Now we assume that $\Delta$ is a root system of type $A_{1} \times A_{1}$, $A_{2}$, or $C_{2}$. 
Then we see that there exists some $k = 1, \ldots, q$ such that $\gamma_{k}$ and $\gamma_{k+1}$ are the simple roots of $\Delta$ (for convenience of notation, we set $\gamma_{q+1} := \gamma_{1}$). 
We set 
%
%%%%%%%%%%%%%%%%
% eq:chain_rank2 %
%%%%%%%%%%%%%%%%
%
\begin{equation}\label{eq:chain_rank2}
(\beta_{1}, \ldots, \beta_{q}) := (|\gamma_{k}|, |\gamma_{k-1}|, \ldots, |\gamma_{1}|, |\gamma_{q}|, \ldots, |\gamma_{k+1}|). 
\end{equation}
Then we have 
\begin{equation*}
(\beta_{1}, \ldots, \beta_{q}) = (\beta_{1}, s_{\beta_{1}}(\beta_{q}), s_{\beta_{1}}s_{\beta_{q}}(\beta_{1}), \ldots, s_{\beta_{q}}(\beta_{1}), \beta_{q}). 
\end{equation*}
Also, if we set 
\begin{equation*}
\Pi^{\pm} := (\mp \beta_{k}, \mp \beta_{k-1}, \ldots, \mp \beta_{1}, \pm \beta_{q}, \pm \beta_{q-1}, \ldots, \pm \beta_{k+1}), 
\end{equation*}
then $\Pi = \Pi^{+}$ or $\Pi = \Pi^{-}$.  
Note that the total order $\prec$ on $\{\beta_{1}, \ldots, \beta_{q}\} = \Delta^{+}$ defined by 
%
%%%%%%%%%%%%%
% eq:def_prec %
%%%%%%%%%%%%%
%
\begin{equation}\label{eq:def_prec}
\beta_{1} \prec \beta_{2} \prec \cdots \prec \beta_{q} 
\end{equation}
is a reflection order; 
the total order $\prec'$ defined by 
%
%%%%%%%%%%%%%
% eq:def_prec' %
%%%%%%%%%%%%%
%
\begin{equation}\label{eq:def_prec'}
\beta_{q} \prec' \beta_{q-1} \prec' \cdots \prec' \beta_{1}
\end{equation}
is also a reflection order. 
We consider the following operators for $k = 0, 1, \ldots, q$: 
\begin{align*}
\sT_{k}^{\pm} &:= \sR_{\pm \beta_{k+1}} \cdots \sR_{\pm \beta_{q}} \sR_{\mp \beta_{1}} \cdots \sR_{\mp \beta_{k}}, \\ 
\sS_{k} &:= \sR_{\beta_{k+1}} \cdots \sR_{\beta_{q}} \sR_{\beta_{1}} \cdots \sR_{\beta_{k}}, \\ 
\sS'_{k} &:= \sR_{\beta_{k}} \cdots \sR_{\beta_{1}} \sR_{\beta_{q}} \cdots \sR_{\beta_{k+1}}.
\end{align*}

In the following proposition, the matrices of operators on $K[W]$ are the representation matrices with respect to the basis $W$ of $K[W]$. 
Note that for a $W$-linear operator $\sT: K[W] \rightarrow K[W]$, the matrix of $\sT$ is defined by $(c_{v, w})_{v, w \in W}$ 
if $\sT w = \sum_{v \in W} c_{v, w} v$, $c_{v, w} \in K$. 
%
%%%%%%%%%%%%%
% prop:matrix %
%%%%%%%%%%%%%
%
\begin{prop}\label{prop:matrix}
\begin{enu}
\item All the entries of the matrix of $\sS_{k}$, $k = 0, 1, \ldots, q$, are of the form $\sum_{j = 1}^{r} m_{j}Q^{\xi_{j}}$, where all $\xi_{j} \in Q^{\vee,+}$ are distinct, and $m_{j} \in \{1, 2\}$. 
\item Let $v, w \in W$. Assume that the $(v, w)$-entry of the matrix of $\sS_{k}$ is of the form $\sum_{j = 1}^{r} m_{j}Q^{\xi_{j}}$ as in (1). 
Also, assume that the $(v, w)$-entry of the matrix of $\sT_{k}^{\pm}$ is of the form $\sum_{\xi \in Q^{\vee, +}} n_{\xi}^{\pm} Q^{\xi}$. 
For $j = 1, \ldots, r$, if $m_{j} = 2$, then $n_{\xi_{j}}^{\pm} = 0$, and if $m_{j} = 1$, then $n_{\xi_{j}}^{\pm} \in \{1, -1\}$. 
Moreover, for $\xi \in Q^{\vee, +} \setminus \{\xi_{1}, \ldots, \xi_{r}\}$, we have $n_{\xi}^{\pm} = 0$. 
\item Let $v, w \in W$. Assume that the $(v, w)$-entry of the matrix of $\sS_{k}$ is of the form $\sum_{j = 1}^{r} m_{j}Q^{\xi_{j}}$ as in (1). 
Also, assume that the $(v, w)$-entry of the matrix of $\sS'_{k}$ is of the form $\sum_{\xi \in Q^{\vee, +}} n_{\xi}Q^{\xi}$. 
For $j = 1, \ldots, r$, if $m_{j} = 2$, then $n_{\xi_{j}} = 0$, and if $m_{j} = 1$, then $n_{\xi_{j}} = 1$. 
\end{enu}
\end{prop}
The proof of Proposition~\ref{prop:matrix} is based on direct calculations, which we give later. 

\begin{proof}[Proof of Proposition~\ref{prop:rank2_shellability}]
First, we show the proposition for the root system $\Delta$ of type $A_{1} \times A_{1}$, $A_{2}$, or $C_{2}$. 
As in \eqref{eq:chain_rank2}, we take a sequence 
\begin{equation*}
(\beta_{1}, \ldots, \beta_{q}) := (|\gamma_{k}|, |\gamma_{k-1}|, \ldots, |\gamma_{1}|, |\gamma_{q}|, \ldots, |\gamma_{k+1}|)
\end{equation*}
of roots. 
Recall from \eqref{eq:Sk} that 
\begin{equation*}
\sS_{k} v = \sum_{w \in W} \sum_{\xi \in Q^{\vee, +}} c_{w, \xi}^{v} Q^{\xi} w, 
\end{equation*}
where $c_{w, \xi}^{v} = \# \CP(v, \Pi; w, \xi)$. 
By Proposition~\ref{prop:matrix}\,(1), we have $c_{w, \xi}^{v} \in \{0, 1, 2\}$. 
Also, again from \eqref{eq:Sk}, we see that 
\begin{equation*}
\sS'_{k} v = \sum_{w \in W} \sum_{\xi \in Q^{\vee, +}} (c_{w, \xi}^{v})' Q^{\xi} w, 
\end{equation*}
where $(c_{w, \xi}^{v})' = \# \CP(v, \Pi'; w, \xi)$. 

We write 
\begin{equation*}
\sT_{k}^{\pm} v = \sum_{w \in W} \sum_{\xi \in Q^{\vee, +}} d_{w, \xi}^{v, \pm} Q^{\xi} w, 
\end{equation*}
where $d_{w, \xi}^{v, \pm} \in \BZ$. 
Then, by \eqref{eq:Tk}, if $\Pi = \Pi^{+}$, then 
%
%%%%%%%%
% eq:d+ %
%%%%%%%%
%
\begin{equation}\label{eq:d+}
d_{w, \xi}^{v, +} = \sum_{\bq \in \CP(v, \Pi; w, \xi)} (-1)^{\nega(\bq)}, \quad d_{w, \xi}^{v, -} = \sum_{\bq \in \CP(v, \Pi; w, \xi)} (-1)^{\ell(\bq)-\nega(\bq)}; 
\end{equation}
if $\Pi = \Pi^{-}$, then 
%
%%%%%%%%
% eq:d- %
%%%%%%%%
%
\begin{equation}\label{eq:d-}
d_{w, \xi}^{v, +} = \sum_{\bq \in \CP(v, \Pi; w, \xi)} (-1)^{\ell(\bq)-\nega(\bq)}, \quad d_{w, \xi}^{v, -} = \sum_{\bq \in \CP(v, \Pi; w, \xi)} (-1)^{\nega(\bq)}. 
\end{equation}

We set $w := \ed(\bp)$ and $\xi := \wt(\bp)$. Since $\bp \in \CP(v, \Pi; w, \xi)$, we have $c_{w, \xi}^{v} \not= 0$. 
First, assume that $c_{w, \xi}^{v} = 2$. 
Then there exists a unique $\bp' \in \CP(v, \Pi; w, \xi) \setminus \{\bp\}$, i.e., 
there exists a unique $\bp' \in \CP(v, \Pi) \setminus \{\bp\}$ such that $\ed(\bp') = \ed(\bp) = w$ and $\wt(\bp') = \wt(\bp) = \xi$. 
By Proposition~\ref{prop:matrix}\,(2), we have $d_{w, \xi}^{v, \pm} = 0$. 
Hence, by \eqref{eq:d+} and \eqref{eq:d-}, we obtain 
\begin{equation*}
\sum_{\bq \in \CP(v, \Pi; w, \xi)} (-1)^{\nega(\bq)} = (-1)^{\nega(\bp)} + (-1)^{\nega(\bp')} = 0. 
\end{equation*}
This shows that $(-1)^{\nega(\bp')} = -(-1)^{\nega(\bp)}$. 
Here, by Proposition~\ref{prop:matrix}\,(3), we deduce that $(c_{w, \xi}^{v})' = 0$. 
Hence there does not exist a $\bq \in \CP(v, \Pi')$ such that $\ed(\bq) = \ed(\bp) = w$ and $\wt(\bq) = \wt(\bp) = \xi$. 
This shows the proposition in the case $c_{w, \xi}^{v} = 2$. 

Next, we assume that $c_{w, \xi}^{v} = 1$. 
In this case, there does not exist a $\bq \in \CP(v, \Pi) \setminus \{\bp\}$ such that $\ed(\bq) = \ed(\bp) = w$ and $\wt(\bq) = \wt(\bp) = \xi$, 
since $\CP(v, \Pi; w, \xi) = \{\bp\}$. 
We set 
\begin{equation*}
(\sT_{k}^{\pm})' := \sR_{\mp \beta_{k}} \cdots \sR_{\mp \beta_{1}} \sR_{\pm \beta_{q}} \cdots \sR_{\pm \beta_{k+1}}. 
\end{equation*}
Then, by the Yang-Baxter equation \eqref{eq:Yang-Baxter_eq}, we have $(\sT_{k}^{\pm})' = \sT_{k}^{\pm}$. 
Hence, if we write 
\begin{equation*}
(\sT_{k}^{\pm})' v = \sum_{w \in W} \sum_{\xi \in Q^{\vee, +}} (d_{w, \xi}^{v, \pm})' Q^{\xi} w, 
\end{equation*}
with $(d_{w, \xi}^{v, \pm})' \in \BZ$, then we see that $(d_{w, \xi}^{v, \pm})' = d_{w, \xi}^{v, \pm}$. 
Here, by Proposition~\ref{prop:matrix}\,(2), we deduce that $(d_{w, \xi}^{v, \pm})' = d_{w, \xi}^{v, \pm} \in \{1, -1\}$. 
Again, by Proposition~\ref{prop:matrix}\,(2) (by replacing ($\beta_{1}, \ldots, \beta_{q})$ with $(\beta_{q}, \ldots, \beta_{1})$), 
we deduce that $(c_{w, \xi}^{v})' = 1$. 
Hence there exists a unique $\bp' \in \CP(v, \Pi'; w, \xi)$, 
i.e., there exists a unique $\bp' \in \CP(v, \Pi')$ such that $\ed(\bp') = \ed(\bp) = w$ and $\wt(\bp') = \wt(\bp) = \xi$. 
If $\Pi = \Pi^{+}$, then 
\begin{equation*}
(-1)^{\nega(\bp)} = d_{w, \xi}^{v, +} = (d_{w, \xi}^{v, +})' = (-1)^{\nega(\bp')}; 
\end{equation*}
if $\Pi = \Pi^{-}$, then 
\begin{equation*}
(-1)^{\nega(\bp)} = d_{w, \xi}^{v, -} = (d_{w, \xi}^{v, -})' = (-1)^{\nega(\bp')}. 
\end{equation*}
This shows that $(-1)^{\nega(\bp')} = (-1)^{\nega(\bp)}$, as desired. 
This completes the proof of the proposition for the root system $\Delta$ of type $A_{1} \times A_{1}$, $A_{2}$, or $C_{2}$. 

Now, assume that the root system $\Delta$ is of an arbitrary type (except $G_{2}$), not necessarily of rank 2. 
Let $\overline{W}$ be the Weyl group of $\Delta_{\alpha, \beta}$. 
Note that $\overline{W}$ is a (dihedral) subgroup of $W$; 
the quantum Bruhat graph (denoted by $\QBG(\overline{W})$) of $\overline{W}$ is no longer a subgraph of $\QBG(W)$. 
By \cite[Proposition~5.1 and Remarks~5.2\,(2)]{LL2}, for each $u \in W$, 
there exist uniquely $\lfloor u \rfloor \in u \overline{W}$ and $\overline{u} \in \overline{W}$ 
such that 
\begin{itemize}
\item $u = \overline{u} \lfloor u \rfloor$, and 
\item for a positive root $\gamma$ of $\Delta_{\alpha, \beta}$, we have $\ell(\overline{u}) < \ell(\overline{u}s_{\gamma})$ 
if and only if $\ell(\lfloor u \rfloor \overline{u}) < \ell(\lfloor u \rfloor \overline{u} s_{\gamma})$. 
\end{itemize}

We set $w := \ed(\bp)$ and $\xi := \wt(\bp)$. 
Suppose, for a contradiction, that there exist two or more directed paths $\bq \in \CP(v, \Pi; w, \xi) \setminus \{ \bp \}$. 
Then we see that $\# \CP(v, \Pi, w, \xi) \ge 3$. 
By \cite[Theorem~5.3]{LL2}, there exists an injection $\CP(v, \Pi; w, \xi) \hookrightarrow \CP(\overline{v}, \Pi; \overline{w}, \gamma)$, 
where $\CP(\overline{v}, \Pi; \overline{w}, \gamma)$ is the set of all $\Pi$-compatible directed paths in $\QBG(\overline{W})$ 
which starts at $\overline{v}$, ends at $\overline{w}$, and has weight $\xi$, 
where $\Pi$ is considered to be a sequence of roots in the root system $\Delta_{\alpha, \beta}$. 
Hence we have $\# \CP(\overline{v}, \Pi; \overline{w}, \xi) \ge 3$. 
This contradicts the proposition for the rank 2 root systems, shown above. 
Hence we conclude that there exists at most one directed path $\bq \in \CP(v, \Pi; w, \xi)$. 
If such a $\bq$ exists, then, by the proposition for the rank 2 root systems and \cite[Theorem~5.3]{LL2}, we have  $(-1)^{\nega(\bq)} = -(-1)^{\nega(\bp)}$. 
Also, a similar argument shows that there exists at most one directed path $\br \in \CP(v, \Pi'; w, \xi)$. 
If such an $\br$ exists, then we have $(-1)^{\nega(\br)} = (-1)^{\nega(\bp)}$. 

We show that at least one of the directed paths $\bq$ and $\br$ exists. 
We write 
\begin{align*}
\sR_{\gamma_{q}} \cdots \sR_{\gamma_{1}} v &= \sum_{w \in W} \sum_{\xi \in Q^{\vee, +}} d_{w, \xi}^{v} Q^{\xi} w, \\ 
\sR_{\gamma_{1}} \cdots \sR_{\gamma_{q}} v &= \sum_{w \in W} \sum_{\xi \in Q^{\vee, +}} (d_{w, \xi}^{v})' Q^{\xi} w. 
\end{align*}
If there does not exist a directed path $\bq \in \CP(v, \Pi; w, \xi) \setminus \{\bp\}$, then by \eqref{eq:Tk}, we have $d_{w, \xi}^{v} = \pm 1$. 
By the Yang-Baxter equation \eqref{eq:Yang-Baxter_eq}, we deduce that $(d_{w, \xi}^{v})' = d_{w, \xi}^{v} = \pm 1$. 
By \eqref{eq:Tk}, we see that $\CP(v, \Pi'; w, \xi) \not= \emptyset$. 
Therefore, we conclude that there exists a directed path $\br \in \CP(v, \Pi'; w, \xi)$ in this case, as desired.

Finally, suppose, for a contradiction, that both $\bq$ and $\br$ exist at the same time. 
Then, by \cite[Theorem~5.3]{LL2}, we have $\# \CP(\overline{v}, \Pi; \overline{w}, \xi) \ge 2$ and $\# \CP(\overline{v}, \Pi'; \overline{w}, \xi) \ge 1$. 
This contradicts the proposition for the rank 2 root systems, shown above. 

This completes the proof of Proposition~\ref{prop:rank2_shellability}.
\end{proof}

Thus it remains to prove Proposition~\ref{prop:matrix}. 
We assume temporarily that $\Delta$ is of type $A_{1} \times A_{1}$, $A_{2}$, or $C_{2}$. 
If $\Delta$ is of type $A_{2}$ (resp., $A_{1} \times A_{1}$, $C_{2}$), then we have $q = 3$ (resp., $q=2, 4$). 
By the shellability of $\QBG(W)$, there exists a unique label-increasing directed path 
(with respect to $\prec$ or $\prec'$, defined in \eqref{eq:def_prec} and \eqref{eq:def_prec'}) from $v$ to $w$ in $\QBG(W)$ for all $v, w \in W$. 
Hence we have 
\begin{align*}
\sT_{0}^{+} v &= \sT_{q}^{-} v = \sS_{0} v = \sS'_{0} v = \sum_{w \in W} Q^{\wt(v \Rightarrow w)} w, \\ 
\sT_{0}^{-} v &= \sT_{q}^{+} v = \sum_{w \in W} (-1)^{\ell(v \Rightarrow w)} Q^{\wt(v \Rightarrow w)} w 
\end{align*}
for all $v \in W$. Therefore, the proposition is obvious in the case $k = 0, q$. 
Hence it suffices to prove the proposition in the case $k = 1, q-1$ for all types, and in the case $k = 2$ for type $C_{2}$. 

%=========================%
% START SUBSECTION 0403 %
%=========================%

\subsection{Proof of Proposition~\ref{prop:matrix}: \texorpdfstring{$k = 1, q-1$}{k=1, q-1}.}
We prove Proposition~\ref{prop:matrix} in the case $k = 1, q-1$; 
recall that $\beta_{1}$ and $\beta_{q}$ are the simple roots of $\Delta$. 
By \eqref{eq:Tk}, for all $v \in W$, we have 
\begin{align*}
\sT_{1}^{+} v &= \sR_{\beta_{2}} \cdots \sR_{\beta_{q}} (\sR_{-\beta_{1}} v) \\
&= \sR_{\beta_{2}} \cdots \sR_{\beta_{q}} ((1 - \sQ_{\beta_{1}}) v) \\
&= \sR_{\beta_{2}} \cdots \sR_{\beta_{q}} (v - Q^{\wt(v \rightarrow vs_{\beta_{1}})} vs_{\beta_{1}}) \\ 
&= \sum_{\bq \in \CP(v, (\beta_{q}, \ldots, \beta_{2}))} Q^{\wt(\bq)} \ed(\bq) - \sum_{\bq \in \CP(vs_{\beta_{1}}, (\beta_{q}, \ldots, \beta_{2}))} Q^{\wt(v \rightarrow vs_{\beta_{1}}) + \wt(\bq)} \ed(\bq). 
\end{align*}
Recall that the total order $\prec'$ on $\Delta^{+} = \{\beta_{1}, \ldots, \beta_{q}\}$, defined by \eqref{eq:def_prec'}, is a reflection order. 
Hence, by the shellability of $\QBG(W)$, for all $w \in W$, 
there exists at most one directed path $\bq \in \CP(v, (\beta_{q}, \ldots, \beta_{2}))$ such that $\ed(\bq) = w$. 
For such $\bq$, we have $\wt(\bq) = \wt(v \Rightarrow w)$ since $\bq$ is a shortest directed path from $v$ to $w$. 
The same argument shows that for all $w \in W$, there exists at most one directed path $\bq \in \CP(vs_{\beta_{1}}, (\beta_{q}, \ldots, \beta_{2}))$ such that $\ed(\bq) = w$ 
and $\wt(\bq) = \wt(vs_{\beta_{1}} \Rightarrow w)$. 
Hence, if we set 
\begin{equation*}
\delta_{v,w} := \begin{cases} 1 & \text{if there exists } \bq \in \CP(v, (\beta_{q}, \ldots, \beta_{2})) \text{ such that } \ed(\bq) = w, \\ 0 & \text{otherwise} \end{cases}
\end{equation*}
for $v, w \in W$, then we have 
%
%%%%%%%%%%%
% eq:k=1_1 %
%%%%%%%%%%%
%
\begin{align}
\sT_{1}^{+} v &= \sum_{w \in W} \delta_{v,w} Q^{\wt(v \Rightarrow w)} w - \sum_{w \in W} \delta_{vs_{\beta_{1}}, w} Q^{\wt(v \rightarrow vs_{\beta_{1}}) + \wt(vs_{\beta_{1}} \Rightarrow w)} w \nonumber \\ 
&= \sum_{w \in W} (\delta_{v,w} Q^{\wt(v \Rightarrow w)} - \delta_{vs_{\beta_{1}}, w} Q^{\wt(v \rightarrow vs_{\beta_{1}}) + \wt(vs_{\beta_{1}} \Rightarrow w)}) w. \label{eq:k=1_1}
\end{align}
Also, by the same argument, we see that 
%
%%%%%%%%%%%
% eq:k=1_2 %
%%%%%%%%%%%
%
\begin{equation}
\sT_{1}^{-} v = \sum_{w \in W} (-1)^{\ell(v \Rightarrow w)} (\delta_{v,w} Q^{\wt(v \Rightarrow w)} - \delta_{vs_{\beta_{1}}, w} Q^{\wt(v \rightarrow vs_{\beta_{1}}) + \wt(vs_{\beta_{1}} \Rightarrow w)}) w; \label{eq:k=1_2}
\end{equation}
note that for a directed path $\bq$ from $vs_{\beta_{1}}$ to $w$, it follows that $(-1)^{\ell(\bq)} = (-1)^{\ell(vs_{\beta_{1}} \Rightarrow w)}$, 
and hence $(-1)^{\ell(vs_{\beta_{1}} \Rightarrow w)} = (-1)^{\ell(vs_{\beta_{1}} \Rightarrow v) + \ell(v \Rightarrow w)} = (-1)^{1 + \ell(v \Rightarrow w)} = -(-1)^{\ell(v \Rightarrow w)}$. 
Let us consider $\sS_{1}$. By the same argument as for $\sT_{1}^{+}$, we deduce that 
%
%%%%%%%%%%%
% eq:k=1_3 %
%%%%%%%%%%%
%
\begin{align}
\sS_{1} v &= \sum_{\bq \in \CP(v, (\beta_{q}, \ldots, \beta_{2}))} Q^{\wt(\bq)} \ed(\bq) + \sum_{\bq \in \CP(vs_{\beta_{1}}, (\beta_{q}, \ldots, \beta_{2}))} Q^{\wt(v \rightarrow vs_{\beta_{1}}) + \wt(\bq)} \ed(\bq) \nonumber \\ 
&= \sum_{w \in W} (\delta_{v,w} Q^{\wt(v \Rightarrow w)} + \delta_{vs_{\beta_{1}}, w} Q^{\wt(v \rightarrow vs_{\beta_{1}}) + \wt(vs_{\beta_{1}} \Rightarrow w)}) w. \label{eq:k=1_3}
\end{align}
Hence equations \eqref{eq:k=1_1}, \eqref{eq:k=1_2}, and \eqref{eq:k=1_3} imply Proposition~\ref{prop:matrix}\,(1)\,(2) in the case $k = 1$, as desired. 

Next, we consider the case $k = q-1$; 
recall that $\beta_{q}$ is a simple root of $\Delta$. 
By \eqref{eq:Tk}, we have 
\begin{align*}
\sT_{q-1}^{+} v &= \sR_{\beta_{q}} (\sR_{-\beta_{1}} \cdots \sR_{-\beta_{q-1}} v) \\ 
&= \sR_{\beta_{q}} \left( \sum_{\bq \in \CP(v, (-\beta_{q-1}, \ldots, -\beta_{1}))} (-1)^{\ell(\bq)} Q^{\wt(\bq)} \ed(\bq) \right) \\ 
&= \sum_{\bq \in \CP(v, (-\beta_{q-1}, \ldots, -\beta_{1}))} (-1)^{\ell(\bq)} (Q^{\wt(\bq)} \ed(\bq) + Q^{\wt(\bq) + \wt(\ed(\bq) \rightarrow \ed(\bq) s_{\beta_{q}})} \ed(\bq)s_{\beta_{q}}).
\end{align*}
Hence, if we set 
\begin{equation*}
\delta'_{v,w} := \begin{cases} 1 & \text{if there exists } \bq \in \CP(v, (-\beta_{q-1}, \ldots, -\beta_{1})) \text{ such that } \ed(\bq) = w, \\ 0 & \text{otherwise} \end{cases}
\end{equation*}
for $v, w \in W$, then we have
%
%%%%%%%%%%%%%
% eq:k=q-1_1 %
%%%%%%%%%%%%%
% 
\begin{equation}\label{eq:k=q-1_1}
\sT_{q-1}^{+} v = \sum_{w \in W} (-1)^{\ell(v \Rightarrow w)} (\delta'_{v,w} Q^{\wt(v \Rightarrow w)} - \delta'_{v,ws_{\beta_{q}}} Q^{\wt(v \Rightarrow ws_{\beta_{q}}) + \wt(ws_{\beta_{q}} \rightarrow w)}) w. 
\end{equation}
Similarly, we have 
%
%%%%%%%%%%%%%
% eq:k=q-1_2 %
%%%%%%%%%%%%%
%
\begin{equation}\label{eq:k=q-1_2}
\sT_{q-1}^{-} v = \sum_{w \in W} (\delta'_{v,w} Q^{\wt(v \Rightarrow w)} - \delta'_{v,ws_{\beta_{q}}} Q^{\wt(v \Rightarrow ws_{\beta_{q}}) + \wt(ws_{\beta_{q}} \rightarrow w)}) w. 
\end{equation}
Also, we see that 
%
%%%%%%%%%%%%%
% eq:k=q-1_3 %
%%%%%%%%%%%%%
%
\begin{equation}\label{eq:k=q-1_3}
\sS_{q-1} v = \sum_{w \in W} (\delta'_{v,w} Q^{\wt(v \Rightarrow w)} + \delta'_{v,ws_{\beta_{q}}} Q^{\wt(v \Rightarrow ws_{\beta_{q}}) + \wt(ws_{\beta_{q}} \rightarrow w)}) w. 
\end{equation}
Hence equations \eqref{eq:k=q-1_1}, \eqref{eq:k=q-1_2}, and \eqref{eq:k=q-1_3} imply Proposition~\ref{prop:matrix}\,(1)\,(2) in the case $k = q-1$. 

It remains to prove Proposition~\ref{prop:matrix}\,(3) in the case $k = 1, q-1$. 
It suffices to prove it in the case $k = 1$; indeed, if we replace $(\beta_{1}, \ldots, \beta_{q})$ with $(\beta_{q}, \ldots, \beta_{1})$ and consider the case $k = 1$, 
then we obtain the proposition in the case $k = q-1$. 
Recall equation \eqref{eq:k=1_3}. 
By the same argument, we see that 
\begin{equation}
\sS'_{1} v = \sum_{w \in W} (\ve_{v, w} Q^{\wt(v \Rightarrow w)} + \ve_{v, ws_{\beta_{1}}} Q^{\wt(v \Rightarrow ws_{\beta_{1}}) + \wt(ws_{\beta_{1}} \rightarrow w)})w, 
\end{equation}
where 
\begin{equation*}
\ve_{v,w} := \begin{cases} 1 & \text{if there exists } \bq \in \CP(v, (\beta_{2}, \ldots, \beta_{q})) \text{ such that } \ed(\bq) = w, \\ 0 & \text{otherwise} \end{cases}
\end{equation*}
for $v, w \in W$. 
Assume that $c_{w, \xi}^{v} = 2$ for some $v, w \in W$ and $\xi \in Q^{\vee, +}$. 
It suffices to show that $(c_{w,\xi}^{v})' = 0$. 
In this case, we deduce from \eqref{eq:k=1_3} that
%
%%%%%%%%%%
% eq:delta %
%%%%%%%%%%
%
%%%%%%%%%%%%%%% 
% eq:wt_equality %
%%%%%%%%%%%%%%%
%
\begin{align}
& \delta_{v,w} = \delta_{vs_{\beta_{1}}, w} = 1, \label{eq:delta} \\ 
& \wt(v \Rightarrow w) = \wt(v \rightarrow vs_{\beta_{1}}) + \wt(vs_{\beta_{1}} \Rightarrow w) = \xi. \label{eq:wt_equality}
\end{align}
By \eqref{eq:wt_equality}, we see that the concatenation of the edge $v \rightarrow vs_{\beta_{1}}$ with any shortest directed path from $vs_{\beta_{1}}$ to $w$ in $\QBG(W)$ 
is a shortest directed path from $v$ to $w$ (cf. \cite[Lemma~6.7]{BFP}, \cite[Lemma~1\,(2)]{P}, and \cite[Proposition~8.1]{LNSSS1}). 
Now, take the (unique) label-increasing directed path $\br_{0}$ from $vs_{\beta_{1}}$ to $w$ in $\QBG(W)$ with respect to $\prec$ defined by \eqref{eq:def_prec}, 
and let $\br$ be the concatenation of the edge $v \rightarrow vs_{\beta_{1}}$ with the path $\br_{0}$. 
Note that $\br_{0}$ is shortest, and hence $\br$ is also shortest. 
We claim that $\br_{0} \in \CP(v, (\beta_{2}, \ldots, \beta_{q}))$; otherwise, the concatenation 
\begin{equation*}
\br: v \xrightarrow{\beta_{1}} \underbrace{vs_{\beta_{1}} \xrightarrow{\beta_{1}} \cdots \rightarrow w}_{\br_{0}}
\end{equation*}
cannot be shortest. 
Hence $\br$ is the label-increasing directed path from $v$ to $w$ in $\QBG(W)$ such that $\br \notin \CP(v, (\beta_{2}, \ldots, \beta_{q}))$. 
By the uniqueness of a label-increasing directed path, we conclude that $\ve_{v,w} = 0$. 
Since $\delta_{v,w} = 1$ by \eqref{eq:delta}, there exists $\br_{1} \in \CP(v, (\beta_{q}, \ldots, \beta_{2}))$ such that $\ed(\br_{1}) = w$. 
Then the concatenation of the path $\br_{1}$ with the edge $w \rightarrow ws_{\beta_{1}}$ is label-increasing with respect to $\prec'$, defined by \eqref{eq:def_prec'}, 
and hence this concatenation is shortest. 
Also, since $\delta_{vs_{\beta_{1}}, w} = 1$ by \eqref{eq:delta}, there exists $\br_{2} \in \CP(vs_{\beta_{1}}, (\beta_{q}, \ldots, \beta_{2}))$ such that $\ed(\br_{2}) = w$. 
Similarly, the concatenation of the path $\br_{2}$ with the edge $w \rightarrow ws_{\beta_{1}}$ is label-increasing with respect to $\prec'$, 
and hence this concatenation is shortest. 
Since the concatenation of the edge $v \rightarrow vs_{\beta_{1}}$ with any shortest directed path from $vs_{\beta_{1}}$ to $w$ is shortest, 
we obtain: 
\begin{align*}
\ell(v \Rightarrow ws_{\beta_{1}}) &= \underbrace{\ell(v \Rightarrow w)}_{= \ell(\br_{1})} + \ell(w \rightarrow ws_{\beta_{1}}) = \ell(v \rightarrow vs_{\beta_{1}}) + \underbrace{\ell(vs_{\beta_{1}} \Rightarrow w)}_{= \ell(\br_{2})} + \ell(w \rightarrow ws_{\beta_{1}}) \\ 
&= \ell(v \rightarrow vs_{\beta_{1}}) + \ell(vs_{\beta_{1}} \Rightarrow ws_{\beta_{1}}). 
\end{align*}
Hence the concatenation of the edge $v \rightarrow vs_{\beta_{1}}$ with any shortest directed path from $vs_{\beta_{1}}$ to $ws_{\beta_{1}}$ is shortest. 
Take the (unique) label-increasing directed path $\br_{3}$ from $vs_{\beta_{1}}$ to $ws_{\beta_{1}}$ in $\QBG(W)$ with respect to $\prec$. 
Then we deduce that $\br_{3} \in \CP(vs_{\beta_{1}}, (\beta_{2}, \ldots, \beta_{q}))$; otherwise, the concatenation 
\begin{equation*}
v \xrightarrow{\beta_{1}} \underbrace{vs_{\beta_{1}} \xrightarrow{\beta_{1}} \cdots \rightarrow ws_{\beta_{1}}}_{\br_{3}}
\end{equation*}
cannot be shortest. Hence we conclude that $\ve_{v, ws_{\beta_{1}}} = 0$. 
This completes the proof that $(c_{w, \xi}^{v})' = 0$. 

It remains to show that if $c_{w, \xi}^{v} = 1$, then $(c_{w, \xi}^{v})' = 1$. 
Assume that $c_{w, \xi}^{v} = 1$. 
By the above argument (i.e., Proposition~\ref{prop:matrix}\,(2) in the case $k=1$), 
we have $d_{w, \xi}^{v, +} = \pm 1$. 
By the Yang-Baxter equation \eqref{eq:Yang-Baxter_eq}, we see that $(d_{w, \xi}^{v, +})' = d_{w, \xi}^{v, +} = \pm 1$. 
Hence we deduce again from the above argument 
(i.e., Proposition~\ref{prop:matrix}\,(2) in the case $k = 1$, with $(\beta_{1}, \ldots, \beta_{q})$ replaced by $(\beta_{q}, \ldots, \beta_{1})$) 
that $(c_{w, \xi}^{v})' = 1$. 

This completes the proof of Proposition~\ref{prop:matrix} in the case $k = 1, q-1$. 

%=========================%
% START SUBSECTION 0404 %
%=========================%

%%%%%%%%%%%%%%%%%%%%%%
% subsec:matrix_typeC2 %
%%%%%%%%%%%%%%%%%%%%%%

\subsection{Proof of Proposition~\ref{prop:matrix}: case of type \texorpdfstring{$C_{2}$}{C2}.}\label{subsec:matrix_typeC2}

We consider the root system $\Delta$ of type $C_{2}$. 
We know that $q = 4$, and $(\beta_{1}, \beta_{2}, \beta_{3}, \beta_{4}) = (\alpha_{1}, 2\alpha_{1}+\alpha_{2}, \alpha_{1}+\alpha_{2}, \alpha_{2})$ 
or $(\beta_{1}, \beta_{2}, \beta_{3}, \beta_{4}) = (\alpha_{2}, \alpha_{1}+\alpha_{2}, 2\alpha_{1}+\alpha_{2}, \alpha_{1})$. 

Since only the case $k = 2$ is remaining, it suffices to calculate the matrices (with respect to the basis $W = \{e, s_{1}, s_{2}, s_{1}s_{2}, s_{2}s_{1}, s_{1}s_{2}s_{1}, s_{2}s_{1}s_{2}, w_{\circ}\}$ of $K[W]$) of the following four operators: 
\begin{enu}
\item $\sR_{\alpha_{1}+\alpha_{2}} \sR_{\alpha_{2}} \sR_{-\alpha_{1}} \sR_{-2\alpha_{1}-\alpha_{2}} = \sR_{-2\alpha_{1}-\alpha_{2}} \sR_{-\alpha_{1}} \sR_{\alpha_{2}} \sR_{\alpha_{1}+\alpha_{2}}$; 
\item $\sR_{2\alpha_{1}+\alpha_{2}} \sR_{\alpha_{1}} \sR_{-\alpha_{2}} \sR_{-\alpha_{1}-\alpha_{2}} = \sR_{-\alpha_{1}-\alpha_{2}} \sR_{-\alpha_{2}} \sR_{\alpha_{1}} \sR_{2\alpha_{1}+\alpha_{2}}$; 
\item $\sR_{\alpha_{1}+\alpha_{2}} \sR_{\alpha_{2}} \sR_{\alpha_{1}} \sR_{2\alpha_{1}+\alpha_{2}}$; and 
\item $\sR_{2\alpha_{1}+\alpha_{2}} \sR_{\alpha_{1}} \sR_{\alpha_{2}} \sR_{\alpha_{1}+\alpha_{2}}$, 
\end{enu}
where the equalities in (1) and (2) follow from the Yang-Baxter equation \eqref{eq:Yang-Baxter_eq}. 
The following are the matrices (with respect to the basis $W$) of operators $\sQ_{\gamma}$, $\gamma \in \Delta^{+} = \{\alpha_{1}, 2\alpha_{1}+\alpha_{2}, \alpha_{1}+\alpha_{2}, \alpha_{2}\}$ (cf. \cite[Fig.~2\,(B)]{LL2}): 
\begin{align*}
\sQ_{\alpha_{1}} &= \left( 
\begin{array}{cccccccc} 
0 & Q_{1} & 0 & 0 & 0 & 0 & 0 & 0 \\ 
1 & 0 & 0 & 0 & 0 & 0 & 0 & 0 \\ 
0 & 0 & 0 & 0 & Q_{1} & 0 & 0 & 0 \\ 
0 & 0 & 0 & 0 & 0 & Q_{1} & 0 & 0 \\ 
0 & 0 & 1 & 0 & 0 & 0 & 0 & 0 \\ 
0 & 0 & 0 & 1 & 0 & 0 & 0 & 0 \\ 
0 & 0 & 0 & 0 & 0 & 0 & 0 & Q_{1} \\ 
0 & 0 & 0 & 0 & 0 & 0 & 1 & 0 
\end{array} 
\right); \displaybreak[1] \\ 
\sQ_{2\alpha_{1}+\alpha_{2}} &= \left( 
\begin{array}{cccccccc}
0 & 0 & 0 & 0 & 0 & Q_{1}Q_{2} & 0 & 0 \\ 
0 & 0 & 0 & 0 & 0 & 0 & 0 & 0 \\ 
0 & 0 & 0 & 0 & 0 & 0 & 0 & Q_{1}Q_{2} \\ 
0 & 0 & 0 & 0 & 0 & 0 & 0 & 0 \\ 
0 & 1 & 0 & 0 & 0 & 0 & 0 & 0 \\ 
0 & 0 & 0 & 0 & 0 & 0 & 0 & 0 \\ 
0 & 0 & 0 & 1 & 0 & 0 & 0 & 0 \\ 
0 & 0 & 0 & 0 & 0 & 0 & 0 & 0 
\end{array} 
\right); \displaybreak[1] \\ 
\sQ_{\alpha_{1}+\alpha_{2}} &= \left( 
\begin{array}{cccccccc}
0 & 0 & 0 & 0 & 0 & 0 & 0 & 0 \\ 
0 & 0 & 0 & 0 & 0 & 0 & 0 & 0 \\ 
0 & 0 & 0 & 0 & 0 & 0 & 0 & 0 \\ 
0 & 0 & 1 & 0 & 0 & 0 & 0 & 0 \\ 
0 & 0 & 0 & 0 & 0 & 0 & 0 & 0 \\ 
0 & 0 & 0 & 0 & 1 & 0 & 0 & 0 \\ 
0 & 0 & 0 & 0 & 0 & 0 & 0 & 0 \\ 
0 & 0 & 0 & 0 & 0 & 0 & 0 & 0 \\ 
\end{array}
\right); \displaybreak[1] \\ 
\sQ_{\alpha_{2}} &= \left( 
\begin{array}{cccccccc}
0 & 0 & Q_{2} & 0 & 0 & 0 & 0 & 0 \\ 
0 & 0 & 0 & Q_{2} & 0 & 0 & 0 & 0 \\ 
1 & 0 & 0 & 0 & 0 & 0 & 0 & 0 \\ 
0 & 1 & 0 & 0 & 0 & 0 & 0 & 0 \\ 
0 & 0 & 0 & 0 & 0 & 0 & Q_{2} & 0 \\ 
0 & 0 & 0 & 0 & 0 & 0 & 0 & Q_{2} \\ 
0 & 0 & 0 & 0 & 1 & 0 & 0 & 0 \\ 
0 & 0 & 0 & 0 & 0 & 1 & 0 & 0 
\end{array}
\right). 
\end{align*}
By explicit calculations (by using, e.g., SageMath), we obtain 
\begin{align*}
&\sR_{\alpha_{1}+\alpha_{2}} \sR_{\alpha_{2}} \sR_{-\alpha_{1}} \sR_{-2\alpha_{1}-\alpha_{2}} \\ 
&= \left( 
\begin{array}{cccccccc}
1 & Q_{1}Q_{2}-Q_{1} & Q_{2} & 0 & -Q_{1}Q_{2} & -Q_{1}Q_{2} & 0 & -Q_{1}Q_{2}^{2} \\ 
-1 & 1 & 0 & Q_{2} & 0 & 0 & 0 & 0 \\ 
1 & 0 & 1 & 0 & -Q_{1} & -Q_{1}Q_{2} & 0 & -Q_{1}Q_{2} \\ 
0 & 1 & 1 & 1 & -Q_{1} & -Q_{1} & 0 & -Q_{1}Q_{2} \\ 
0 & -1 & -1 & -Q_{2} & 1 & 0 & Q_{2} & 0 \\ 
0 & -1 & -1 & -1 & 1 & 1 & 0 & Q_{2} \\ 
0 & -1 & -1 & -1 & 1 & 0 & 1 & Q_{1}Q_{2} - Q_{1} \\ 
0 & 0 & 0 & 0 & 0 & 1 & -1 & 1 
\end{array}
\right), \\ 
&\sR_{2\alpha_{1}+\alpha_{2}} \sR_{\alpha_{1}} \sR_{-\alpha_{2}} \sR_{-\alpha_{1}-\alpha_{2}} \\ 
&= \left( 
\begin{array}{cccccccc}
1 & -Q_{1}Q_{2}+Q_{1} & -Q_{2} & 0 & -Q_{1}Q_{2} & Q_{1}Q_{2} & 0 & -Q_{1}Q_{2}^{2} \\ 
1 & 1 & 0 & -Q_{2} & 0 & 0 & 0 & 0 \\ 
-1 & 0 & 1 & 0 & Q_{1} & -Q_{1}Q_{2} & 0 & Q_{1}Q_{2} \\ 
0 & -1 & -1 & 1 & -Q_{1} & Q_{1} & 0 & -Q_{1}Q_{2} \\ 
0 & 1 & 1 & -Q_{2} & 1 & 0 & -Q_{2} & 0 \\ 
0 & -1 & -1 & 1 & -1 & 1 & 0 & -Q_{2} \\ 
0 & -1 & -1 & 1 & -1 & 0 & 1 & -Q_{1}Q_{2}+Q_{1} \\ 
0 & 0 & 0 & 0 & 0 & -1 & 1 & 1 
\end{array}
\right). 
\end{align*}
Also, we obtain 
\begin{align*}
&\sR_{\alpha_{1}+\alpha_{2}} \sR_{\alpha_{2}} \sR_{\alpha_{1}} \sR_{2\alpha_{1}+\alpha_{2}} \\ 
&= \left(
\begin{array}{cccccccc}
1 & Q_{1}Q_{2}+Q_{1} & Q_{2} & 0 & Q_{1}Q_{2} & Q_{1}Q_{2} & 0 & Q_{1}Q_{2}^{2} \\
1 & 1 & 0 & Q_{2} & 0 & 2Q_{1}Q_{2} & 0 &0 \\ 
1 & 2Q_{1} & 1 & 0 & Q_{1} & Q_{1}Q_{2} & 0 & Q_{1}Q_{2} \\ 
2 & 2Q_{1}+1 & 1 & 1 & Q_{1} & 2Q_{1}Q_{2}+Q_{1} & 0 & Q_{1}Q_{2} \\ 
0 & 1 & 1 & Q_{2} & 1 & 0 & Q_{2} & 2Q_{1}Q_{2} \\ 
0 & 1 & 1 & 2Q_{2}+1 & 1 & 1 & 2Q_{2} & 2Q_{1}Q_{2}+Q_{2} \\ 
0 & 1 & 1 & 1 & 1 & 0 & 1 & Q_{1}Q_{2}+Q_{1} \\ 
0 & 0 & 0 & 2 & 0 & 1 & 1 & 1 
\end{array}
\right), \\ 
&\sR_{2\alpha_{1}+\alpha_{2}} \sR_{\alpha_{1}} \sR_{\alpha_{2}} \sR_{\alpha_{1}+\alpha_{2}} \\ 
&= \left( 
\begin{array}{cccccccc}
1 & Q_{1}Q_{2} + Q_{1} & 2Q_{1}Q_{2}+Q_{2} & 2Q_{1}Q_{2} & Q_{1}Q_{2} & Q_{1}Q_{2} & 0 & Q_{1}Q_{2}^{2} \\ 
1 & 1 & 2Q_{2} & Q_{2} & 0 & 0 & 0 & 0 \\ 
1 & 0 & 1 & 0 & 2Q_{1}Q_{2}+Q_{1} & Q_{1}Q_{2} & 2Q_{1}Q_{2} & Q_{1}Q_{2} \\ 
0 & 1 & 1 & 1 & Q_{1} & Q_{1} & 0 & Q_{1}Q_{2} \\ 
2 & 1 & 2Q_{2}+1 & Q_{2} & 1 & 0 & Q_{2} & 0 \\ 
0 & 1 & 1 & 1 & 1 & 1 & 0 & Q_{2} \\ 
0 & 1 & 1 & 1 & 2Q_{1}+1 & 2Q_{1} & 1 & Q_{1}Q_{2}+Q_{1} \\ 
0 & 0 & 0 & 0 & 2 & 1 & 1 & 1 
\end{array}
\right). 
\end{align*}

This proves the proposition. 

%=========================%
% START SUBSECTION 0405 %
%=========================%

%%%%%%%%%%%%%%%%
% subsec:typeG2 %
%%%%%%%%%%%%%%%%

\subsection{Case of type \texorpdfstring{$G_{2}$}{G2}. }\label{subsec:typeG2}
We consider the root system $\Delta$ of type $G_{2}$. 
We have $q = 6$, and 
$(\beta_{1}, \beta_{2}, \beta_{3}, \beta_{4}, \beta_{5}, \beta_{6}) = (\alpha_{1}, 3\alpha_{1}+\alpha_{2}, 2\alpha_{1}+\alpha_{2}, 3\alpha_{1}+2\alpha_{2}, \alpha_{1}+\alpha_{2}, \alpha_{2})$ 
or $(\beta_{1}, \beta_{2}, \beta_{3}, \beta_{4}, \beta_{5}, \beta_{6}) = (\alpha_{2}, \alpha_{1}+\alpha_{2}, 3\alpha_{1}+2\alpha_{2}, 2\alpha_{1}+\alpha_{2}, 3\alpha_{1}+\alpha_{2}, \alpha_{1})$. 
In this case, a key proposition to the proof of Theorem~\ref{thm:YB-move} is slightly different from that in the other types. 
First, we state the following proposition for quantum Bruhat operators in type $G_{2}$. 

%
%%%%%%%%%%%%%%%
% prop:matrix_G %
%%%%%%%%%%%%%%%
%
\begin{prop}\label{prop:matrix_G}
\begin{enu}
\item All the entries of the matrix of $\sS_{k}$, $k = 0, 1, \ldots, 6$, are of the form $\sum_{j = 1}^{r} m_{j}Q^{\xi_{j}}$, where all $\xi_{j} \in Q^{\vee,+}$ are distinct, and $m_{j} \in \{1, 2, 3\}$. 
\item Let $v, w \in W$. Assume that the $(v, w)$-entry of the matrix of $\sS_{k}$ is of the form $\sum_{j = 1}^{r} m_{j}Q^{\xi_{j}}$ as in (1). 
Also, assume that the $(v, w)$-entry of the matrix of $\sT_{k}^{\pm}$ is of the form $\sum_{\xi \in Q^{\vee, +}} n_{\xi}^{\pm} Q^{\xi}$. 
For $j = 1, \ldots, r$, if $m_{j} = 2$, then $n_{\xi_{j}}^{\pm} = 0$, and if $m_{j} = 1$ or if $m_{j} = 3$, then $n_{\xi_{j}}^{\pm} \in \{1, -1\}$. 
Moreover, for $\xi \in Q^{\vee, +} \setminus \{\xi_{1}, \ldots, \xi_{r}\}$, we have $n_{\xi}^{\pm} = 0$. 
\item Let $v, w \in W$. Assume that the $(v, w)$-entry of the matrix of $\sS_{k}$ is of the form $\sum_{j = 1}^{r} m_{j}Q^{\xi_{j}}$ as in (1). 
Also, assume that the $(v, w)$-entry of the matrix of $\sS'_{k}$ is of the form $\sum_{\xi \in Q^{\vee, +}} n_{\xi}Q^{\xi}$. 
For $j = 1, \ldots, r$, if $m_{j} = 2$, then $n_{\xi_{j}} = 0$ or $n_{\xi_{j}} = 2$, and if $m_{j} = 1$, then $n_{\xi_{j}} = 1$ or $n_{\xi_{j}} = 3$. 
Moreover, if $m_{j} = 3$, then $n_{\xi_{j}} = 1$. 
\end{enu}
\end{prop}

The proof in the case $k = 0, 1, 5, 6$ is the same as that in types $A_{1} \times A_{1}$, $A_{2}$, and $C_{2}$, given above. 
Since the case $k =2, 3, 4$ is remaining, we need to calculate the matrices (with respect to the basis 
\begin{equation*}
W = \{e, s_{1}, s_{2}, s_{1}s_{2}, s_{2}s_{1}, s_{1}s_{2}s_{1}, s_{2}s_{1}s_{2}, s_{1}s_{2}s_{1}s_{2}, s_{2}s_{1}s_{2}s_{1}, s_{1}s_{2}s_{1}s_{2}s_{1}, s_{2}s_{1}s_{2}s_{1}s_{2}, w_{\circ}\} 
\end{equation*}
of $K[W]$) of the following 12 operators: 
\begin{enu}
\item $\sR_{2\alpha_{1}+\alpha_{2}} \sR_{3\alpha_{1}+2\alpha_{2}} \sR_{\alpha_{1}+\alpha_{2}} \sR_{\alpha_{2}} \sR_{-\alpha_{1}} \sR_{-3\alpha_{1}-\alpha_{2}}$; 
\item $\sR_{3\alpha_{1}+2\alpha_{2}} \sR_{\alpha_{1}+\alpha_{2}} \sR_{\alpha_{2}} \sR_{-\alpha_{1}} \sR_{-3\alpha_{1}-\alpha_{2}} \sR_{-2\alpha_{1}-\alpha_{2}}$; 
\item $\sR_{\alpha_{1}+\alpha_{2}} \sR_{\alpha_{2}} \sR_{-\alpha_{1}} \sR_{-3\alpha_{1}-\alpha_{2}} \sR_{-2\alpha_{1}-\alpha_{2}} \sR_{-3\alpha_{1}-2\alpha_{2}}$; 
\item $\sR_{3\alpha_{1}+2\alpha_{2}} \sR_{2\alpha_{1}+\alpha_{2}} \sR_{3\alpha_{1}+\alpha_{2}} \sR_{\alpha_{1}} \sR_{-\alpha_{2}} \sR_{-\alpha_{1}-\alpha_{2}}$; 
\item $\sR_{2\alpha_{1}+\alpha_{2}} \sR_{3\alpha_{1}+\alpha_{2}} \sR_{\alpha_{1}} \sR_{-\alpha_{2}} \sR_{-\alpha_{1}-\alpha_{2}} \sR_{-3\alpha_{1}-2\alpha_{2}}$; 
\item $\sR_{3\alpha_{1}+\alpha_{2}} \sR_{\alpha_{1}} \sR_{-\alpha_{2}} \sR_{-\alpha_{1}-\alpha_{2}} \sR_{-3\alpha_{1}-2\alpha_{2}} \sR_{-2\alpha_{1}-\alpha_{2}}$; 
\item $\sR_{2\alpha_{1}+\alpha_{2}} \sR_{3\alpha_{1}+2\alpha_{2}} \sR_{\alpha_{1}+\alpha_{2}} \sR_{\alpha_{2}} \sR_{\alpha_{1}} \sR_{3\alpha_{1}+\alpha_{2}}$; 
\item $\sR_{3\alpha_{1}+2\alpha_{2}} \sR_{\alpha_{1}+\alpha_{2}} \sR_{\alpha_{2}} \sR_{\alpha_{1}} \sR_{3\alpha_{1}+\alpha_{2}} \sR_{2\alpha_{1}+\alpha_{2}}$; 
\item $\sR_{\alpha_{1}+\alpha_{2}} \sR_{\alpha_{2}} \sR_{\alpha_{1}} \sR_{3\alpha_{1}+\alpha_{2}} \sR_{2\alpha_{1}+\alpha_{2}} \sR_{3\alpha_{1}+2\alpha_{2}}$; 
\item $\sR_{3\alpha_{1}+2\alpha_{2}} \sR_{2\alpha_{1}+\alpha_{2}} \sR_{3\alpha_{1}+\alpha_{2}} \sR_{\alpha_{1}} \sR_{\alpha_{2}} \sR_{\alpha_{1}+\alpha_{2}}$; 
\item $\sR_{2\alpha_{1}+\alpha_{2}} \sR_{3\alpha_{1}+\alpha_{2}} \sR_{\alpha_{1}} \sR_{\alpha_{2}} \sR_{\alpha_{1}+\alpha_{2}} \sR_{3\alpha_{1}+2\alpha_{2}}$; 
\item $\sR_{3\alpha_{1}+\alpha_{2}} \sR_{\alpha_{1}} \sR_{\alpha_{2}} \sR_{\alpha_{1}+\alpha_{2}} \sR_{3\alpha_{1}+2\alpha_{2}} \sR_{2\alpha_{1}+\alpha_{2}}$. 
\end{enu}
Proposition~\ref{prop:matrix_G} can be verified by direct calculations; 
the list of results of calculations is given in Appendix~\ref{sec:App_A}. 

%
%%%%%%%%%%%%%%%%%
% rem:exception_G %
%%%%%%%%%%%%%%%%%
%
\begin{rem}\label{rem:exception_G}
From direct calculations of matrices of the operators above, 
we see that only in the following two cases, $m_{j} = 3$ in Proposition~\ref{prop:matrix_G}\,(1), or $n_{\xi_{j}} = 3$ in Proposition~\ref{prop:matrix_G}\,(3): 
\begin{enu}
\item The $(4, 6)$-entry (i.e., the $(s_{1}s_{2}, s_{1}s_{2}s_{1})$-entry) of the matrix of the operator (9); 
\item The $(7, 9)$-entry (i.e., the $(s_{2}s_{1}s_{2}, s_{2}s_{1}s_{2}s_{1})$-entry) of the matrix of the operator (10). 
\end{enu}
\end{rem}

Let $v \in W$, and let $\bp$ be a $\Pi$-compatible directed path in $\QBG(W)$ which starts at $v$, i.e., $\bp \in \CP(v, \Pi)$. 
First, we consider the cases except the following: 
\begin{itemize}
\item[(E1)] $\Pi = (\pm (3\alpha_{1}+2\alpha_{2}), \pm (2\alpha_{1}+\alpha_{2}), \pm (3\alpha_{1}+\alpha_{2}), \pm \alpha_{1}, \mp \alpha_{2}, \mp (\alpha_{1}+\alpha_{2}))$, 
$v = s_{1}s_{2}s_{1}$, $\ed(\bp) = s_{1}s_{2}$, and $\wt(\bp) = \alpha_{1}^{\vee}+\alpha_{2}^{\vee}$; 
\item[(E2)] $\Pi = (\pm(\alpha_{1}+\alpha_{2}), \pm \alpha_{2}, \mp \alpha_{1}, \mp(3\alpha_{1}+\alpha_{2}), \mp(2\alpha_{1}+\alpha_{2}), \mp(3\alpha_{1}+2\alpha_{2}))$, 
$v = s_{1}s_{2}s_{1}$, $\ed(\bp) = s_{1}s_{2}$, and $\wt(\bp) = \alpha_{1}^{\vee}+\alpha_{2}^{\vee}$; 
\item[(E3)] $\Pi = (\pm (3\alpha_{1}+2\alpha_{2}), \pm (2\alpha_{1}+\alpha_{2}), \pm (3\alpha_{1}+\alpha_{2}), \pm \alpha_{1}, \mp \alpha_{2}, \mp (\alpha_{1}+\alpha_{2}))$, 
$v = s_{2}s_{1}s_{2}s_{1}$, $\ed(\bp) = s_{2}s_{1}s_{2}$, and $\wt(\bp) = \alpha_{1}^{\vee}+\alpha_{2}^{\vee}$; 
\item[(E4)] $\Pi = (\pm(\alpha_{1}+\alpha_{2}), \pm \alpha_{2}, \mp \alpha_{1}, \mp(3\alpha_{1}+\alpha_{2}), \mp(2\alpha_{1}+\alpha_{2}), \mp(3\alpha_{1}+2\alpha_{2}))$, 
$v = s_{2}s_{1}s_{2}s_{1}$, $\ed(\bp) = s_{2}s_{1}s_{2}$, and $\wt(\bp) = \alpha_{1}^{\vee}+\alpha_{2}^{\vee}$. 
\end{itemize}

Then, we have the following:

%
%%%%%%%%%%%%%%%%%%%%%%%%%
% prop:rank2_shellability_G1 %
%%%%%%%%%%%%%%%%%%%%%%%%%
%
\begin{prop}\label{prop:rank2_shellability_G1} 
Only one of the following occurs. 
\begin{enu}
\item There exists a unique $\bp' \in \CP(v, \Pi) \setminus \{\bp\}$ such that $\ed(\bp') = \ed(\bp)$ and $\wt(\bp') = \wt(\bp)$. 
This $\bp'$ satisfies $(-1)^{\nega(\bp')} = -(-1)^{\nega(\bp)}$. 
Moreover, there does not exist a path $\bq \in \CP(v, \Pi')$ such that $\ed(\bq) = \ed(\bp)$ and $\wt(\bq) = \wt(\bp)$. 
\item There exists a unique $\bp' \in \CP(v, \Pi')$ such that $\ed(\bp') = \ed(\bp)$ and $\wt(\bp') = \wt(\bp)$. 
This $\bp'$ satisfies $(-1)^{\nega(\bp')} = (-1)^{\nega(\bp)}$. 
Moreover, there does not exist a path $\bq \in \CP(v, \Pi) \setminus \{\bp\}$ such that $\ed(\bq) = \ed(\bp)$ and $\wt(\bq) = \wt(\bp)$. 
\item There exists a unique $\bp' \in \CP(v, \Pi) \setminus \{\bp\}$ such that $\ed(\bp') = \ed(\bp)$ and $\wt(\bp') = \wt(\bp)$. 
This $\bp'$ satisfies $(-1)^{\nega(\bp')} = -(-1)^{\nega(\bp)}$. 
Moreover, there exist exactly two paths $\bq_{1}, \bq_{2} \in \CP(v, \Pi')$ such that $\ed(\bq_{1}) = \ed(\bq_{2}) = \ed(\bp)$ and $\wt(\bq_{1}) = \wt(\bq_{2}) = \wt(\bp)$. 
These $\bq_{1}$, $\bq_{2}$ satisfy $(-1)^{\nega(\bq_{2})} = -(-1)^{\nega(\bq_{1})}$. 
\end{enu}
\end{prop}

Next, we consider the exceptional cases (E1)--(E4) above.
Then, we have the following:

%
%%%%%%%%%%%%%%%%%%%%%%%%%
% prop:rank2_shellability_G2 %
%%%%%%%%%%%%%%%%%%%%%%%%%
%
\begin{prop}\label{prop:rank2_shellability_G2} 
\begin{enu}
\item In case (E1) or (E4), there exist exactly two paths $\bp_{1}, \bp_{2} \in \CP(v, \Pi) \setminus \{ \bp \}$ and a unique path $\bq \in \CP(v, \Pi')$ 
such that $\ed(\bp_{1}) = \ed(\bp_{2}) = \ed(\bq) = \ed(\bp)$ and $\wt(\bp_{1}) = \wt(\bp_{2}) = \wt(\bq) = \wt(\bp)$. 
Moreover, there exist two paths $\br_{1}, \br_{2} \in \{\bp, \bp_{1}, \bp_{2}\}$ such that 
for $\br_{3} \in \{\bp, \bp_{1}, \bp_{2}\} \setminus \{\br_{1}, \br_{2}\}$, $(-1)^{\nega(\br_{1})} = (-1)^{\nega(\br_{2})} = -(-1)^{\nega(\br_{3})}$ and $(-1)^{\nega(\br_{1})} = (-1)^{\nega(\bq)}$. 
\item In case (E2) or (E3), there exist exactly three paths $\bq_{1}, \bq_{2}, \bq_{3} \in \CP(v, \Pi')$  
such that $\ed(\bq_{1}) = \ed(\bq_{2}) = \ed(\bq_{3}) = \ed(\bp)$ and $\wt(\bq_{1}) = \wt(\bq_{2}) = \wt(\bq_{3}) = \wt(\bp)$. 
In this case, there does not exist a path $\br \in \CP(v, \Pi) \setminus \{\bp\}$ such that $\ed(\br) = \ed(\bp)$ and $\wt(\br) = \wt(\bp)$. 
Moreover, there exist two paths $\br_{1}, \br_{2} \in \{\bq_{1}, \bq_{2}, \bq_{3}\}$ such that for $\br_{3} \in \{\bq_{1}, \bq_{2}, \bq_{3}\} \setminus \{\br_{1}, \br_{2}\}$, 
$(-1)^{\nega(\br_{1})} = (-1)^{\nega(\br_{2})} = -(-1)^{\nega(\br_{3})}$ and $(-1)^{\nega(\br_{1})} = (-1)^{\nega(\bp)}$. 
\end{enu}
\end{prop}

By using Proposition~\ref{prop:matrix_G} and Remark~\ref{rem:exception_G}, we can prove Propositions~\ref{prop:rank2_shellability_G1} and \ref{prop:rank2_shellability_G2} by the same argument as in types $A_{1} \times A_{1}$, $A_{2}$, and $C_{2}$. 

%
%%%%%%%%%%%%%%%%%%%%%%%%%%%%%%
% rem:explicit_rank2_shellability_G %
%%%%%%%%%%%%%%%%%%%%%%%%%%%%%%
%
\begin{rem}\label{rem:explicit_rank2_shellability_G}
We have explicit descriptions of $\bp_{1}$, $\bp_{2}$, $\bq$ in Proposition~\ref{prop:rank2_shellability_G2}\,(1), and those of $\bq_{1}$, $\bq_{2}$, $\bq_{3}$ in Proposition~\ref{prop:rank2_shellability_G2}\,(2). 
\begin{enu}
\item  Case (E1). 
In this case, we have 
\begin{equation*}
\Pi = (\pm (3\alpha_{1}+2\alpha_{2}), \pm (2\alpha_{1}+\alpha_{2}), \pm (3\alpha_{1}+\alpha_{2}), \pm \alpha_{1}, \mp \alpha_{2}, \mp (\alpha_{1}+\alpha_{2}))
\end{equation*}
and 
\begin{equation*}
\Pi' = (\mp(\alpha_{1}+\alpha_{2}), \mp \alpha_{2}, \pm \alpha_{1}, \pm(3\alpha_{1}+\alpha_{2}), \pm(2\alpha_{1}+\alpha_{2}), \pm(3\alpha_{1}+2\alpha_{2})). 
\end{equation*}
If we set 
\begin{align*}
\bp^{\E{1}}_{1} &: s_{1}s_{2}s_{1} \xrightarrow{3\alpha_{1}+2\alpha_{2}} s_{2}s_{1}s_{2}s_{1} \xrightarrow{3\alpha_{1}+\alpha_{2}} s_{2} \xrightarrow{\alpha_{1}+\alpha_{2}} s_{1}s_{2}, \\ 
\bp^{\E{1}}_{2} &: s_{1}s_{2}s_{1} \xrightarrow{3\alpha_{1}+\alpha_{2}} e \xrightarrow{\alpha_{1}} s_{1} \xrightarrow{\alpha_{2}} s_{1}s_{2}, \\ 
\bp^{\E{1}}_{3} &: s_{1}s_{2}s_{1} \xrightarrow{3\alpha_{1}+\alpha_{2}} e \xrightarrow{\alpha_{2}} s_{2} \xrightarrow{\alpha_{1}+\alpha_{2}} s_{1}s_{2}, \\ 
\bq^{\E{1}} &: s_{1}s_{2}s_{1} \xrightarrow{\alpha_{2}} s_{1}s_{2}s_{1}s_{2} \xrightarrow{\alpha_{1}} s_{1}s_{2}s_{1}s_{2}s_{1} \xrightarrow{3\alpha_{1}+\alpha_{2}} s_{1}s_{2}, 
\end{align*}
then in Proposition~\ref{prop:rank2_shellability_G2}\,(1), we have $\{\bp, \bp_{1}, \bp_{2}\} = \{\bp^{\E{1}}_{1}, \bp^{\E{1}}_{2}, \bp^{\E{1}}_{3}\}$ 
and $\bq = \bq^{\E{1}}$. 
\item Case (E2). In this case, we have 
\begin{equation*}
\Pi = (\pm(\alpha_{1}+\alpha_{2}), \pm \alpha_{2}, \mp \alpha_{1}, \mp(3\alpha_{1}+\alpha_{2}), \mp(2\alpha_{1}+\alpha_{2}), \mp(3\alpha_{1}+2\alpha_{2}))
\end{equation*}
and 
\begin{equation*}
\Pi' = (\mp (3\alpha_{1}+2\alpha_{2}), \mp (2\alpha_{1}+\alpha_{2}), \mp (3\alpha_{1}+\alpha_{2}), \mp \alpha_{1}, \pm \alpha_{2}, \pm (\alpha_{1}+\alpha_{2})). 
\end{equation*}
If we set 
\begin{align*}
\bp^{\E{2}} &: s_{1}s_{2}s_{1} \xrightarrow{\alpha_{2}} s_{1}s_{2}s_{1}s_{2} \xrightarrow{\alpha_{1}} s_{1}s_{2}s_{1}s_{2}s_{1} \xrightarrow{3\alpha_{1}+\alpha_{2}} s_{1}s_{2}, \\ 
\bq^{\E{2}}_{1} &: s_{1}s_{2}s_{1} \xrightarrow{3\alpha_{1}+2\alpha_{2}} s_{2}s_{1}s_{2}s_{1} \xrightarrow{3\alpha_{1}+\alpha_{2}} s_{2} \xrightarrow{\alpha_{1}+\alpha_{2}} s_{1}s_{2}, \\ 
\bq^{\E{2}}_{2} &: s_{1}s_{2}s_{1} \xrightarrow{3\alpha_{1}+\alpha_{2}} e \xrightarrow{\alpha_{1}} s_{1} \xrightarrow{\alpha_{2}} s_{1}s_{2}, \\ 
\bq^{\E{2}}_{3} &: s_{1}s_{2}s_{1} \xrightarrow{3\alpha_{1}+\alpha_{2}} e \xrightarrow{\alpha_{2}} s_{2} \xrightarrow{\alpha_{1}+\alpha_{2}} s_{1}s_{2}, 
\end{align*}
then in Proposition~\ref{prop:rank2_shellability_G2}\,(2), we have $\bp = \bp^{\E{2}}$ and $\{\bq_{1}, \bq_{2}, \bq_{3}\} = \{\bq^{\E{2}}_{1}, \bq^{\E{2}}_{2}, \bq^{\E{2}}_{3}\}$. 

\item Case (E3). 
In this case, we have 
\begin{equation*}
\Pi = (\pm (3\alpha_{1}+2\alpha_{2}), \pm (2\alpha_{1}+\alpha_{2}), \pm (3\alpha_{1}+\alpha_{2}), \pm \alpha_{1}, \mp \alpha_{2}, \mp (\alpha_{1}+\alpha_{2}))
\end{equation*}
and 
\begin{equation*}
\Pi' = (\mp(\alpha_{1}+\alpha_{2}), \mp \alpha_{2}, \pm \alpha_{1}, \pm(3\alpha_{1}+\alpha_{2}), \pm(2\alpha_{1}+\alpha_{2}), \pm(3\alpha_{1}+2\alpha_{2})). 
\end{equation*}
If we set 
\begin{align*}
\bp^{\E{3}} &: s_{2}s_{1}s_{2}s_{1} \xrightarrow{3\alpha_{1}+\alpha_{2}} s_{2} \xrightarrow{\alpha_{1}} s_{2}s_{1} \xrightarrow{\alpha_{2}} s_{2}s_{1}s_{2}, \\ 
\bq^{\E{3}}_{1} &: s_{2}s_{1}s_{2}s_{1} \xrightarrow{\alpha_{1}+\alpha_{2}} s_{1}s_{2}s_{1}s_{2}s_{1} \xrightarrow{3\alpha_{1}+\alpha_{2}} s_{1}s_{2} \xrightarrow{3\alpha_{1}+2\alpha_{2}} s_{2}s_{1}s_{2} \\ 
\bq^{\E{3}}_{2} &: s_{2}s_{1}s_{2}s_{1} \xrightarrow{\alpha_{2}} s_{2}s_{1}s_{2}s_{1}s_{2} \xrightarrow{\alpha_{1}} w_{\circ} \xrightarrow{3\alpha_{1}+2\alpha_{2}} s_{2}s_{1}s_{2} \\ 
\bq^{\E{3}}_{3} &: s_{2}s_{1}s_{2}s_{1} \xrightarrow{\alpha_{1}+\alpha_{2}} s_{1}s_{2}s_{1}s_{2}s_{1} \xrightarrow{\alpha_{2}} w_{\circ} \xrightarrow{3\alpha_{1}+\alpha_{2}} s_{2}s_{1}s_{2}, 
\end{align*}
then in Proposition~\ref{prop:rank2_shellability_G2}\,(2), we have $\bp = \bp^{\E{3}}$, and $\{\bq_{1}, \bq_{2}, \bq_{3}\} = \{\bq^{\E{3}}_{1}, \bq^{\E{3}}_{2}, \bq^{\E{3}}_{3}\}$. 
\item Case (E4). In this case, we have 
\begin{equation*}
\Pi = (\pm(\alpha_{1}+\alpha_{2}), \pm \alpha_{2}, \mp \alpha_{1}, \mp(3\alpha_{1}+\alpha_{2}), \mp(2\alpha_{1}+\alpha_{2}), \mp(3\alpha_{1}+2\alpha_{2}))
\end{equation*}
and 
\begin{equation*}
\Pi' = (\mp (3\alpha_{1}+2\alpha_{2}), \mp (2\alpha_{1}+\alpha_{2}), \mp (3\alpha_{1}+\alpha_{2}), \mp \alpha_{1}, \pm \alpha_{2}, \pm (\alpha_{1}+\alpha_{2})). 
\end{equation*}
\end{enu}
If we set 
\begin{align*}
\bp^{\E{4}}_{1} &: s_{2}s_{1}s_{2}s_{1} \xrightarrow{\alpha_{1}+\alpha_{2}} s_{1}s_{2}s_{1}s_{2}s_{1} \xrightarrow{3\alpha_{1}+\alpha_{2}} s_{1}s_{2} \xrightarrow{3\alpha_{1}+2\alpha_{2}} s_{2}s_{1}s_{2} \\ 
\bp^{\E{4}}_{2} &: s_{2}s_{1}s_{2}s_{1} \xrightarrow{\alpha_{2}} s_{2}s_{1}s_{2}s_{1}s_{2} \xrightarrow{\alpha_{1}} w_{\circ} \xrightarrow{3\alpha_{1}+2\alpha_{2}} s_{2}s_{1}s_{2} \\ 
\bp^{\E{4}}_{3} &: s_{2}s_{1}s_{2}s_{1} \xrightarrow{\alpha_{1}+\alpha_{2}} s_{1}s_{2}s_{1}s_{2}s_{1} \xrightarrow{\alpha_{2}} w_{\circ} \xrightarrow{3\alpha_{1}+\alpha_{2}} s_{2}s_{1}s_{2}, \\ 
\bq^{\E{4}} &: s_{2}s_{1}s_{2}s_{1} \xrightarrow{3\alpha_{1}+\alpha_{2}} s_{2} \xrightarrow{\alpha_{1}} s_{2}s_{1} \xrightarrow{\alpha_{2}} s_{2}s_{1}s_{2},  
\end{align*}
then in Proposition~\ref{prop:rank2_shellability_G2}\,(1), we have $\{\bp, \bp_{1}, \bp_{2}\} = \{\bp^{\E{4}}_{1}, \bp^{\E{4}}_{2}, \bp^{\E{4}}_{3}\}$ 
and $\bq = \bq^{\E{4}}$. 
\end{rem}

%=========================%
% START SUBSECTION 0406 %
%=========================%

\subsection{Proof of Theorem~\ref{thm:YB-move}.}
Based on Propositions~\ref{prop:rank2_shellability}, \ref{prop:rank2_shellability_G1}, and \ref{prop:rank2_shellability_G2}, we can prove the existence of a generalization of quantum Yang-Baxter moves. 
In the same way as in \eqref{eq:division_A}, we divide $\bp(A)$ for $A \in \CA(w, \Gamma_{1})$ into three parts $\bp(A)^{(1)}$, $\bp(A)^{(2)}$, $\bp(A)^{(3)}$.
If we write $A = \{a_{1}, \ldots, a_{l}\}$, then $\bp(A)$ is of the form: 
\begin{equation*}
\bp(A): w = w_{0} \xrightarrow{|\beta_{a_{1}}|} \cdots \xrightarrow{|\beta_{a_{l}}|} w_{l}, 
\end{equation*}
with $a_{1} < \cdots < a_{l}$; we set $a_{0} := 0$. 
Let $0 \le i_{1} \le l$ be maximal such that $a_{i_{1}} \le t$, 
and $0 \le i_{2} \le l$ maximal such that $a_{i_{2}} \le t+q$. 
Then, we set 
\begin{align*}
& \bp(A)^{(1)}: w = w_{0} \xrightarrow{|\beta_{a_{1}}|} \cdots \xrightarrow{|\beta_{a_{i_{1}}}|} w_{a_{i_{1}}}, \\ 
& \bp(A)^{(2)}: w_{a_{i_{1}}} \xrightarrow{|\beta_{a_{i_{1}+1}}|} \cdots \xrightarrow{|\beta_{a_{i_{2}}}|} w_{a_{i_{2}}}, \\ 
& \bp(A)^{(3)}: w_{a_{i_{2}}} \xrightarrow{|\beta_{a_{i_{2}+1}}|} \cdots \xrightarrow{|\beta_{a_{l}}|} w_{a_{l}}. 
\end{align*}
Note that the concatenation of $\bp(A)^{(1)}$, $\bp(A)^{(2)}$, and $\bp(A)^{(3)}$ coincides with $\bp(A)$. 

Also, in the same way as in \eqref{eq:division_B}, we divide $\bp(B)$ for each $B \in \CA(w, \Gamma_{2})$ into three parts $\bp(B)^{(1)}$, $\bp(B)^{(2)}$, $\bp(B)^{(3)}$. 
If we write $B = \{b_{1}, \ldots, b_{m}\}$, then $\bp(B)$ is of the form: 
\begin{equation*}
\bp(B): w = w_{0} \xrightarrow{|\beta'_{b_{1}}|} \cdots \xrightarrow{|\beta'_{b_{m}}|} w_{m}, 
\end{equation*}
with $b_{1} < \cdots < b_{m}$; we set $b_{0} := 0$. 
Let $0 \le i_{1} \le m$ be maximal such that $b_{i_{1}} \le t$, 
and $0 \le i_{2} \le m$ maximal such that $b_{i_{2}} \le t+q$. 
Then, we set 
\begin{align*}
& \bp(B)^{(1)}: w = w_{0} \xrightarrow{|\beta'_{b_{1}}|} \cdots \xrightarrow{|\beta'_{b_{i_{1}}}|} w_{b_{i_{1}}}, \\ 
& \bp(B)^{(2)}: w_{b_{i_{1}}} \xrightarrow{|\beta'_{b_{i_{1}+1}}|} \cdots \xrightarrow{|\beta'_{b_{i_{2}}}|} w_{b_{i_{2}}}, \\ 
& \bp(B)^{(3)}: w_{b_{i_{2}}} \xrightarrow{|\beta'_{b_{i_{2}+1}}|} \cdots \xrightarrow{|\beta'_{b_{m}}|} w_{b_{m}}. 
\end{align*}
Note that the concatenation of $\bp(B)^{(1)}$, $\bp(B)^{(2)}$, and $\bp(B)^{(3)}$ coincides with $\bp(B)$.

\begin{proof}[Proof of Theorem~\ref{thm:YB-move}]
We divide the proof of the theorem into two parts: 
\begin{enu}
\item $\Delta$ is not of type $G_{2}$; 
\item $\Delta$ is of type $G_{2}$. 
\end{enu}

\paragraph{Part~1: $\Delta$ is not of type $G_{2}$.}
We assume that $\Delta$ is not of type $G_{2}$. 
Let $A \in \CA(w, \Gamma_{1})$. 
Then, by Proposition~\ref{prop:rank2_shellability} with $\Pi = \Gamma_{1}^{(2)}$ and $\Pi' = \Gamma_{2}^{(2)}$, we see that only one of the following occurs: 
\begin{enu}
\item there exists $\br_{0} \in \CP(\ed(\bp(A)^{(1)}), \Gamma_{1}^{(2)}) \setminus \{\bp(A)^{(2)}\}$ such that $\ed(\br_{0}) = \ed(\bp(A)^{(2)})$ and $\wt(\br_{0}) = \wt(\bp(A)^{(2)})$; 
\item there exists $\br_{0} \in \CP(\ed(\bp(A)^{(1)}), \Gamma_{2}^{(2)})$ such that $\ed(\br_{0}) = \ed(\bp(A)^{(2)})$ and $\wt(\br_{0}) = \wt(\bp(A)^{(2)})$. 
\end{enu}
For convenience of explanation, we set 
\begin{equation*}
\varphi(A) := \begin{cases} 1 & \text{if (1) of the above holds,} \\ 2 & \text{if (2) of the above holds.} \end{cases}
\end{equation*}
We define a set $\CA_{0}(w, \Gamma_{1}) \subset \CA(w, \Gamma_{1})$ by 
\begin{equation*}
\CA_{0}(w, \Gamma_{1}) := \{ A \in \CA(w, \Gamma_{1}) \mid \varphi(A) = 2\}. 
\end{equation*}
Then we have 
%
%%%%%%%%%%%
% eq:CA0C %
%%%%%%%%%%%
%
\begin{equation}\label{eq:CA0C}
\CA_{0}^{C}(w, \Gamma_{1}) = \CA(w, \Gamma_{1}) \setminus \CA_{0}(w, \Gamma_{1}) = \{ A \in \CA(w, \Gamma_{1}) \mid \varphi(A) = 1\}. 
\end{equation}

Let us define a map $Y: \CA_{0}(w, \Gamma_{1}) \rightarrow \CA(w, \Gamma_{2})$. 
Let $A \in \CA_{0}(w, \Gamma_{1})$. 
Then, by applying Proposition~\ref{prop:rank2_shellability} with $\Pi = \Gamma_{1}^{(2)}$ and $\Pi' = \Gamma_{2}^{(2)}$, 
there exists a unique $\br_{0} \in \CP(\ed(\bp(A)^{(1)}), \Gamma_{2}^{(2)})$ such that $\ed(\br_{0}) = \ed(\bp(A)^{(2)})$ and $\wt(\br_{0}) = \wt(\bp(A)^{(2)})$. 
We write the $\br_{0}$ as: 
\begin{equation*}
\br_{0}: \ed(\bp(A)^{(1)}) = x_{0} \xrightarrow{|\beta'_{j_{1}}|} \cdots \xrightarrow{|\beta'_{j_{p}}|} x_{p}.  
\end{equation*}
Since $\br_{0}$ is $\Gamma_{2}^{(2)}$-compatible, we have $t+1 \le j_{1} < \cdots < j_{p} \le t+q$. 
Set $B^{(2)} := \{j_{1}, \ldots, j_{p}\}$, and define $Y(A)$ by $Y(A) := A^{(1)} \sqcup B^{(2)} \sqcup A^{(3)}$; 
note that $Y(A) \in \CA(w, \Gamma_{2})$. 
We define a set $\CA_{0}(w, \Gamma_{2})$ by 
\begin{equation*}
\CA_{0}(w, \Gamma_{2}) := \{ Y(A) \mid A \in \CA_{0}(w, \Gamma_{1}) \}. 
\end{equation*}
We claim that $Y$ defines a bijection $Y: \CA_{0}(w, \Gamma_{1}) \rightarrow \CA_{0}(w, \Gamma_{2})$. 
To verify this claim, it suffices to show that $Y$ is injective. 

Let $A_{1}, A_{2} \in \CA_{0}(w, \Gamma_{1})$, and assume that $Y(A_{1}) = Y(A_{2})$. 
We show that $A_{1} = A_{2}$. 
First, we see that 
\begin{equation*}
A_{1}^{(1)} = (Y(A_{1}))^{(1)} = (Y(A_{2}))^{(1)} = A_{2}^{(1)}, 
\end{equation*}
and 
\begin{equation*}
A_{1}^{(3)} = (Y(A_{1}))^{(3)} = (Y(A_{2}))^{(3)} = A_{2}^{(3)}. 
\end{equation*}
Hence it remains to show that $A_{1}^{(2)} = A_{2}^{(2)}$. 
By the definition of the map $Y$, we have $\ed(\bp(Y(A_{1}))^{(2)}) = \ed(\bp(A_{1})^{(2)})$ and $\wt(\bp(Y(A_{1}))^{(2)}) = \wt(\bp(A_{1})^{(2)})$. 
Also, we have $\ed(\bp(Y(A_{2}))^{(2)}) = \ed(\bp(A_{2})^{(2)})$ and $\wt(\bp(Y(A_{2}))^{(2)}) = \wt(\bp(A_{2})^{(2)})$. 
Since $\bp(Y(A_{1}))^{(2)} = \bp(Y(A_{2}))^{(2)}$, the uniqueness in Proposition~\ref{prop:rank2_shellability}\,(2) (with $\Pi = \Gamma_{2}^{(2)}$ and $\Pi' = \Gamma_{1}^{(2)}$) 
implies that $\bp(A_{1})^{(2)} = \bp(A_{2})^{(2)}$, from which we obtain $A_{1}^{(2)} = A_{2}^{(2)}$, as desired. 
This shows the injectivity of $Y$. 

To prove that $Y$ satisfies the condition of Theorem~\ref{thm:YB-move}\,(1), 
it remains to show that $\ed(Y(A)) = \ed(A)$, $\down(Y(A)) = \down(A)$, and $(-1)^{n(Y(A))} = (-1)^{n(A)}$. 
The first equation is obvious, since 
\begin{equation*}
\ed(Y(A)) = \ed(\bp(Y(A))) = \ed(\bp(Y(A))^{(3)}) = \ed(\bp(A)^{(3)}) = \ed(\bp(A)) = \ed(A). 
\end{equation*}
The second equation is shown as follows: 
\begin{align*}
\down(Y(A)) &= \wt(\bp(Y(A))) \\ 
&= \wt(\bp(Y(A))^{(1)}) + \wt(\bp(Y(A))^{(2)}) + \wt(\bp(Y(A))^{(3)}) \\ 
&= \wt(\bp(A)^{(1)}) + \wt(\bp(Y(A))^{(2)}) + \wt(\bp(A)^{(3)}) \\ 
&= \wt(\bp(A)^{(1)}) + \wt(\bp(A)^{(2)}) + \wt(\bp(A)^{(3)}) \\ 
&= \wt(\bp(A)) \\
&= \down(A). 
\end{align*}
Since $(-1)^{\nega(\bp(Y(A))^{(2)})} = (-1)^{\nega(\bp(A)^{(2)})}$ by Proposition~\ref{prop:rank2_shellability}\,(2), 
the remaining equation is shown as follows: 
\begin{align*}
(-1)^{n(Y(A))} &= (-1)^{\nega(\bp(Y(A)))} \\ 
&= (-1)^{\nega(\bp(Y(A))^{(1)})} (-1)^{\nega(\bp(Y(A))^{(2)})} (-1)^{\nega(\bp(Y(A))^{(3)})} \\ 
&= (-1)^{\nega(\bp(A)^{(1)})} (-1)^{\nega(\bp(Y(A))^{(2)})} (-1)^{\nega(\bp(A)^{(3)})} \\ 
&= (-1)^{\nega(\bp(A)^{(1)})} (-1)^{\nega(\bp(A)^{(2)})} (-1)^{\nega(\bp(A)^{(3)})} \\ 
&= (-1)^{\nega(\bp(A))} \\
&= (-1)^{n(A)}. 
\end{align*}

Next, we construct an involution $I_{1}$ which satisfies the condition of Theorem~\ref{thm:YB-move}\,(2). 
Let $A \in \CA_{0}^{C}(w, \Gamma_{1})$. 
Then, by applying Proposition~\ref{prop:rank2_shellability} with $\Pi = \Gamma_{1}^{(2)}$, 
we see that there exists a unique $\br_{0} \in \CP(\ed(\bp(A)^{(1)}), \Gamma_{1}^{(2)}) \setminus \{ \bp(A)^{(2)} \}$ such that $\ed(\br_{0}) = \ed(\bp(A)^{(2)})$ and $\wt(\br_{0}) = \wt(\bp(A)^{(2)})$. 
We write the $\br_{0}$ as: 
\begin{equation*}
\br_{0}: \ed(\bp(A)^{(1)}) = x_{0} \xrightarrow{|\beta_{j_{1}}|} \cdots \xrightarrow{|\beta_{j_{p}}|} x_{p}.  
\end{equation*}
Since $\br_{0}$ is $\Gamma_{1}^{(2)}$-compatible, we have $t+1 \le j_{1} < \cdots < j_{p} \le t+q$. 
Set $B^{(2)} := \{j_{1}, \ldots, j_{p}\}$, and define $I_{1}(A)$ by $I_{1}(A) := A^{(1)} \sqcup B^{(2)} \sqcup A^{(3)}$; 
note that $I_{1}(A) \in \CA(w, \Gamma_{1})$.
Since $\bp(A)^{(2)} \in \CP(\ed(\bp(I_{1}(A))^{(1)}), \Gamma_{1}^{(2)})$ satisfies the condition of Proposition~\ref{prop:rank2_shellability}\,(1), with $\bp = \bp(I_{1}(A))^{(2)}$, 
we deduce that $I_{1}(A) \in \CA_{0}^{C}(w, \Gamma_{1})$,
and that $I_{1}(I_{1}(A)) = I_{1}(A^{(1)} \sqcup B^{(2)} \sqcup A^{(3)}) = A^{(1)} \sqcup A^{(2)} \sqcup A^{(3)} = A$ by the definition of $I_{1}$. 
This shows that $I_{1}$ is an involution. 
Hence it remains to show that $\ed(I_{1}(A)) = \ed(A)$, $\down(I_{1}(A)) = \down(A)$, and $(-1)^{n(I_{1}(A))} = -(-1)^{n(A)}$, 
which can be shown by the same argument as that for $Y$. 
This completes the construction of $I_{1}$. 

Finally, we show the existence of an involution $I_{2}$ on $\CA_{0}^{C}(w, \Gamma_{2})$. 
To do this, we examine the set $\CA_{0}^{C}(w, \Gamma_{2})$ in detail. 
Let $B \in \CA(w, \Gamma_{2})$. 
Then, in the same way as for $A \in \CA(w, \Gamma_{1})$, 
we see by Proposition~\ref{prop:rank2_shellability}, with $\Pi = \Gamma_{2}^{(2)}$ and $\Pi' = \Gamma_{1}^{(2)}$, 
that only one of the following occurs: 
\begin{enu}
\item[(1)$'$] there exists $\br_{0} \in \CP(\ed(\bp(B)^{(1)}), \Gamma_{2}^{(2)}) \setminus \{\bp(B)^{(2)}\}$ such that $\ed(\br_{0}) = \ed(\bp(B)^{(2)})$ and $\wt(\br_{0}) = \wt(\bp(B)^{(2)})$; 
\item[(2)$'$] there exists $\br_{0} \in \CP(\ed(\bp(B)^{(1)}), \Gamma_{1}^{(2)})$ such that $\ed(\br_{0}) = \ed(\bp(B)^{(2)})$ and $\wt(\br_{0}) = \wt(\bp(B)^{(2)})$. 
\end{enu}
For convenience of explanation, we set 
\begin{equation*}
\varphi'(B) = \begin{cases} 1 & \text{if (1)$'$ of the above holds,} \\ 2 & \text{if (2)$'$ of the above holds.} \end{cases}
\end{equation*}
We claim that 
%
%%%%%%%%%%%
% eq:CA0_2 %
%%%%%%%%%%%
%
\begin{equation}\label{eq:CA0_2}
\CA_{0}(w, \Gamma_{2}) = \{ B \in \CA(w, \Gamma_{2}) \mid \varphi'(B) = 2 \}. \\ 
\end{equation}
If this equation is shown, then the following holds:
\begin{equation*}
\CA_{0}^{C}(w, \Gamma_{2}) := \CA(w, \Gamma_{2}) \setminus \CA_{0}(w, \Gamma_{2}) = \{ B \in \CA(w, \Gamma_{2}) \mid \varphi'(B) = 1 \}.  
\end{equation*}
First, we take $B \in \CA_{0}(w, \Gamma_{2})$. 
Then, by the definition of $\CA_{0}(w, \Gamma_{2})$, there exists $A \in \CA_{0}(w, \Gamma_{1})$ such that $Y(A) = B$. 
Since $\bp(A)^{(2)} \in \CP(\ed(\bp(A)^{(1)}), \Gamma_{1}^{(2)})$ satisfies the condition of Proposition~\ref{prop:rank2_shellability}\,(2), 
with $\Pi = \Gamma_{2}^{(2)}$, $\Pi' = \Gamma_{1}^{(2)}$, and $\bp = \bp(B)^{(2)}$, 
we have $\varphi'(B) = 2$. 
Next, we take $B \in \CA(w, \Gamma_{2})$ such that $\varphi'(B) = 2$. 
Then, by the definition of $\varphi'$, there exists $\br_{0} \in \CP(\ed(\bp(B)^{(1)}), \Gamma_{1}^{(2)})$ such that $\ed(\br_{0}) = \ed(\bp(B)^{(2)})$ and $\wt(\br_{0}) = \wt(\bp(B)^{(2)})$. 
We write the $\br_{0}$ as: 
\begin{equation*}
\br_{0}: \ed(\bp(B)^{(1)}) = x_{0} \xrightarrow{|\beta_{j_{1}}|} \cdots \xrightarrow{|\beta_{j_{p}}|} x_{p}. 
\end{equation*}
Then, we have $t+1 \le j_{1} < \cdots < j_{p} \le t+q$. 
Set $A^{(2)} := \{j_{1}, \ldots, j_{p}\}$, and then $A := B^{(1)} \sqcup A^{(2)} \sqcup B^{(3)}$. 
We see that $A \in \CA_{0}(w, \Gamma_{1})$ and $Y(A) = B$, 
and hence $B \in \CA_{0}(w, \Gamma_{2})$. 
Thus, equation \eqref{eq:CA0_2} is shown. 
Hence the existence of the desired involution $I_{2}$ on $\CA_{0}^{C}(w, \Gamma_{2})$ can be shown 
by the same argument as that of the involution $I_{1}$ on $\CA_{0}^{C}(w, \Gamma_{1})$. 
This completes the proof of Theorem~\ref{thm:YB-move} for $\Delta$ not of type $G_{2}$. 

\paragraph{Part~2: $\Delta$ is of type $G_{2}$.}
We prove the theorem for $\Delta$ of type $G_{2}$. As in Part~1, 
we define $\varphi(A)$ for $A \in \CA(w, \Gamma_{1})$. 
Let $A \in \CA(w, \Gamma_{1})$. 
We set $\Pi := \Gamma_{1}^{(2)}$, $\Pi' := \Gamma_{2}^{(2)}$, and $\bp := \bp(A)^{(2)}$, $v := \ed(\bp(A)^{(1)})$. 
We first consider the cases except cases (E1)--(E4). Then, by Proposition~\ref{prop:rank2_shellability_G1}, only one of the following (1)--(3) occurs. 
\begin{enu}
\item There exists a unique $\br_{0} \in \CP(\ed(\bp(A)^{(1)}), \Gamma_{1}^{(2)}) \setminus \{\bp(A)^{(2)}\}$ such that $\ed(\br_{0}) = \ed(\bp(A)^{(2)})$ and $\wt(\br_{0}) = \wt(\bp(A)^{(2)})$. 
This $\br_{0}$ satisfies $(-1)^{\nega(\br_{0})} = -(-1)^{\nega(\bp(A)^{(2)})}$. 
Moreover, there does not exist a path $\bq \in \CP(\ed(\bp(A)^{(1)}), \Gamma_{2}^{(2)})$ such that $\ed(\bq) = \ed(\bp(A)^{(2)})$ and $\wt(\bq) = \wt(\bp(A)^{(2)})$. 
\item There exists a unique $\br_{0} \in \CP(\ed(\bp(A)^{(1)}), \Gamma_{2}^{(2)})$ such that $\ed(\br_{0}) = \ed(\bp(A)^{(2)})$ and $\wt(\br_{0}) = \wt(\bp(A)^{(2)})$. 
This $\br_{0}$ satisfies $(-1)^{\nega(\br_{0})} = (-1)^{\nega(\bp(A)^{(2)})}$. 
Moreover, there does not exist a path $\bq \in \CP(\ed(\bp(A)^{(1)}), \Gamma_{1}^{(2)}) \setminus \{\bp(A)^{(2)}\}$ such that $\ed(\bq) = \ed(\bp(A)^{(2)})$ and $\wt(\bq) = \wt(\bp(A)^{(2)})$. 
\item There exists a unique $\br_{0} \in \CP(\ed(\bp(A)^{(1)}), \Gamma_{1}^{(2)}) \setminus \{\bp(A)^{(2)}\}$ such that $\ed(\br_{0}) = \ed(\bp(A)^{(2)})$ and $\wt(\br_{0}) = \wt(\bp(A)^{(2)})$. 
This $\br_{0}$ satisfies $(-1)^{\nega(\br_{0})} = -(-1)^{\nega(\bp(A)^{(2)})}$. 
Moreover, there exist exactly two paths $\bq_{1}, \bq_{2} \in \CP(\ed(\bp(A)^{(1)}), \Gamma_{2}^{(2)})$ such that $\ed(\bq_{1}) = \ed(\bq_{2}) = \ed(\bp(A)^{(2)})$ and $\wt(\bq_{1}) = \wt(\bq_{2}) = \wt(\bp(A)^{(2)})$. 
These $\bq_{1}$, $\bq_{2}$ satisfy $(-1)^{\nega(\bq_{2})} = -(-1)^{\nega(\bq_{1})}$. 
\end{enu}
We set 
\begin{equation*}
\varphi(A) := \begin{cases} 1 & \text{if (1) or (3) of the above holds,} \\ 2 & \text{if (2) of the above holds.} \end{cases}
\end{equation*}

Next, consider the exceptional cases (E1) and (E4). Recall Proposition~\ref{prop:rank2_shellability_G2} and Remark~\ref{rem:explicit_rank2_shellability_G}. 
For $A \in \CA(w, \Gamma_{1})$, we define $\varphi(A)$ by 
\begin{equation*}
\varphi(A) := \begin{cases} 3 & \text{if } \bp(A)^{(2)} = \bp_{1}^{\E{1}}, \bp_{3}^{\E{1}}, \bp_{1}^{\E{4}}, \bp_{3}^{\E{4}}, \\ 4 & \text{if } \bp(A)^{(2)} = \bp_{2}^{\E{1}}, \bp_{2}^{\E{4}}. \end{cases}
\end{equation*}
Finally, in the exceptional cases (E2) and case (E3), we set $\varphi(A) := 5$ for $A \in \CA(w, \Gamma_{1})$. 

With $\varphi(A)$ as above, we define a subset $\CA_{0}(w, \Gamma_{1}) \subset \CA(w, \Gamma_{1})$ by 
\begin{equation*}
\CA_{0}(w, \Gamma_{1}) := \{ A \in \CA(w, \Gamma_{1}) \mid \varphi(A) = 2, 4, 5 \}. 
\end{equation*}
Also, we set 
\begin{equation*}
\CA_{0}^{C}(w, \Gamma_{1}) := \CA(w, \Gamma_{1}) \setminus \CA_{0}(w, \Gamma_{1}) = \{ A \in \CA(w, \Gamma_{1}) \mid \varphi(A) = 1, 3 \}. 
\end{equation*}

Let us define a map $Y: \CA_{0}(w, \Gamma_{1}) \rightarrow \CA(w, \Gamma_{2})$. 
Let $A \in \CA_{0}(w, \Gamma_{1})$. First, assume that $\varphi(A) = 2$. 
Take a unique $\br_{0} \in \CP(\ed(\bp(A)^{(1)}), \Gamma_{2}^{(2)})$ such that $\ed(\br_{0}) = \ed(\bp(A)^{(2)})$ and $\wt(\br_{0}) = \wt(\bp(A)^{(2)})$. 
Next, assume that $\varphi(A) = 4$. Then, $\bp(A)^{(2)} = \bp_{2}^{\E{1}}$ or $\bp(A)^{(2)} = \bp_{2}^{\E{4}}$. 
If $\bp(A)^{(2)} = \bp_{2}^{\E{1}}$, then we set $\br_{0} := \bq^{\E{1}}$; 
if $\bp(A)^{(2)} = \bp_{2}^{\E{4}}$, then we set $\br_{0} := \bq^{\E{4}}$. 
Also, assume that $\varphi(A) = 5$. Then, $\bp(A)^{(2)} = \bp^{\E{2}}$ or $\bp(A)^{(2)} = \bp^{\E{3}}$. 
If $\bp(A)^{(2)} = \bp^{\E{2}}$, then we set $\br_{0} := \bq_{2}^{\E{2}}$; 
if $\bp(A)^{(2)} = \bp^{\E{3}}$, then we set $\br_{0} := \bq_{2}^{\E{3}}$. 
We write the $\br_{0}$ as: 
\begin{equation*}
\br_{0}: \ed(\bp(A)^{(1)}) = x_{0} \xrightarrow{|\beta'_{j_{1}}|} x_{1} \xrightarrow{|\beta'_{j_{2}}|} \cdots \xrightarrow{|\beta'_{j_{p}}|} x_{p}. 
\end{equation*}
Set $B^{(2)} := \{j_{1}, \ldots, j_{p}\}$, 
and define $Y(A) := A^{(1)} \sqcup B^{(2)} \sqcup A^{(3)}$. 
By the same argument as that in Part~1, 
if we set 
\begin{equation*}
\CA_{0}(w, \Gamma_{2}) := \{ Y(A) \mid A \in \CA_{0}(w, \Gamma_{1}) \}, 
\end{equation*}
then $Y$ defines a bijection between $\CA_{0}(w, \Gamma_{1})$ and $\CA_{0}(w, \Gamma_{2})$. 
Also, we can show that $(-1)^{n(Y(A))} = (-1)^{n(A)}$, $\ed(Y(A)) = \ed(A)$, and $\down(Y(A)) = \down(A)$. 

Next, we construct an involution on $\CA_{0}^{C}(w, \Gamma_{1})$. Let $A \in \CA_{0}^{C}(w, \Gamma_{1})$. 
If $\varphi(A) = 1$, then there exists a unique $\br_{0} \in \CP(\ed(\bp(A)^{(1)}), \Gamma_{1}^{(2)}) \setminus \{ \bp(A)^{(2)} \}$ 
such that $\ed(\br_{0}) = \ed(\bp(A)^{(2)})$ and $\wt(\br_{0}) = \wt(\bp(A)^{(2)})$. 
If $\varphi(A) = 3$, then we set 
\begin{equation*}
\br_{0} := \begin{cases} \bp_{3}^{\E{1}} & \text{if } \bp(A)^{(2)} = \bp_{1}^{\E{1}}, \\ \bp_{3}^{\E{4}} & \text{if } \bp(A)^{(2)} = \bp_{1}^{\E{4}}, \\ \bp_{1}^{\E{1}} & \text{if } \bp(A)^{(2)} = \bp_{3}^{\E{1}}, \\ \bp_{1}^{\E{4}} & \text{if } \bp(A)^{(2)} = \bp_{3}^{\E{4}}. \end{cases}
\end{equation*}
We write the $\br_{0}$ as: 
\begin{equation*}
\br_{0}: \ed(\bp(A)^{(1)}) = x_{0} \xrightarrow{|\beta_{j_{1}}|} x_{1} \xrightarrow{|\beta'_{j_{2}}|} \cdots \xrightarrow{|\beta_{j_{p}}|} x_{p}. 
\end{equation*}
Set $B^{(2)} := \{j_{1}, \ldots, j_{p}\}$, 
and define $I_{1}(A) := A^{(1)} \sqcup B^{(2)} \sqcup A^{(3)}$. 
We see that $I_{1}$ defines an involution on $\CA_{0}^{C}(w, \Gamma_{1})$ such that 
$(-1)^{n(I_{1}(A))} = -(-1)^{n(A)}$, $\ed(I_{1}(A)) = \ed(A)$, and $\down(I_{1}(A)) = \down(A)$. 

Finally, we construct an involution $I_{2}$. 
Let $B \in \CA(w, \Gamma_{2})$. Set $\Pi := \Gamma_{2}^{(2)}$, $\Pi' := \Gamma_{1}^{(2)}$, and $\bp := \bp(B)^{(2)}$, $v := \ed(\bp(B)^{(1)})$. 
Let us consider the cases except cases (E1)--(E4). Then, by Proposition~\ref{prop:rank2_shellability_G1}, only one of the following (1)$'$--(3)$'$ occurs. 
\begin{enu}
\item[(1)$'$] There exists a unique $\br_{0} \in \CP(\ed(\bp(B)^{(1)}), \Gamma_{2}^{(2)}) \setminus \{\bp(B)^{(2)}\}$ such that $\ed(\br_{0}) = \ed(\bp(B)^{(2)})$ and $\wt(\br_{0}) = \wt(\bp(B)^{(2)})$. 
This $\br_{0}$ satisfies $(-1)^{\nega(\br_{0})} = -(-1)^{\nega(\bp(B)^{(2)})}$. 
Moreover, there does not exist a path $\bq \in \CP(\ed(\bp(B)^{(1)}), \Gamma_{1}^{(2)})$ such that $\ed(\bq) = \ed(\bp(B)^{(2)})$ and $\wt(\bq) = \wt(\bp(B)^{(2)})$. 
\item[(2)$'$] There exists a unique $\br_{0} \in \CP(\ed(\bp(B)^{(1)}), \Gamma_{1}^{(2)})$ such that $\ed(\br_{0}) = \ed(\bp(B)^{(2)})$ and $\wt(\br_{0}) = \wt(\bp(B)^{(2)})$. 
This $\br_{0}$ satisfies $(-1)^{\nega(\br_{0})} = (-1)^{\nega(\bp(B)^{(2)})}$. 
Moreover, there does not exist a path $\bq \in \CP(\ed(\bp(B)^{(1)}), \Gamma_{2}^{(2)}) \setminus \{\bp(B)^{(2)}\}$ such that $\ed(\bq) = \ed(\bp(B)^{(2)})$ and $\wt(\bq) = \wt(\bp(B)^{(2)})$. 
\item[(3)$'$] There exists a unique $\br_{0} \in \CP(\ed(\bp(B)^{(1)}), \Gamma_{2}^{(2)}) \setminus \{\bp(B)^{(2)}\}$ such that $\ed(\br_{0}) = \ed(\bp(B)^{(2)})$ and $\wt(\br_{0}) = \wt(\bp(B)^{(2)})$. 
This $\br_{0}$ satisfies $(-1)^{\nega(\br_{0})} = -(-1)^{\nega(\bp(B)^{(2)})}$. 
Moreover, there exist exactly two paths $\bq_{1}, \bq_{2} \in \CP(\ed(\bp(B)^{(1)}), \Gamma_{1}^{(2)})$ such that $\ed(\bq_{1}) = \ed(\bq_{2}) = \ed(\bp(B)^{(2)})$ and $\wt(\bq_{1}) = \wt(\bq_{2}) = \wt(\bp(B)^{(2)})$. 
These $\bq_{1}$, $\bq_{2}$ satisfy $(-1)^{\nega(\bq_{2})} = -(-1)^{\nega(\bq_{1})}$. 
\end{enu}
We define $\varphi'(B)$ by 
\begin{equation*}
\varphi'(B) := \begin{cases} 1 & \text{if (1)$'$ or (3)$'$ of the above holds,} \\ 2 & \text{if (2)$'$ of the above holds.} \end{cases}
\end{equation*}

For the exceptional cases (E1) and (E4), we set 
\begin{equation*}
\varphi'(B) := \begin{cases} 3 & \text{if } \bp(B)^{(2)} = \bp_{1}^{\E{1}}, \bp_{3}^{\E{1}}, \bp_{1}^{\E{4}}, \bp_{3}^{\E{4}}, \\ 4 & \text{if } \bp(B)^{(2)} = \bp_{2}^{\E{1}}, \bp_{2}^{\E{4}}; \end{cases}
\end{equation*}
for the exceptional cases (E2) and (E3), we set $\varphi'(B) := 5$. 

Now, by the same argument as that in Part~1, we obtain 
\begin{align*}
\CA_{0}(w, \Gamma_{2}) &= \{ B \in \CA(w, \Gamma_{2}) \mid \varphi'(B) = 2, 4, 5 \}, \\ 
\CA_{0}^{C}(w, \Gamma_{2}) &= \{ B \in \CA(w, \Gamma_{2}) \mid \varphi'(B) = 1, 3 \}. 
\end{align*}
Hence an involution $I_{2}$ on $\CA_{0}^{C}(w, \Gamma_{2})$ can be defined in the same way as $I_{1}$ on $\CA_{0}(w, \Gamma_{1})$. 
This completes the proof of the theorem. 
\end{proof}

%=========================%
% START SUBSECTION 0407 %
%=========================%

\subsection{Proof of Theorem~\ref{thm:wt_height_preserving}.}
We will prove that the maps $Y$, $I_{1}$ and $I_{2}$ preserve weights and heights. 

We need additional notation. 
Let $\Psi = (\psi_{1}, \ldots, \psi_{p})$ be a sequence of roots $\psi_{1}, \ldots, \psi_{p} \in \Delta$, 
$\bk = (k_{1}, \ldots, k_{p})$ a sequence of integers $k_{1}, \ldots, k_{p} \in \BZ$, 
and $w \in W$. 
For a subset $J = \{j_{1} < \cdots < j_{a}\} \subset \{ 1, \ldots, p \}$ such that 
\begin{equation*}
w \xrightarrow{|\psi_{j_{1}}|} ws_{|\psi_{j_{1}}|} \xrightarrow{|\psi_{j_{2}}|} \cdots \xrightarrow{|\psi_{j_{a}}|} ws_{|\psi_{j_{1}}|} \cdots s_{|\psi_{j_{a}}|} 
\end{equation*}
is a directed path in $\QBG(W)$ 
(note that if $\Psi$ is a $\lambda$-chain for some $\lambda \in P$, then $J$ is a $w$-admissible subset), 
we define $\height_{\bk, \Psi}(w, J)$ by 
\begin{equation*}
\height_{\bk, \Psi}(w, J) := \sum_{j \in J^{-}} \sgn(\psi_{j}) k_{j}, 
\end{equation*}
where 
\begin{equation*}
J^{-} = \left\{ j_{i} \in J \ \middle| \ ws_{|\psi_{j_{1}}|} \cdots s_{|\psi_{j_{i-1}}|} \xrightarrow{|\psi_{j_{i}}|} ws_{|\psi_{j_{1}}|} \cdots s_{|\psi_{j_{i}}|} \text{ is a quantum edge} \right\}. 
\end{equation*}
Also, we generalize the definition of down statistics: 
\begin{equation*}
\down_{\Psi}(w, J) := \sum_{j \in J^{-}} |\psi_{j}|^{\vee}. 
\end{equation*}
Note that if $\Psi = \Gamma_{1}$, $\bk = (\wti{l_{1}}, \ldots, \wti{l_{r}})$, $w \in W$, and $J = A \in \CA(w, \Gamma_{1})$, then we have 
\begin{equation*}
\down_{\Psi}(w, A) = \down(A), \quad \height_{\bk, \Psi}(w, A) = \height(A); 
\end{equation*}
if $\Psi = \Gamma_{2}$, $\bk = (\wti{l'_{1}}, \ldots, \wti{l'_{r}}) = (\wti{l_{1}}, \ldots, \wti{l_{t}}, \wti{l_{t+q}}, \ldots, \wti{l_{t+1}}, \wti{l_{t+q+1}}, \ldots, \wti{l_{r}})$, $w \in W$, and $J = B \in \CA(w, \Gamma_{2})$, then we have 
\begin{equation*}
\down_{\Psi}(w, B) = \down(B), \quad \height_{\bk, \Psi}(w, B) = \height(B). 
\end{equation*}
In addition, for $A \in \CA(w, \Gamma_{1})$, it follows that 
\begin{equation*}
\down(A) = \down_{\Gamma_{1}^{(1)}}(w, A^{(1)}) + \down_{\Gamma_{1}^{(2)}}(\ed(\bp(A)^{(1)}), A^{(2)}) + \down_{\Gamma_{1}^{(3)}}(\ed(\bp(A)^{(2)}), A^{(3)}), \\ 
\end{equation*}
and that 
%
%%%%%%%%%%%%%
% eq:height_A %
%%%%%%%%%%%%%
%
\begin{equation}\label{eq:height_A}
\begin{split}
\height(A) &= \height_{(\wti{l_{1}}, \ldots, \wti{l_{t}}), \Gamma_{1}^{(1)}}(w, A^{(1)}) + \height_{(\wti{l_{t+1}}, \ldots, \wti{l_{t+q}}), \Gamma_{1}^{(2)}}(\ed(\bp(A)^{(1)}), A^{(2)}) \\ 
& \hspace*{4mm} + \height_{(\wti{l_{t+q+1}}, \ldots, \wti{l_{r}}), \Gamma_{1}^{(3)}}(\ed(\bp(A)^{(2)}), A^{(3)}); 
\end{split}
\end{equation}
for $B \in \CA(w, \Gamma_{2})$, it follows that 
\begin{equation*}
\down(B) = \down_{\Gamma_{2}^{(1)}}(w, B^{(1)}) + \down_{\Gamma_{2}^{(2)}}(\ed(\bp(B)^{(1)}), B^{(2)}) + \down_{\Gamma_{2}^{(3)}}(\ed(\bp(B)^{(2)}), B^{(3)}), \\ 
\end{equation*}
and that 
%
%%%%%%%%%%%%%
% eq:height_B %
%%%%%%%%%%%%%
%
\begin{align}
\begin{split}
\height(B) &= \height_{(\wti{l'_{1}}, \ldots, \wti{l'_{t}}), \Gamma_{2}^{(1)}}(w, B^{(1)}) + \height_{(\wti{l'_{t+1}}, \ldots, \wti{l'_{t+q}}), \Gamma_{2}^{(2)}}(\ed(\bp(B)^{(1)}), B^{(2)}) \\ 
& \hspace*{4mm} + \height_{(\wti{l'_{t+q+1}}, \ldots, \wti{l'_{r}}), \Gamma_{2}^{(3)}}(\ed(\bp(B)^{(2)}), B^{(3)}). 
\end{split} \nonumber \\ 
\begin{split}
&= \height_{(\wti{l_{1}}, \ldots, \wti{l_{t}}), \Gamma_{2}^{(1)}}(w, B^{(1)}) + \height_{(\wti{l_{t+q}}, \ldots, \wti{l_{t+1}}), \Gamma_{2}^{(2)}}(\ed(\bp(B)^{(1)}), B^{(2)}) \\ 
& \hspace*{4mm} + \height_{(\wti{l_{t+q+1}}, \ldots, \wti{l_{r}}), \Gamma_{2}^{(3)}}(\ed(\bp(B)^{(2)}), B^{(3)}). 
\end{split} \label{eq:height_B}
\end{align}

Next, we consider weights. For the above $J$, we set 
\begin{equation*}
\wh{r}_{\bk, \Psi}(J) := s_{\psi_{j_{1}}, -k_{j_{1}}} \cdots s_{\psi_{j_{a}}, -k_{j_{a}}}. 
\end{equation*}
Then, for $A \in \CA(w, \Gamma_{1})$, we have 
%
%%%%%%%%%%%%%
% eq:weight_A %
%%%%%%%%%%%%%
%
\begin{align}
\wt(A) &= -\wh{r}_{(l_{1}, \ldots, l_{r}), \Gamma_{1}}(A) (-\lambda). \nonumber \\ 
&= - \wh{r}_{(l_{1}, \ldots, l_{t}), \Gamma_{1}^{(1)}}(A^{(1)}) \wh{r}_{(l_{t+1}, \ldots, l_{t+q}), \Gamma_{1}^{(2)}}(A^{(2)}) \wh{r}_{(l_{t+q+1}, \ldots, l_{r}), \Gamma_{1}^{(3)}}(A^{(3)}) (-\lambda) \label{eq:weight_A};
\end{align}
for $B \in \CA(w, \Gamma_{2})$, we have 
%
%%%%%%%%%%%%%
% eq:weight_B %
%%%%%%%%%%%%%
%
\begin{align}
\wt(B) &= -\wh{r}_{(l'_{1}, \ldots, l'_{r}), \Gamma_{2}}(B) (-\lambda) \nonumber \\ 
&= -\wh{r}_{(l_{1}, \ldots, l_{t}, l_{t+q}, \ldots, l_{t+1}, l_{t+q+1}, \ldots, l_{r}), \Gamma_{2}}(B) (-\lambda) \nonumber \\ 
&= -\wh{r}_{(l_{1}, \ldots, l_{t}), \Gamma_{2}^{(1)}}(B^{(1)}) \wh{r}_{(l_{t+q}, \ldots, l_{t+1}), \Gamma_{2}^{(2)}}(B^{(2)}) \wh{r}_{(l_{t+q+1}, \ldots, l_{r}), \Gamma_{2}^{(3)}}(B^{(3)}) (-\lambda). \label{eq:weight_B} 
\end{align}

\begin{proof}[Proof of Theorem~\ref{thm:wt_height_preserving}]
First, we consider heights. 
For $A \in \CA(w, \Gamma_{1})$, we see that 
%
%%%%%%%%%%%%%%%
% eq:height_LHS %
%%%%%%%%%%%%%%%
%
\begin{align}
& \height_{(\wti{l_{t+1}}, \ldots, \wti{l_{t+q}}), \Gamma_{1}^{(2)}}(\ed(\bp(A)^{(1)}), A^{(2)}) \nonumber \\ 
&= \sum_{j \in (A^{(2)})^{-}} \sgn(\beta_{j})(\pair{\lambda}{\beta_{j}^{\vee}} - l_{j}) \nonumber \\ 
&= \sum_{j \in (A^{(2)})^{-}} \sgn(\beta_{j})\pair{\lambda}{\beta_{j}^{\vee}} - \height_{(l_{t+1}, \ldots, l_{t+q}), \Gamma_{1}^{(2)}}(\ed(\bp(A)^{(2)}), A^{(2)}) \nonumber \\ 
&= \lrpair{\lambda}{\sum_{j \in (A^{(2)})^{-}} \sgn(\beta_{j})\beta_{j}^{\vee}} - \height_{(l_{t+1}, \ldots, l_{t+q}), \Gamma_{1}^{(2)}}(\ed(\bp(A)^{(2)}), A^{(2)}) \nonumber \\ 
&= \pair{\lambda}{\down_{\Gamma_{1}^{(2)}}(\ed(\bp(A)^{(1)}), A^{(2)})} - \height_{(l_{t+1}, \ldots, l_{t+q}), \Gamma_{1}^{(2)}}(\ed(\bp(A)^{(2)}), A^{(2)}). \label{eq:height_LHS}
\end{align}
Let us assume that $A \in \CA_{0}(w, \Gamma_{1})$, and set $B := Y(A)$. Then we have 
%
%%%%%%%%%%%%%%%
% eq:height_RHS %
%%%%%%%%%%%%%%%
%
\begin{align}
& \height_{(\wti{l_{t+q}}, \ldots, \wti{l_{t+1}}), \Gamma_{2}^{(2)}}(\ed(\bp(B)^{(1)}), B^{(2)}) \nonumber \\ 
&= \pair{\lambda}{\down_{\Gamma_{2}^{(2)}}(\ed(\bp(B)^{(1)}), B^{(2)})} - \height_{(l_{t+q}, \ldots, l_{t+1}), \Gamma_{2}^{(2)}}(\ed(\bp(B)^{(2)}), B^{(2)}). \label{eq:height_RHS}
\end{align}
Here, since $\down(Y(A)) = \down(A)$, it follows that 
\begin{equation*}
\down_{\Gamma_{2}^{(2)}}(\ed(\bp(B)^{(1)}), B^{(2)}) = \down_{\Gamma_{1}^{(2)}}(\ed(\bp(A)^{(1)}), A^{(2)}). 
\end{equation*}
Also, by \cite[Lemma~3.5]{LL2}, we know that 
\begin{equation*}
\bigcap_{k = t+1}^{t+q} H_{\beta_{k}, -l_{k}} \not= \emptyset. 
\end{equation*}
Therefore, by \cite[Lemma~44]{LNS}, we have 
\begin{equation*}
\height_{(l_{t+q}, \ldots, l_{t+1}), \Gamma_{2}^{(2)}}(\ed(\bp(B)^{(2)}), B^{(2)}) = \height_{(l_{t+1}, \ldots, l_{t+q}), \Gamma_{1}^{(2)}}(\ed(\bp(A)^{(2)}), A^{(2)}),
\end{equation*}
and hence by \eqref{eq:height_LHS} and \eqref{eq:height_RHS}, we obtain
\begin{equation*}
\height_{(\wti{l_{t+q}}, \ldots, \wti{l_{t+1}}), \Gamma_{2}^{(2)}}(\ed(\bp(B)^{(1)}), B^{(2)}) = \height_{(\wti{l_{t+1}}, \ldots, \wti{l_{t+q}}), \Gamma_{1}^{(2)}}(\ed(\bp(A)^{(1)}), A^{(2)}). 
\end{equation*}
Now, by \eqref{eq:height_A} and \eqref{eq:height_B}, we deduce that $\height(B) = \height(A)$, as desired. 

Assume that $A \in \CA_{0}^{C}(w, \Gamma_{1})$, and set $B := I_{1}(A)$. As in \eqref{eq:height_LHS}, we have 
\begin{align*}
& \height_{(\wti{l_{t+1}}, \ldots, \wti{l_{t+q}}), \Gamma_{1}^{(2)}}(\ed(\bp(B)^{(1)}), B^{(2)}) \\ 
&= \pair{\lambda}{\down_{\Gamma_{1}^{(2)}}(\ed(\bp(B)^{(1)}), B^{(2)})} - \height_{(l_{t+1}, \ldots, l_{t+q}), \Gamma_{1}^{(2)}}(\ed(\bp(B)^{(2)}), B^{(2)}). 
\end{align*}
Again, by the definition of $I_{1}$ and \cite[Lemma~44]{LNS}, we have 
\begin{align*}
& \height_{(\wti{l_{t+1}}, \ldots, \wti{l_{t+q}}), \Gamma_{1}^{(2)}}(\ed(\bp(B)^{(1)}), B^{(2)}) \\ 
&= \pair{\lambda}{\down_{\Gamma_{1}^{(2)}}(\ed(\bp(B)^{(1)}), B^{(2)})} - \height_{(l_{t+1}, \ldots, l_{t+q}), \Gamma_{1}^{(2)}}(\ed(\bp(B)^{(2)}), B^{(2)}) \\ 
&= \pair{\lambda}{\down_{\Gamma_{1}^{(2)}}(\ed(\bp(A)^{(1)}), A^{(2)})} - \height_{(l_{t+1}, \ldots, l_{t+q}), \Gamma_{1}^{(2)}}(\ed(\bp(A)^{(2)}), A^{(2)}) \\ 
&= \height_{(\wti{l_{t+1}}, \ldots, \wti{l_{t+q}}), \Gamma_{1}^{(2)}}(\ed(\bp(A)^{(1)}), A^{(2)}), 
\end{align*}
and hence obtain $\height(B) = \height(A)$, as desired. 
Now, the assertion for $I_{2}$ is shown by the same argument as for $I_{1}$. 

It remains to consider weights. Again, by \cite[Lemma~3.5]{LL2}, we can take $\mu \in \Fh_{\BR}^{\ast}$ such that 
\begin{equation*}
\mu \in \bigcap_{k = t+1}^{t+q} H_{\beta_{k}, -l_{k}} \not= \emptyset. 
\end{equation*}
Note that $\pair{\mu}{\beta_{k}^{\vee}} = -l_{k}$ for $k = t+1, \ldots, t+q$. 

Let $A \in \CA_{0}(w, \Gamma_{1})$, and set $B := Y(A)$. 
Recall that $\ed(\bp(B)^{(2)}) = \ed(\bp(A)^{(2)})$.
For simplicity of notation, we set $v := \ed(\bp(A)^{(1)}) = \ed(\bp(B)^{(1)})$. 
For $\nu \in \Fh_{\BR}^{\ast}$, we denote by $t_{\nu}$ the translation by $\nu$, i.e., 
we define $t_{\nu}: \Fh_{\BR}^{\ast} \rightarrow \Fh_{\BR}^{\ast}$ by $t_{\nu}(\xi) := \xi + \nu$. 
We see that for $\nu \in \Fh_{\BR}^{\ast}$ and $\gamma \in \Delta^{+}$, 
$s_{\gamma} t_{\nu} = t_{s_{\gamma}(\nu)} s_{\gamma}$, 
and for $\nu_{1}, \nu_{2} \in \Fh_{\BR}^{\ast}$, $t_{\nu_{1}}t_{\nu_{2}} = t_{\nu_{1}+\nu_{2}}$. 
Also, for $\gamma \in \Delta$ and $k \in \BZ$, we have $t_{k\gamma} s_{|\gamma|} = s_{\gamma, k}$. 
If we write $A^{(2)} = \{j_{1}, \ldots, j_{a}\}$, then we have 
%
%%%%%%%%%%%%%%%%%%%
% eq:wtA_calculation %
%%%%%%%%%%%%%%%%%%%
%
\begin{align}
t_{\mu} v^{-1} \ed(\bp(A^{(2)})) t_{-\mu} &= t_{\mu} v^{-1} (v s_{|\beta_{j_{1}}|} \cdots s_{|\beta_{j_{a}}|}) t_{-\mu} \nonumber \\ 
&= (t_{\mu} s_{|\beta_{j_{1}}|} t_{-\mu}) \cdots (t_{\mu} s_{|\beta_{j_{a}}|} t_{-\mu}) \nonumber \\ 
&= (t_{\mu} t_{s_{|\beta_{j_{1}}|}(-\mu)}s_{|\beta_{j_{1}}|}) \cdots (t_{\mu} t_{s_{|\beta_{j_{a}}|}(-\mu)}s_{|\beta_{j_{a}}|}) \nonumber \\ 
&= (t_{\pair{\mu}{|\beta_{j_{1}}|^{\vee}}|\beta_{j_{1}}|} s_{|\beta_{j_{1}}|}) \cdots (t_{\pair{\mu}{|\beta_{j_{a}}|^{\vee}}|\beta_{j_{a}}|} s_{|\beta_{j_{a}}|}) \nonumber \\ 
&= (t_{\pair{\mu}{\beta_{j_{1}}^{\vee}}\beta_{j_{1}}} s_{|\beta_{j_{1}}|}) \cdots (t_{\pair{\mu}{\beta_{j_{a}}^{\vee}}\beta_{j_{a}}} s_{|\beta_{j_{a}}|}) \nonumber \\ 
&= (t_{-l_{j_{1}}\beta_{j_{1}}} s_{|\beta_{j_{1}}|}) \cdots (t_{-l_{j_{a}}\beta_{j_{a}}} s_{|\beta_{j_{a}}|}) \nonumber \\ 
&= s_{\beta_{j_{1}}, -l_{j_{1}}} \cdots s_{\beta_{j_{a}}, -l_{j_{a}}} \nonumber \\ 
&= \wh{r}_{(l_{t+1}, \ldots, l_{t+q}), \Gamma_{1}^{(2)}}(A^{(2)}). \label{eq:wtA_calculation}
\end{align}
By the same calculation, we have 
\begin{equation*}
t_{\mu} v^{-1} \ed(\bp(B^{(2)})) t_{-\mu} = \wh{r}_{(l_{t+q}, \ldots, l_{t+1}), \Gamma_{2}^{(2)}}(B^{(2)}). 
\end{equation*}
Since $\ed(\bp(B)^{(2)}) = \ed(\bp(A)^{(2)})$, it follows that 
\begin{equation*}
\wh{r}_{(l_{t+q}, \ldots, l_{t+1}), \Gamma_{2}^{(2)}}(B^{(2)}) = \wh{r}_{(l_{t+1}, \ldots, l_{t+q}), \Gamma_{1}^{(2)}}(A^{(2)}). 
\end{equation*}
Hence, by \eqref{eq:weight_A} and \eqref{eq:weight_B}, we obtain $\wt(B) = \wt(A)$, as desired. 

Next, assume that $A \in \CA_{0}^{C}(w, \Gamma_{1})$, and set $B := I_{1}(A)$. 
By the same calculation as for \eqref{eq:wtA_calculation}, we have 
\begin{equation*}
t_{\mu} v^{-1} \ed(\bp(B^{(2)})) t_{-\mu} = \wh{r}_{(l_{t+1}, \ldots, l_{t+q}), \Gamma_{1}^{(2)}}(B^{(2)}). 
\end{equation*}
Hence, from the equality $\ed(\bp(B)^{(2)}) = \ed(\bp(A)^{(2)})$, we deduce that 
\begin{equation*}
\wh{r}_{(l_{t+1}, \ldots, l_{t+q}), \Gamma_{1}^{(2)}}(B^{(2)}) = \wh{r}_{(l_{t+1}, \ldots, l_{t+q}), \Gamma_{1}^{(2)}}(A^{(2)}). 
\end{equation*}
Therefore, we conclude by \eqref{eq:weight_A} and \eqref{eq:weight_B} that $\wt(B) = \wt(A)$, as desired. 

The assertion for $I_{2}$ is shown by the same argument as for $I_{1}$. 
This completes the proof of the theorem. 
\end{proof}

%===================%
% START SECTION 05 %
%===================%

%%%%%%%%%%%%%%%%%%%%%%%
% sec:generating_function %
%%%%%%%%%%%%%%%%%%%%%%%

\section{Generating functions of certain statistics.}\label{sec:generating_function}

We consider the generating function of the statistics associated with the quantum alcove model. 
We describe the relationship between the generating functions associated to two alcove paths which are related by the procedures (YB) and (D). We also investigate the composition of generating functions. 
As an application, we derive an identity of ``Chevalley type'' for the graded characters of Demazure submodules of (level-zero) extremal weight modules over a quantum affine algebra. 

%=========================%
% START SUBSECTION 0501 %
%=========================%

%%%%%%%%%%%%%%%%%%%%%%%%%%
% subsec:generating_function %
%%%%%%%%%%%%%%%%%%%%%%%%%%

\subsection{Generating functions.}\label{subsec:generating_function}

Take an indeterminate $q$, and consider the ring $R := \BZ[q, q^{-1}]$ of Laurent polynomials in $q$. Recall that an element of the affine Weyl group $W_{\af}$ can be written $x = w t_{\xi}$,  with $w$ in the finite Weyl group $W$ 
and $\xi$ in the coroot lattice $Q^{\vee}$.

\begin{dfn}
For each $\lambda$-chain $\Gamma$ and $x = w t_{\xi} \in W_{\af}$, we define a \emph{generating function} $\bG_{\Gamma}(x) \in R[P][W_{\af}]$ 
associated to the set $\CA(w, \Gamma)$ of $w$-admissible subsets by 
\begin{equation}\label{def:gf-adm}
\bG_{\Gamma}(x) := \sum_{A \in \CA(w, \Gamma)} (-1)^{n(A)} q^{-\height(A)-\pair{\lambda}{\xi}} e^{\wt(A)} \ed(A)t_{\xi+\down(A)}. 
\end{equation}
We also think of $\bG_{\Gamma}$ as a linear function on $R[P][W_{\af}]$.
\end{dfn}

Let $\lambda \in P$, and take $\lambda$-chains $\Gamma_1$, $\Gamma_2$. 
We consider the relation of the two generating functions $\bG_{\Gamma_1}(x)$ and $\bG_{\Gamma_2}(x)$ for $x= w t_{\xi} \in W_{\af}$. 

First, we consider the case in which $\Gamma_{2}$ is obtained from $\Gamma_{1}$ by the procedure (YB). 
As a corollary of Theorems~\ref{thm:YB-move} and \ref{thm:wt_height_preserving}, we obtain the equality between the two generating functions. 
%
%%%%%%%%%%%%%%%%%%%%%%%%%%%
% prop:generating_function_YB %
%%%%%%%%%%%%%%%%%%%%%%%%%%%
%
\begin{prop}\label{prop:generating_function_YB}
Assume that $\Gamma_{2}$ is obtained from $\Gamma_{1}$ by (YB). 
Then we have $\bG_{\Gamma_1}(x) = \bG_{\Gamma_2}(x)$. 
\end{prop}

\begin{proof}
As in Theorem~\ref{thm:YB-move}, we take subsets $\CA_{0}(w, \Gamma_{1}), \CA_{0}^{C}(w, \Gamma_{1})$ of $\CA(w, \Gamma_{1})$. 
Also, we take subsets $\CA_{0}(w, \Gamma_{2}), \CA_{0}^{C}(w, \Gamma_{2})$ of $\CA(w, \Gamma_{2})$. 
Then we have the maps $Y$, $I_{1}$, $I_{2}$ as in Theorem~\ref{thm:YB-move}. 
Note that by Theorem~\ref{thm:wt_height_preserving}, $Y$, $I_{1}$, and $I_{2}$ preserve 
weights and heights. 

Since $I_{1}$ is a sign-reversing involution which preserves weights, heights, and down statistics, we have 
\begin{equation*}
\sum_{A \in \CA_{0}^{C}(w, \Gamma_{1})} (-1)^{n(A)} q^{-\height(A)-\pair{\lambda}{\xi}} e^{\wt(A)} \ed(A)t_{\xi+\down(A)} = 0, 
\end{equation*}
and hence 
\begin{equation*}
\bG_{\Gamma_1}(x) = \sum_{A \in \CA_{0}(w, \Gamma_{1})}(-1)^{n(A)} q^{-\height(A)-\pair{\lambda}{\xi}} e^{\wt(A)}  \ed(A)t_{\xi+\down(A)}. 
\end{equation*}
We derive the similar result for $\bG_{\Gamma_2}(x)$ via the sign-reversing involution $I_{2}$. 
Using the map $Y$ given by a generalization of quantum Yang-Baxter moves, we deduce that 
\begin{align*}
\bG_{\Gamma_1}(x) &= \sum_{A \in \CA_{0}(w, \Gamma_{1})}(-1)^{n(Y(A))} q^{-\height(Y(A))-\pair{\lambda}{\xi}} e^{\wt(Y(A))}  \ed(Y(A))t_{\xi+\down(Y(A))} \\ 
&= \bG_{\Gamma_2}(x). 
\end{align*}
This concludes the proof.
\end{proof}

Next, we consider the case in which $\Gamma_{2}$ is obtained from $\Gamma_{1}$ by the procedure (D). 
%
%%%%%%%%%%%%%%%%%%%%%%%%%%
% prop:generating_function_D %
%%%%%%%%%%%%%%%%%%%%%%%%%%
%
\begin{prop}\label{prop:generating_function_D}
Assume that $\Gamma_{2}$ is obtained from $\Gamma_{1}$ by the procedure (D), 
which deletes the segment $(\pm \beta, \mp \beta)$ of $\Gamma_{1}$, where $\beta$ is not a simple root. 
Then we have $\bG_{\Gamma_1}(x) = \bG_{\Gamma_2}(x)$. 
\end{prop}

\begin{proof}
We write $\Gamma_{1} = (\beta_{1}, \ldots, \beta_{r})$. 
By the assumption, there exists $u \in \{1, \ldots, r-2\}$ such that 
\begin{itemize}
\item $\beta_{u+2} = -\beta_{u+1}$, 
\item $\beta_{u+1}$ and $\beta_{u+2}$ are not simple roots, and 
\item $\Gamma_{2} = (\beta_{1}, \ldots, \beta_{u}, \beta_{u+3}, \ldots, \beta_{r})$. 
\end{itemize}
Set $\beta := |\beta_{u+1}| = |\beta_{u+2}|$. Since $\beta$ is not a simple root, 
there does not exist any path of the form $v \xrightarrow{|\beta_{u+1}| = \beta} v' \xrightarrow{|\beta_{u+2}| = \beta} v'' = v$. 
Hence, for $A \in \CA(w, \Gamma_{1})$, we have $A \cap \{u+1, u+2\} \not= \{u+1, u+2\}$. 
We define a subset $\CA_{\emptyset}(w, \Gamma_{1}) \subset \CA(w, \Gamma_{1})$ by 
\begin{equation*}
\CA_{\emptyset}(w, \Gamma_{1}) := \{ A \in \CA(w, \Gamma_{1}) \mid A \cap \{u+1, u+2\} = \emptyset \}. 
\end{equation*}
Also, we define a subset $\CA_{\emptyset}^{C}(w, \Gamma_{1}) \subset \CA(w, \Gamma_{1})$ by 
\begin{align*}
\CA_{\emptyset}^{C}(w, \Gamma_{1}) &:= \CA(w, \Gamma_{1}) \setminus \CA_{\emptyset}(w, \Gamma_{1}) \\ 
&= \{ A \in \CA(w, \Gamma_{1}) \mid A \cap \{u+1, u+2\} = \{u+1\}, \{u+2\} \}. 
\end{align*}
We define a map $I_{D}: \CA_{\emptyset}^{C}(w, \Gamma_{1}) \rightarrow \CA_{\emptyset}^{C}(w, \Gamma_{1})$ as follows. 
If $A \in \CA_{\emptyset}^{C}(w, \Gamma_{1})$ satisfies $A \cap \{u+1, u+2\} = \{u+1\}$, then we set 
\begin{equation*}
I_{D}(A) := (A \cap \{1, \ldots, u\}) \sqcup \{u+2\} \sqcup (A \cap \{u+3, \ldots, r\}); 
\end{equation*}
if $A \in \CA_{\emptyset}^{C}(w, \Gamma_{1})$ satisfies $A \cap \{u+1, u+2\} = \{u+2\}$, then we set 
\begin{equation*}
I_{D}(A) := (A \cap \{1, \ldots, u\}) \sqcup \{u+1\} \sqcup (A \cap \{u+3, \ldots, r\}). 
\end{equation*}
We see that for all $A \in \CA_{\emptyset}^{C}(w, \Gamma_{1})$, we have $I_{D}(A) \in \CA_{\emptyset}^{C}(w, \Gamma_{1})$. 
Also, we see that $I_{D}(I_{D}(A)) = A$. 
Hence $I_{D}$ defines an involution on $\CA_{\emptyset}^{C}(w, \Gamma_{1})$, and it is easy to see that it preserves the statistics $n(\cdot)$, $\down(\cdot)$, $\ed(\cdot)$, $\height(\cdot)$, and $\wt(\cdot)$. 
Therefore, we obtain 
\begin{equation*}
\sum_{A \in \CA_{\emptyset}^{C}(w, \Gamma_{1})} (-1)^{n(A)} q^{-\height(A)-\pair{\lambda}{\xi}} e^{\wt(A)} \ed(A)t_{\xi+\down(A)} = 0. 
\end{equation*}

Now, we define a bijection $D: \CA_{\emptyset}(w, \Gamma_{1}) \rightarrow \CA(w, \Gamma_{2})$ by 
\begin{equation*}
D(A) := (A \cap \{1, \ldots, u\}) \sqcup \{ j-2 \mid j \in A \cap \{u+3, \ldots, r\}\}. 
\end{equation*}
It is again easy to see that this bijection preserves all the statistics. 
Therefore, we deduce that 
\begin{align*}
\bG_{\Gamma_1}(x) &= \sum_{A \in \CA(w, \Gamma_{1})} (-1)^{n(A)} q^{-\height(A)-\pair{\lambda}{\xi}} e^{\wt(A)} \ed(A)t_{\xi+\down(A)} \\ 
&= \sum_{A \in \CA_{\emptyset}(w, \Gamma_{1})} (-1)^{n(A)} q^{-\height(A)-\pair{\lambda}{\xi}} e^{\wt(A)} \ed(A)t_{\xi+\down(A)} \\ 
&= \sum_{A \in \CA_{\emptyset}(w, \Gamma_{1})} (-1)^{n(D(A))} q^{-\height(D(A))-\pair{\lambda}{\xi}} e^{\wt(D(A))} \ed(D(A))t_{\xi+\down(D(A))} \\ 
&= \bG_{\Gamma_2}(x), 
\end{align*}
as desired. 
This concludes the proof.
\end{proof}

\begin{rem} In the setup of of Proposition~\ref{prop:generating_function_D}, assume that $\beta$ is a simple root such that $\pm\beta$ appears in positions $u+1$ and $u+2$ in $\Gamma_1$. Then the equality $\bG_{\Gamma_1}(x) = \bG_{\Gamma_2}(x)$ does not hold. This is because there exists a directed path of the form $v \xrightarrow{\beta} v' \xrightarrow{\beta} v$ for all $v \in W$, 
in contrast to the case that $\beta$ is not a simple root. In fact, in $\CA(w,\Gamma_1)$ we can pair each $A$ for which $A\cap\{u+1,u+2\}=\emptyset$ with $A':=A\sqcup\{u+1,u+2\}$. Let $h\in\BZ$ be the contribution of (one of the) positions $u+1$, $u+2$ to $\height(A')$; note that this is independent of $A$. By using the above pairing, as well as the map $D$ and the cancellations given by the involution $I_D$ in the proof of  Proposition~\ref{prop:generating_function_D}, we derive
\begin{equation*}
\bG_{\Gamma_1}(x) = \bG_{\Gamma_2}(x)\,(1-q^{-h}t_{\beta^\vee}). 
\end{equation*}
\end{rem}

We need the following weaker version of the notion of a reduced $\lambda$-chain.

\begin{dfn}\label{def-weak-red} A $\lambda$-chain is \emph{weakly reduced} if it does not contain both a simple root and its negative.
\end{dfn}

Now let us consider arbitrary weakly reduced $\lambda$-chains $\Gamma_1$ and $\Gamma_2$. 
Then there exists a sequence $\Gamma_{1} = \Psi_{0}, \Psi_{1}, \ldots, \Psi_{p} = \Gamma_{1}^\ast$ of $\lambda$-chains 
such that $\Gamma_1^\ast$ is reduced, and each $\Psi_{k}$ is obtained from $\Psi_{k-1}$ by one of the procedure (YB) or (D). In a similar way, we relate $\Gamma_2$ to a reduced $\lambda$-chain $\Gamma_2^\ast$. Finally, we relate $\Gamma_1^\ast$ to $\Gamma_2^\ast$ by successively applying the procedure (YB). The weakly reduced property of  $\Gamma_1$ and $\Gamma_2$ implies that, in the above process, the procedure (D) never deletes a segment $(\pm\beta,\mp\beta)$ where $\beta$ is a simple root. 
By Propositions~\ref{prop:generating_function_YB} and \ref{prop:generating_function_D}, we derive the following theorem. 
%
%%%%%%%%%%%%%%%%%%%%%%%%%%%%%%%
% thm:generating_function_reduced %
%%%%%%%%%%%%%%%%%%%%%%%%%%%%%%%
%
\begin{thm}\label{thm:generating_function_reduced}
For arbitrary weakly reduced $\lambda$-chains $\Gamma_1$ and $\Gamma_2$, we have $\bG_{\Gamma_1}(x) = \bG_{\Gamma_2}(x)$. 
\end{thm}

%=========================%
% START SUBSECTION 0502 %
%=========================%

%%%%%%%%%%%%%%%
% subsec:comm %
%%%%%%%%%%%%%%%

\subsection{Combinatorial realization of commutativity.}\label{subsec:comm}

In this section we realize combinatorially the symmetry of the general Chevalley formula in~\cite{LNS} coming from commutativity in equivariant $K$-theory. As explained in the Introduction, this realization involves the commutativity of the composition of two functions $\bG_{\Gamma_1}$ and $\bG_{\Gamma_2}$, and is based on the generalized quantum Yang-Baxter moves. The main result here will also play an important role in the proof of the character identity in Section~\ref{subsec:character_identity}. 

We start by developing the notion of a weakly reduced chain of roots in Definition~\ref{def-weak-red}. Consider an arbitrary weight $\lambda$ and an arbitrary decomposition of it $\lambda=\lambda_1+\cdots+\lambda_p$. Let $\lambda_j=\sum_{i\in I} m_{ij}\varpi_i$.

\begin{dfn} The weight decomposition $\lambda=\lambda_1+\cdots+\lambda_p$ is \emph{cancellation free} if, for any $i\in I$, all the nonzero coefficients among $m_{i1},\ldots,m_{ip}$ have the same sign.
\end{dfn}

Given the above weight decomposition, consider $\lambda_j$-chains of roots $\Gamma_j$, for $j=1,\ldots,p$. Their concatenation, defined in the obvious way and denoted $\Gamma=\Gamma_1\ast\cdots\ast\Gamma_p$, is clearly a $\lambda$-chain. Note that the alcove path corresponding to $\Gamma$ is obtained by considering the shift of the alcove path for $\Gamma_j$ by $\lambda_1+\cdots+\lambda_{j-1}$, for $j=1,\ldots,p$, and by concatenating them in this order.

\begin{prop}\label{prop-wr} The $\lambda$-chain $\Gamma$ is a weakly reduced if and only if the weight decomposition $\lambda=\lambda_1+\cdots+\lambda_p$ is cancellation free and each $\lambda_j$-chain $\Gamma_j$ is weakly reduced.
\end{prop}

\begin{proof} The result is easily derived from the following general fact about a (not necessarily reduced) $\lambda$-chain $\Gamma=(\beta_1,\ldots,\beta_r)$, for an arbitrary weight $\lambda$, where $\alpha$ is a positive root (see, e.g., \cite[Lemma~5.3]{LP1}): 
\begin{equation*}
\pair{\lambda}{\alpha^\vee}= \#\{j \mid \beta_j = \alpha\} - \#\{j \mid \beta_j = -\alpha\}. 
\end{equation*}
This fact is applied to a simple root $\alpha=\alpha_i$, by noting that $\pair{\lambda}{\alpha_i^\vee}$ is the coefficient of $\varpi_i$ in the expansion of $\lambda$.
\end{proof}

Let us now consider a cancellation free weight decomposition $\lambda=\mu+\nu$, and weakly reduced chains of roots $\Gamma_1$ and $\Gamma_2$ corresponding to $\mu$ and $\nu$, respectively. Then, by Proposition~\ref{prop-wr}, we have the weakly reduced $\lambda$-chain $\Gamma:=\Gamma_1\ast\Gamma_2$. Observe that there exists a natural bijection 
\begin{equation}\label{concat-adm-subs}
\{ (A, B) \mid A \in \CA(w, \Gamma_1), \ B \in \CA(\ed(A), \Gamma_2) \} \rightarrow \CA(w, \Gamma); 
\end{equation}
let $A \ast B \in \CA(w, \Gamma)$ denote the image of $(A, B)$ under this bijection. 
The following lemma relates the statistics of interest under the bijection; its proof is based on completely similar arguments to those in the proof of \cite[Theorem~46]{LNS}.
%
%%%%%%%%%%%%%%%
% lem:statistics %
%%%%%%%%%%%%%%%
%
\begin{lem}[{\cite{LNS}}]\label{lem:statistics}
For $A \in \CA(w, \Gamma_1)$ and $B \in \CA(\ed(A), \Gamma_2)$, the following hold: 
\begin{enu}
\item $n(A \ast B) = n(A)+n(B)$; 
\item $\ed(A \ast B) = \ed(B)$; 
\item $\down(A \ast B) = \down(A) + \down(B)$; 
\item $\height(A \ast B) = \height(A) + \height(B)+\pair{\nu}{\down(A)}$; 
\item $\wt(A \ast B) = \wt(A) + \wt(B)$. 
\end{enu}
\end{lem}

We are now ready to prove the main result of this section.

\begin{thm}\label{comm_mu_nu} Given the above setup and any $x=wt_\xi \in W_{\af}$, we have
\begin{equation*}
\bG_{\Gamma_1}\circ\bG_{\Gamma_2}(x) = \bG_{\Gamma_2}\circ\bG_{\Gamma_1}(x)=\bG_{\Gamma}(x). 
\end{equation*}
These identities are realized combinatorially via the bijection~\eqref{concat-adm-subs} and the generalized quantum Yang-Baxter moves.
\end{thm}

\begin{proof} It suffices to prove the second equality. Indeed, this would imply that $\bG_{\Gamma_1}\circ\bG_{\Gamma_2}(x)=\bG_{\Gamma'}(x)$, where $\Gamma':=\Gamma_2\ast\Gamma_1$. The proof is then concluded by using Theorem~\ref{thm:generating_function_reduced} to show that $\bG_{\Gamma}(x)=\bG_{\Gamma'}(x)$. Recall that the mentioned theorem is proved by applying the generalized quantum Yang-Baxter moves.

By iterating the definition~\eqref{def:gf-adm}, we obtain
\begin{align*}
& \bG_{\Gamma_2}\circ\bG_{\Gamma_1}(x) \\[3mm]
& = \sum_{A \in \CA(w,\Gamma_1)}
     (-1)^{n(A)}q^{-\height(A)-\pair{{\mu}}{\xi}} e^{\wt(A)}
      \bG_{\Gamma_2}\left(\ed(A)t_{\xi+\down(A)}\right) \\[3mm]
\begin{split}
& = \sum_{A \in \CA(w,\Gamma_1)}\sum_{B \in \CA(\ed(A),\Gamma_2)}
    (-1)^{n(A)+n(B)} q^{-\height(A)-\pair{\mu}{\xi}-\height(B)-\pair{\nu}{\xi+\down(A)} } \\
& \hspace*{47mm} \times e^{\wt(A)+\wt(B)} \ed(B)t_{\xi+\down(A)+\down(B)} 
\end{split} \\[3mm]
& = \sum_{A \in \CA(w,\Gamma_1)}\sum_{B \in \CA(\ed(A),\Gamma_2)}
 (-1)^{n(A\ast B)}q^{-\height(A\ast B)-\pair{\lambda}{\xi}} e^{\wt(A\ast B)} \ed(A\ast B)t_{\xi+\down(A\ast B)} \\[3mm]
& =\bG_{\Gamma}(x). 
\end{align*}
The last two equalities are based on the bijection~\eqref{concat-adm-subs} and Lemma~\ref{lem:statistics}.
\end{proof}

Theorem~\ref{comm_mu_nu} immediately implies the following corollary involving a composition of more than two functions $\bG_{\Gamma_i}$. Here we use the corresponding setup that was defined above. Namely, we consider the cancellation free weight decomposition $\lambda=\lambda_1+\cdots+\lambda_p$, the weakly reduced $\lambda_j$-chains of roots $\Gamma_j$, for $j=1,\ldots,p$, and their concatenation $\Gamma=\Gamma_1\ast\cdots\ast\Gamma_p$.

\begin{cor}\label{cor:mult-ch} In the above setup, the composite of generating functions $\bG_{\Gamma_1}\circ\cdots\circ\bG_{\Gamma_p}(x)$ is invariant under permuting the maps $\bG_{\Gamma_\cdot}$, and coincides with $\bG_{\Gamma}(x)$.
\end{cor}

We now generalize the function $\bG_\Gamma$ on $R[P][W_{\af}]$ by defining the function $\widehat{\bG}_\Gamma$, which expresses the general $K$-theory Chevalley formula for semi-infinite flag manifolds in \cite{LNS}. In order to do this, we need some notation for partitions. 
Let $\lambda \in P$ and write it as $\lambda = \sum_{i \in I} m_{i}\vpi_{i}$. 
We define the set $\bPar(\lambda)$ by 
%
%%%%%%%%%%%%
% eq:def_Par %
%%%%%%%%%%%%
%
\begin{equation}\label{eq:def_Par}
\bPar(\lambda) := \left\{ \bchi = (\chi^{(i)})_{i \in I} \ \middle| \begin{array}{l} \chi^{(i)} \text{ is a partition whose length is} \\ \text{less than or equal to } \max\{m_{i}, 0\} \end{array} \right\}. 
\end{equation}
For $\bchi = (\chi^{(i)})_{i \in I} \in \bPar(\lambda)$, we write it as $\chi^{(i)} = (\chi_{1}^{(i)} \ge \chi_{2}^{(i)} \ge \cdots \ge \chi_{l_{i}}^{(i)} > 0)$, 
where $0 \le l_{i} \le \max\{m_{i}, 0\}$ and $\chi_{1}^{(i)}, \ldots, \chi_{l_{i}}^{(i)} \in \BZ$, and set 
%
%%%%%%%%%%%%%%%%%%%%%
% eq:def_statistics_par %
%%%%%%%%%%%%%%%%%%%%%
%
\begin{equation}\label{eq:def_statistics_par}
|\bchi| := \sum_{i \in I} \sum_{k = 1}^{l_{i}} \chi_{k}^{(i)}, \quad \iota(\bchi) := \sum_{i \in I} \chi_{1}^{(i)} \alpha_{i}^{\vee}; 
\end{equation}
if $\chi^{(i)} = \emptyset$, then we understand that $l_{i} = 0$ and $\chi_{1}^{(i)} = 0$.

\begin{dfn}
For each $\lambda$-chain $\Gamma$ and $x \in W_{\af}$, we define 
\begin{equation}\label{def:gf-adm-par}
\widehat{\bG}_{\Gamma}(x) := \sum_{\bchi\in\bPar(\lambda)}q^{-|\bchi|} \bG_\Gamma(x)t_{\iota(\bchi)}. 
\end{equation}
\end{dfn}

Like above, we now consider a cancellation free weight decomposition $\lambda=\mu+\nu$, weakly reduced chains of roots $\Gamma_1$ and $\Gamma_2$ corresponding to $\mu$ and $\nu$, respectively, and the weakly reduced $\lambda$-chain $\Gamma:=\Gamma_1\ast\Gamma_2$. Let $\mu=\sum_{i\in I} m_{i1}\varpi_i$ and $\nu=\sum_{i\in I} m_{i2}\varpi_i$, so $\lambda=\sum_{i\in I} m_{i}\varpi_i$ with $m_i=m_{i1}+m_{i2}$. We will show that there exists a natural bijection 
\begin{equation}\label{concat-par}
\bPar(\mu)\times\bPar(\nu)\rightarrow\bPar(\lambda), \quad (\bpsi,\bomega)\mapsto\bchi:=\bpsi\ast\bomega,
\end{equation}
which is compatible with the corresponding statistics. The above map is constructed by defining the partition $\chi^{(i)}$ in terms of the partitions $\psi^{(i)}$ and $\omega^{(i)}$, for each $i\in I$; we will identify a partition with its Young diagram. We may assume that $m_{i1},m_{i2}\ge 0$, and at least one is positive; indeed, otherwise $m_{i1},m_{i2}\le 0$, so $\psi^{(i)}=\omega^{(i)}=\emptyset$, and we let $\chi^{(i)}:=\emptyset$. In the non-trivial case, we consider a rectangular partition with $m_{i2}$ rows of size $\psi^{(i)}_1$; then $\chi^{(i)}$ is defined as the result of attaching $\omega^{(i)}$ at the right of the rectangle (top justified) and $\psi^{(i)}$ at the bottom of the rectangle (left justified). It is easy to verify that the result is indeed a partition of length at most $m_i$, as needed, as well as the fact that this map is invertible.

\begin{lem}\label{lem:statistics-par} For $\bpsi\in\bPar(\mu)$ and $\bomega\in\bPar(\nu)$, the following hold: 
\begin{enu}
\item $\iota(\bpsi\ast\bomega)=\iota(\bpsi)+\iota(\bomega)$; 
\item $|\bpsi\ast\bomega|=|\bpsi|+|\bomega|+\pair{\nu}{\iota(\bpsi)}$. 
\end{enu}
\end{lem}

\begin{proof} We use the above notation, in particular $\bchi:=\bpsi\ast\bomega$, as well as the fact that the weight decomposition $\lambda=\mu+\nu$ is cancellation free. The first relation is clear by construction. The second one follows from the fact that $\pair{\nu}{\iota(\bpsi)}=\sum_{i\in I}m_{i2}\psi^{(i)}_1=\sum_{i\in I}\max\{m_{i2},0\} \psi^{(i)}_1$; here we note that $m_{i2}\psi^{(i)}_1$ is the size of the rectangle used to construct $\chi^{(i)}$ in the non-trivial case.
\end{proof}

\begin{thm}\label{comm_mu_nu_par} Given the above setup and any $x=wt_\xi \in W_{\af}$, we have
\begin{equation*}
\widehat{\bG}_{\Gamma_1}\circ\widehat{\bG}_{\Gamma_2}(x) = \widehat{\bG}_{\Gamma_2}\circ\widehat{\bG}_{\Gamma_1}(x)=\widehat{\bG}_{\Gamma}(x). 
\end{equation*}
These identities are realized combinatorially via the bijections~\eqref{concat-adm-subs}, \eqref{concat-par}, and the generalized quantum Yang-Baxter moves.
\end{thm}

\begin{proof} Like in the proof of Theorem~\ref{comm_mu_nu}, it suffices to prove the second equality. By iterating the definition~\eqref{def:gf-adm-par} and by also using~\eqref{def:gf-adm}, we obtain
\begin{align*}
& \widehat{\bG}_{\Gamma_2}\circ\widehat{\bG}_{\Gamma_1}(x) \\[3mm]
&=\sum_{\bpsi\in\bPar(\mu)}\sum_{A\in\CA(w,\Gamma_1)}(-1)^{n(A)}q^{-\pair{\mu}{\xi}-\height(A)-|\bpsi|}e^{\wt(A)}\widehat{\bG}_{\Gamma_2}(\ed(A)t_{\xi+\down(A)+\iota(\bpsi)})\\[3mm]
\begin{split}
&=\sum_{\bpsi\in\bPar(\mu)}\sum_{\bomega\in\bPar(\nu)}\sum_{A\in\CA(w,\Gamma_1)}(-1)^{n(A)}q^{-\pair{\mu}{\xi}-\height(A)-|\bpsi|-|\bomega|}\\
 &\hspace*{51mm} \times e^{\wt(A)}{\bG}_{\Gamma_2}(\ed(A)t_{\xi+\down(A)+\iota(\bpsi)})t_{\iota(\bomega)}
\end{split} \\[3mm]
\begin{split}
&=\sum_{\bpsi\in\bPar(\mu)}\sum_{\bomega\in\bPar(\nu)}q^{-|\bpsi|-|\bomega|-\pair{\nu}{\iota(\bpsi)}}\sum_{A\in\CA(w,\Gamma_1)}(-1)^{n(A)}q^{-\pair{\mu}{\xi}-\height(A)}\\
 &\hspace*{80mm} \times e^{\wt(A)}{\bG}_{\Gamma_2}(\ed(A)t_{\xi+\down(A)})t_{\iota(\bpsi)+\iota(\bomega)} 
\end{split} \\[3mm]
&=\sum_{\bpsi\in\bPar(\mu)}\sum_{\bomega\in\bPar(\nu)}q^{-|\bpsi|-|\bomega|-\pair{\nu}{\iota(\bpsi)}}{\bG}_{\Gamma_2}\circ\bG_{\Gamma_1}(wt_{\xi})t_{\iota(\bpsi)+\iota(\bomega)} \\[3mm]
&=\sum_{\bpsi\in\bPar(\mu)}\sum_{\bomega\in\bPar(\nu)}q^{-|\bpsi\ast\bomega|}{\bG}_{\Gamma}(wt_{\xi})t_{\iota(\bpsi\ast\bomega)}\\[3mm]
&=\widehat{\bG}_{\Gamma}(x).
\end{align*}
The last two equalities are based on the bijection~\eqref{concat-par}, Lemma~\ref{lem:statistics-par}, and Theorem~\ref{comm_mu_nu}.
\end{proof}

\begin{rems} (1) Theorem~\ref{comm_mu_nu_par} exhibits a combinatorial realization of the symmetry of the general Chevalley formula~\cite[Theorem~33]{LNS} coming from commutativity in equivariant $K$-theory.

(2) Corollary~\ref{cor:mult-ch} can be extended to the setup of Theorem~\ref{comm_mu_nu_par}.
\end{rems}

%=========================%
% START SUBSECTION 0503 %
%=========================%

%%%%%%%%%%%%%%%%%%%%%%%%%
% subsec:character_identity %
%%%%%%%%%%%%%%%%%%%%%%%%%

\subsection{Identity of Chevalley type for graded characters.}\label{subsec:character_identity}
As an application of the results in Sections~\ref{subsec:generating_function} and \ref{subsec:comm}, 
we obtain an identity of ``Chevalley-type'' for the graded characters of Demazure submodules of (level-zero) extremal weight modules over a quantum affine algebra. 

Let $\Fg_{\af}$ be the untwisted affine Lie algebra whose underlying finite-dimensional simple Lie algebra is $\Fg$, 
and let $U_{\q}(\Fg_{\af})$ denote the quantum affine algebra associated to $\Fg_{\af}$ with Chevalley generator $E_{i}, F_{i} \in U_{\q}(\Fg_{\af})$, $i \in I_{\af} = I \sqcup \{0\}$, 
where $\q$ is an indeterminate. 
We denote by $U_{\q}^{-}(\Fg_{\af}) := \langle F_{i} \rangle_{i \in I_{\af}} \subset U_{\q}(\Fg_{\af})$ 
the subalgebra of $U_{\q}(\Fg_{\af})$ generated by $\{F_{i} \mid i \in I_{\af}\}$. 
Also, let $W_{\af} = W \ltimes \{t_{\xi} \mid \xi \in Q^{\vee}\} \simeq W \ltimes Q^{\vee}$ be the (affine) Weyl group of $\Fg_{\af}$, 
where $t_{\xi}$, $\xi \in Q^{\vee}$, denotes the translation by $\xi$ (see \cite[Proposition~6.5]{Kac}). 

For each $\lambda \in P^{+}$, we denote by $V(\lambda)$ 
the \emph{level-zero extremal weight module} of extremal weight $\lambda$ over $U_{\q}(\Fg_{\af})$, 
which is equipped with a family $\{v_{x}\}_{x \in W_{\af}} \subset V(\lambda)$ of extremal weight vectors, 
where $v_{x} \in V(\lambda)$, $x \in W_{\af}$, is an extremal weight vector of weight $x\lambda$ (see \cite[Proposition~8.2.2]{Kas}). 
For $x \in W_{\af}$ and $\lambda \in P^{+}$, the \emph{Demazure submodule} $V_{x}^{-}(\lambda)$ is defined by $V_{x}^{-}(\lambda) := U_{\q}^{-}(\Fg_{\af})v_{x}$. 
We denote by $\gch V_{x}^{-}(\lambda)$ the \emph{graded character} of $V_{x}^{-}(\lambda)$ (see \cite[\S 2.4]{KNS}).
If $x = w t_{\xi}$ with $w \in W$ and $\xi \in Q^{\vee}$, then we know that $\gch V_{x}^{-}(\lambda) \in \mathbb{Z}[P]\bra{q^{-1}}q^{-\pair{\lambda}{\xi}}$; in fact, we know that $\gch V_{w}^{-}(\lambda) \in \mathbb{Z}\bra{q^{-1}}[P]$ for $w \in W$.

We will prove the following identity for graded characters, which is a representation-theoretic analogue 
of the general Chevalley formula in the equivariant $K$-group of semi-infinite flag manifolds (\cite[Theorem~33]{LNS}). 
%
%%%%%%%%%%%%%%%%%%%%%
% thm:PC-type_formula %
%%%%%%%%%%%%%%%%%%%%%
%
\begin{thm}\label{thm:PC-type_formula}
Let $\mu \in P^{+}$ and $x \in W_{\af}$. 
We write $x$ as $x = w t_{\xi}$ with $w \in W$ and $\xi \in Q^{\vee}$. 
Take $\lambda \in P$ such that $\mu + \lambda \in P^{+}$, 
and let $\Gamma$ be an arbitrary reduced $\lambda$-chain. 
Then we have 
\begin{equation}
\begin{split}\label{eq:PC-type_formula}
& \gch V_{x}^{-}(\mu + \lambda) = \\ 
& \sum_{A \in \CA(w, \Gamma)} \sum_{\bchi \in \bPar(\lambda)} (-1)^{n(A)} q^{-\height(A) - \pair{\lambda}{\xi} - |\bchi|} e^{\wt(A)} \gch V_{\ed(A)t_{\xi + \down(A) + \iota(\bchi)}}^{-}(\mu). 
\end{split}
\end{equation}
\end{thm}

\begin{rem}\label{rhs-cancel}
The right-hand side of \eqref{eq:PC-type_formula} is identical to zero if $\mu + \lambda \notin P^{+}$; 
the proof is given in Appendix~\ref{sec:App_C}. 
\end{rem}

Although Theorem~\ref{thm:PC-type_formula} can be proved in a parallel way to \cite[Theorem~33]{LNS},
we show that it follows immediately from the results in Sections~\ref{subsec:generating_function} and \ref{subsec:comm}. 

Now we recall two special cases of Theorem~\ref{thm:PC-type_formula}, 
i.e., the cases that $\lambda$ is dominant or anti-dominant.
The following theorem gives the identity for dominant weights;
this is a restatement of \cite[Corollary~C.1]{NOS} in terms of the quantum alcove model, which is given by exactly the same argument as for \cite[Theorem~29]{LNS}.
Here, for a dominant weight $\lambda \in P^{+}$, the \emph{lex $\lambda$-chain} is a $\lambda$-chain constructed in \cite[Proposition~4.2]{LP2}. 
%
%%%%%%%%%%%%%%%%%%
% thm:PC-dominant %
%%%%%%%%%%%%%%%%%%
%
\begin{thm}[{cf. \cite[Theorem~29]{LNS} and \cite[Corollary~C.1]{NOS}; see also \cite[Proposition~D.1]{KNS}}]\label{thm:PC-dominant}
Let $\mu, \lambda \in P^{+}$, and $x = w t_{\xi} \in W_{\af}$ with $w \in W$ and $\xi \in Q^{\vee}$. 
Let $\Gamma$ be the lex $\lambda$-chain.
Then we have 
\begin{equation*}
\gch V_{x}^{-}(\mu + \lambda) = \sum_{A \in \CA(w, \Gamma)} \sum_{\bchi \in \bPar(\lambda)} q^{-\height(A) - \pair{\lambda}{\xi} - |\bchi|} e^{\wt(A)} \gch V_{\ed(A)t_{\xi + \down(A) + \iota(\bchi)}}^{-}(\mu). 
\end{equation*}
\end{thm}

Also, the following theorem gives the identity for anti-dominant weights;
this is a restatement of \cite[Corollary~3.15]{NOS} in terms of the quantum 
alcove model, which is given by exactly the same argument as for \cite[Theorem~32]{LNS}.
Here, following \cite[Section~4.2]{LNS}, 
the \emph{lex $\lambda$-chain} for an anti-dominant weight $\lambda \in -P^{+}$ 
is defined to be the reverse of the lex $(-\lambda)$-chain. 
%
%%%%%%%%%%%%%%%%%%%%%%
% thm:PC-anti_dominant %
%%%%%%%%%%%%%%%%%%%%%%
%
\begin{thm}[{cf. \cite[Theorem~32]{LNS} and \cite[Corollary~3.15]{NOS}; see also \cite[Proposition~D.1]{KNS}}]\label{thm:PC-anti_dominant}
Let $\mu \in P^{+}$, and $x = w t_{\xi} \in W_{\af}$ with $w \in W$ and $\xi \in Q^{\vee}$. 
Take $\lambda \in -P^{+}$ such that $\mu + \lambda \in P^{+}$, 
and let $\Gamma$ be the lex $\lambda$-chain. 
Then we have 
\begin{equation*}
\gch V_{x}^{-}(\mu + \lambda) = \sum_{A \in \CA(w, \Gamma)} (-1)^{|A|} q^{-\height(A) - \pair{\lambda}{\xi}} e^{\wt(A)} \gch V_{\ed(A)t_{\xi + \down(A)}}^{-}(\mu). 
\end{equation*}
\end{thm}

\begin{proof}[Proof of Theorem~\ref{thm:PC-type_formula}]
Let $\mu, \lambda, x$ and $\Gamma$ be as in the statement of Theorem~\ref{thm:PC-type_formula}. Write $\lambda = \lambda^{+} + \lambda^{-}$, with $\lambda^{+} \in P^{+}$ and $\lambda^{-} \in -P^{+}$ given by
\begin{equation*}
\lambda^{+} := \sum_{i \in I} \max\{ \langle\lambda, \alpha_{i}^{\vee}\rangle, 0 \} \varpi_{i}, \qquad \lambda^{-} := \sum_{i \in I} \min\{\langle\lambda, \alpha_{i}^{\vee}\rangle, 0\} \varpi_{i}.
\end{equation*}
Note that the weight decomposition $\lambda = \lambda^{+} + \lambda^{-}$ is cancellation free. 
Take a lex $\lambda^{+}$-chain (resp., lex $\lambda^{-}$-chain) $\Gamma^{+}$ (resp., $\Gamma^{-}$).  
Note that the two chains of roots are reduced, and $\Gamma^+$ consists of positive roots, while $\Gamma^-$ consists of negative roots. 
Define a $\lambda$-chain $\Gamma_{0}$ as the concatenation $\Gamma^+\ast\Gamma^-$, which is weakly reduced by Proposition~\ref{prop-wr}.

By Theorems~\ref{comm_mu_nu_par} and \ref{thm:generating_function_reduced}, for $x\in W_{\af}$ we have
\begin{equation}\label{proof_PC}\widehat{\bG}_{\Gamma^-}\circ\widehat{\bG}_{\Gamma^+}(x)=\widehat{\bG}_{\Gamma_0}(x)=\widehat{\bG}_{\Gamma}(x).\end{equation} 

Now consider the correspondence $x \mapsto \gch V_{x}^{-}(\mu)$ for $x \in W_{\af}$, which defines an $R[P]$-module homomorphism $R[P][W_\af] \rightarrow \BZ \pra{q^{-1}}[P]$. Under this homomorphism, $\widehat{\bG}_{\Gamma}(x)$ is mapped to the right-hand side of~\eqref{eq:PC-type_formula}. By~\eqref{proof_PC}, we obtain the same result by applying the homomorphism to $\widehat{\bG}_{\Gamma^-}\circ\widehat{\bG}_{\Gamma^+}(x)$. We observe that doing this parallels the process of expanding $\gch V_{x}^{-}(\mu+\lambda)=\gch V_{x}^{-}((\mu+\lambda^-)+\lambda^+)$ in terms of $\gch V_{\cdot}^{-}(\mu+\lambda^-)$ by Theorem~\ref{thm:PC-dominant}, followed by expanding the result in terms of $\gch V_{\cdot}^{-}(\mu)$ by Theorem~\ref{thm:PC-anti_dominant}; here we use the fact that $\mu+\lambda\in P^+$ implies $\mu+\lambda^-\in P^+$. The mentioned observation proves that, by applying the above homomorphism to $\widehat{\bG}_{\Gamma^-}\circ\widehat{\bG}_{\Gamma^+}(x)$, we obtain $\gch V_{x}^{-}(\mu+\lambda)$. We conclude that the right-hand side of~\eqref{eq:PC-type_formula} coincides with $\gch V_{x}^{-}(\mu+\lambda)$. 
\end{proof}

\begin{rem}
Theorem~\ref{thm:PC-type_formula} can also be proved by using the $\lambda$-chain $\Gamma_{0}^{\ast} :=\Gamma^-\ast\Gamma^+$ instead of $\Gamma_0$.
\end{rem}

%=========================%
% START SUBSECTION 0504 %
%=========================%

\subsection{Towards a signed crystal structure on the quantum alcove model for an arbitrary weight.} 
Crystals are colored directed graphs encoding the structure of representations of quantum algebras $U_{\q}(\Fg)$ in the limit $\q\to 0$, where $\Fg$ is a symmetrizable Kac-Moody algebra. The vertices $B$ of the crystal correspond to the elements of a crystal basis for the representation, and the edges correspond to the action of the Chevalley generators $e_i,\,f_i$ in the above limit. Formally, we define the \emph{crystal operators} $\widetilde{e}_i,\widetilde{f}_i:B\rightarrow B\sqcup\{\mathbf{0}\}$, where the value $\mathbf{0}$ corresponds to the operators being undefined, and $\widetilde{e}_i$ is a partial inverse to $\widetilde{f}_i$. These operators are subject to several conditions; see e.g.~\cite{HK} for all the background information on crystals. 

We define a \emph{signed crystal} simply as a crystal together with a sign function on the vertex set $B$; note that we do not require the crystal operators to preserve signs. An isomorphism of signed crystals $B$ and $B'$ is a sijection between $B$ and $B'$ which commutes with the crystal operators.

Given a dominant weight $\lambda$ and a reduced $\lambda$-chain $\Gamma$, an affine crystal structure was constructed on $\CA(e,\Gamma)$ in~\cite{LL1}, which was then shown in~\cite{LNSSS2} to uniformly describe tensor products of single-column Kirillov-Reshetikhin crystals of quantum affine algebras; 
note that for all $w \in W$, the set $\CA(w, \Gamma)$ can be equipped with an affine crystal structure 
through the bijection with the affine crystal $\CA(e, \Gamma)$, which is afforded by \cite[Proposition~28]{LNS} and quantum Yang-Baxter moves. 
Also, for an anti-dominant weight $\lambda'$ and a reduced $\lambda'$-chain $\Gamma'$, an argument similar to that in the proof of \cite[Theorem~8.6]{LP1} yields a signed crystal structure on the set $\CA(\Gamma', w) := \{ A \in \CA(\Gamma') \mid \ed(A) = w \}$, with $\CA(\Gamma') := \bigsqcup_{w \in W} \CA(w, \Gamma')$, 
which is in bijective correspondence with $\CA(ww_{\circ}, (\Gamma')^{\ast})$, where $(\Gamma')^{\ast}$ is a reduced $w_{\circ}\lambda'$-chain of roots dual to $\Gamma'$. 
In both the dominant and anti-dominant cases, as well as in general (below), we use the the same sign function as throughout the paper; in particular, in the dominant case the sign function is identically $1$.

We now propose the construction of a (partial) signed crystal structure on $\CA(\Gamma) := \bigsqcup_{w \in W} \CA(w, \Gamma)$, where $\Gamma$ is a reduced chain of roots corresponding to an arbitrary weight $\lambda$. We use the same objects and facts as in Sections~\ref{subsec:comm} and \ref{subsec:character_identity}, namely $\lambda=\lambda^++\lambda^-$, the lex $\lambda^{+}$-chain (resp., lex $\lambda^{-}$-chain) $\Gamma^{+}$ (resp., $\Gamma^{-}$), their concatenation $\Gamma_0^{\ast}$ (not $\Gamma_0$), and the bijection~\eqref{concat-adm-subs}. 
Based on these facts, we can define a signed crystal structure on $\CA(\Gamma_0^{\ast})$ by decomposing it as $\CA(\Gamma_0^{\ast}) = \bigsqcup_{w \in W} (\CA(\Gamma^{-}, w) \ast \CA(w, \Gamma^{+}))$, where $\CA(\Gamma^{-}, w) := \{ A \in \CA(\Gamma^{-}) \mid \text{end}(A) = w \}$ for $w \in W$; note that the concatenation $\CA(\Gamma^{-}, w) \ast \CA(w, \Gamma^{+})$ is a well-defined crystal for all $w \in W$.
On another hand, we know that $\Gamma_0^{\ast}$ can be related to the reduced $\lambda$-chain $\Gamma$ by the procedures (YB) and (D). By propagating the signed crystal structures through the corresponding generalized quantum Yang-Baxter moves, we end up with a (partial) signed crystal structure on $\CA(\Gamma)$, that is, a signed crystal structure for which crystal operators are defined only on a subset of $\CA(\Gamma)$.
This construction is inspired by the construction of the level-zero extremal weight module $V(\lambda)$ over a quantum affine algebra as a quotient, with a certain universality, of $V(\Lambda^{-}) \otimes  V(\Lambda^{+})$, where $\Lambda^{+}$ (resp., $\Lambda^{-}$) is an affine dominant (resp., anti-dominant) weight such that $\lambda = \Lambda^{+} + \Lambda^{-}$, and $V(\Lambda^{+})$ (resp., $V(\Lambda^{-})$) is the irreducible highest (resp., lowest) weight module of the corresponding weight. 

We state the following conjecture, which will have several consequences.

\begin{conj}\label{YB-move-crystal} For all reduced chain of roots $\Gamma$ corresponding to an arbitrary weight $\lambda$, there exists a signed crystal structure on the whole of $\CA(\Gamma)$, which extends the (partial) crystal structure defined above. Moreover, the sijection $(I_1,I_2,Y)$ defining a generalized quantum Yang-Baxter move in Theorem~\ref{thm:YB-move} commutes with the crystal operators defined on $\CA(\Gamma)$. 
\end{conj}

Conjecture~\ref{YB-move-crystal} would generalize the corresponding result for the quantum Yang-Baxter moves, namely~\cite[Theorem~3.8]{LL2}. Moreover, it would realize the generating function identity in Theorem~\ref{thm:generating_function_reduced} as an isomorphism of the underlying signed crystals, given that the proof of this identity is based on the generalized quantum Yang-Baxter moves.

Recall that, based on the generalized quantum Yang-Baxter moves, we also proved the invariance of a composite of generating functions $\bG_{\Gamma_1}\circ\cdots\circ\bG_{\Gamma_p}(x)$ with respect to permuting the maps $\bG_{\Gamma_\cdot}$ (under the appropriate conditions, see Theorem~\ref{comm_mu_nu} and Corollary~\ref{cor:mult-ch}). This proof essentially amounts to concatenating the chains of roots $\Gamma_1,\ldots,\Gamma_p$ in different orders, considering the quantum alcove models based on these concatenations, and relating them by the procedures (YB) and (D). Thus, Conjecture~\ref{YB-move-crystal} would also realize the invariance of a composite of $\bG_{\Gamma_\cdot}$ as an isomorphism of the underlying signed crystals. This result would generalize the realization of the combinatorial $R$-matrix in~\cite{LL2}. Indeed, given $\lambda=\sum_j\varpi_{i_j}$ and a chain of roots $\Gamma$ obtained by concatenating $\varpi_{i_j}$-chains, we know from~\cite{LNSSS2} that $\CA(e,\Gamma)$ realizes the tensor product of single-column Kirillov-Reshetikhin crystals corresponding to $\varpi_{i_j}$, for $j=1,2,\ldots$, in that order; therefore, changing the order of $j$ corresponds to permuting the tensor factors and the $\varpi_{i_j}$-subchains of $\Gamma$, while the quantum alcove models based on the respective $\lambda$-chains are related by quantum Yang-Baxter moves. 

Finally, we believe that Conjecture~\ref{YB-move-crystal} would lead to a (signed) crystal-theoretic realization of the character identity \eqref{eq:PC-type_formula} in Theorem~\ref{thm:PC-type_formula}. Indeed, both sides of this identity involve characters of (Demazure submodules of) level-zero extremal weight modules, whereas the corresponding crystals were realized in~\cite{INS} in terms of so-called \emph{semi-infinite Lakshmibai-Seshadri paths}. The right-hand side also involves the sets $\CA(w,\Gamma)$, with the signed crystal structure defined above. Thus, we have a signed crystal structure on the objects on the right-hand side coming from tensor products of a crystal of the first type with one of the second type. The sijection relating the objects on the two sides would include a sign-reversing involution realizing significant cancellations on the right-hand side, as Remark~\ref{rhs-cancel} suggests (see Proposition~\ref{prop:sum_equals_0} in Appendix~C). As supporting evidence for the crystal-theoretic realization of~\eqref{eq:PC-type_formula}, we indicate the crystal-theoretic proof in \cite[Appendix~B]{NOS} (based on semi-infinite Lakshmibai-Seshadri paths) of Proposition~\ref{prop:sum_equals_0_anti-dominant}, which is used to prove Proposition~\ref{prop:sum_equals_0}.

%%%%%%%%%
\appendix
%%%%%%%%%
%
%%%%%%%%%%%%%%%%%%%%
\section*{Appendices.}
%%%%%%%%%%%%%%%%%%%%
%
%===================%
% START SECTION 0A %
%===================%

%%%%%%%%%%%%
% sec:App_A %
%%%%%%%%%%%%

\section{Matrices of some operators in type \texorpdfstring{$G_{2}$}{G2}.}\label{sec:App_A}
Here we give the list of the matrices of operators (1)--(12) in Section~\ref{subsec:typeG2}. 
First, the following are the matrices (with respect to the basis 
\begin{equation*}
W = \{e, s_{1}, s_{2}, s_{1}s_{2}, s_{2}s_{1}, s_{1}s_{2}s_{1}, s_{2}s_{1}s_{2}, s_{1}s_{2}s_{1}s_{2}, s_{2}s_{1}s_{2}s_{1}, s_{1}s_{2}s_{1}s_{2}s_{1}, s_{2}s_{1}s_{2}s_{1}s_{2}, w_{\circ}\} 
\end{equation*}
of $K[W]$) of operators $\sQ_{\gamma}$, 
$\gamma \in \Delta^{+} = \{ \alpha_{1}, 3\alpha_{1}+\alpha_{2}, 2\alpha_{1}+\alpha_{2}, 3\alpha_{1}+2\alpha_{2}, \alpha_{1}+\alpha_{2}, \alpha_{2} \}$ 
(cf. \cite[Fig.~2\,(C)]{LL2}): 

\vspace{5mm}

\noindent [1] $\sQ_{\alpha_{1}}$: 
\begin{equation*}
\left( 
\begin{array}{cccccccccccc}
0 & Q_{1} & 0 & 0 & 0 & 0 & 0 & 0 & 0 & 0 & 0 & 0 \\ 
1 & 0 & 0 & 0 & 0 & 0 & 0 & 0 & 0 & 0 & 0 & 0 \\ 
0 & 0 & 0 & 0 & Q_{1} & 0 & 0 & 0 & 0 & 0 & 0 & 0 \\ 
0 & 0 & 0 & 0 & 0 & Q_{1} & 0 & 0 & 0 & 0 & 0 & 0 \\ 
0 & 0 & 1 & 0 & 0 & 0 & 0 & 0 & 0 & 0 & 0 & 0 \\ 
0 & 0 & 0 & 1 & 0 & 0 & 0 & 0 & 0 & 0 & 0 & 0 \\ 
0 & 0 & 0 & 0 & 0 & 0 & 0 & 0 & Q_{1} & 0 & 0 & 0 \\ 
0 & 0 & 0 & 0 & 0 & 0 & 0 & 0 & 0 & Q_{1} & 0 & 0 \\ 
0 & 0 & 0 & 0 & 0 & 0 & 1 & 0 & 0 & 0 & 0 & 0 \\ 
0 & 0 & 0 & 0 & 0 & 0 & 0 & 1 & 0 & 0 & 0 & 0 \\ 
0 & 0 & 0 & 0 & 0 & 0 & 0 & 0 & 0 & 0 & 0 & Q_{1} \\ 
0 & 0 & 0 & 0 & 0 & 0 & 0 & 0 & 0 & 0 & 1 & 0 
\end{array}
\right)
\end{equation*}

\noindent [2] $\sQ_{3\alpha_{1}+\alpha_{2}}$: 
\begin{equation*}
\left( 
\begin{array}{cccccccccccc}
0 & 0 & 0 & 0 & 0 & Q_{1}Q_{2} & 0 & 0 & 0 & 0 & 0 & 0 \\ 
0 & 0 & 0 & 0 & 0 & 0 & 0 & 0 & 0 & 0 & 0 & 0 \\ 
0 & 0 & 0 & 0 & 0 & 0 & 0 & 0 & Q_{1}Q_{2} & 0 & 0 & 0 \\ 
0 & 0 & 0 & 0 & 0 & 0 & 0 & 0 & 0 & Q_{1}Q_{2} & 0 & 0 \\ 
0 & 1 & 0 & 0 & 0 & 0 & 0 & 0 & 0 & 0 & 0 & 0 \\ 
0 & 0 & 0 & 0 & 0 & 0 & 0 & 0 & 0 & 0 & 0 & 0 \\ 
0 & 0 & 0 & 0 & 0 & 0 & 0 & 0 & 0 & 0 & 0 & Q_{1}Q_{2} \\ 
0 & 0 & 0 & 0 & 0 & 0 & 0 & 0 & 0 & 0 & 0 & 0 \\ 
0 & 0 & 0 & 0 & 0 & 0 & 0 & 0 & 0 & 0 & 0 & 0 \\ 
0 & 0 & 0 & 0 & 0 & 0 & 0 & 0 & 0 & 0 & 0 & 0 \\ 
0 & 0 & 0 & 0 & 0 & 0 & 0 & 1 & 0 & 0 & 0 & 0 \\ 
0 & 0 & 0 & 0 & 0 & 0 & 0 & 0 & 0 & 0 & 0 & 0 
\end{array}
\right) 
\end{equation*}

\noindent [3] $\sQ_{2\alpha_{1}+\alpha_{2}}$: 
\begin{equation*}
\left( 
\begin{array}{cccccccccccc}
0 & 0 & 0 & 0 & 0 & 0 & 0 & 0 & 0 & 0 & 0 & 0 \\ 
0 & 0 & 0 & 0 & 0 & 0 & 0 & 0 & 0 & 0 & 0 & 0 \\ 
0 & 0 & 0 & 0 & 0 & 0 & 0 & 0 & 0 & 0 & 0 & 0 \\ 
0 & 0 & 0 & 0 & 0 & 0 & 0 & 0 & 0 & 0 & 0 & 0 \\ 
0 & 0 & 0 & 0 & 0 & 0 & 0 & 0 & 0 & 0 & 0 & 0 \\ 
0 & 0 & 0 & 0 & 1 & 0 & 0 & 0 & 0 & 0 & 0 & 0 \\ 
0 & 0 & 0 & 0 & 0 & 0 & 0 & 0 & 0 & 0 & 0 & 0 \\ 
0 & 0 & 0 & 0 & 0 & 0 & 1 & 0 & 0 & 0 & 0 & 0 \\ 
0 & 0 & 0 & 0 & 0 & 0 & 0 & 0 & 0 & 0 & 0 & 0 \\ 
0 & 0 & 0 & 0 & 0 & 0 & 0 & 0 & 0 & 0 & 0 & 0 \\ 
0 & 0 & 0 & 0 & 0 & 0 & 0 & 0 & 0 & 0 & 0 & 0 \\ 
0 & 0 & 0 & 0 & 0 & 0 & 0 & 0 & 0 & 0 & 0 & 0 
\end{array}
\right) 
\end{equation*}

\noindent [4] $\sQ_{3\alpha_{1}+2\alpha_{2}}$: 
\begin{equation*}
\left( 
\begin{array}{cccccccccccc}
0 & 0 & 0 & 0 & 0 & 0 & 0 & 0 & 0 & 0 & Q_{1}Q_{2}^{2} & 0 \\ 
0 & 0 & 0 & 0 & 0 & 0 & 0 & 0 & 0 & 0 & 0 & Q_{1}Q_{2}^{2} \\ 
0 & 0 & 0 & 0 & 0 & 0 & 0 & 0 & 0 & 0 & 0 & 0 \\ 
0 & 0 & 0 & 0 & 0 & 0 & 0 & 0 & 0 & 0 & 0 & 0 \\ 
0 & 0 & 0 & 0 & 0 & 0 & 0 & 0 & 0 & 0 & 0 & 0 \\ 
0 & 0 & 0 & 0 & 0 & 0 & 0 & 0 & 0 & 0 & 0 & 0 \\ 
0 & 0 & 0 & 1 & 0 & 0 & 0 & 0 & 0 & 0 & 0 & 0 \\ 
0 & 0 & 0 & 0 & 0 & 0 & 0 & 0 & 0 & 0 & 0 & 0 \\ 
0 & 0 & 0 & 0 & 0 & 1 & 0 & 0 & 0 & 0 & 0 & 0 \\ 
0 & 0 & 0 & 0 & 0 & 0 & 0 & 0 & 0 & 0 & 0 & 0 \\ 
0 & 0 & 0 & 0 & 0 & 0 & 0 & 0 & 0 & 0 & 0 & 0 \\ 
0 & 0 & 0 & 0 & 0 & 0 & 0 & 0 & 0 & 0 & 0 & 0 
\end{array}
\right) 
\end{equation*}

\noindent [5] $\sQ_{\alpha_{1}+\alpha_{2}}$: 
\begin{equation*}
\left( 
\begin{array}{cccccccccccc}
0 & 0 & 0 & 0 & 0 & 0 & 0 & 0 & 0 & 0 & 0 & 0 \\ 
0 & 0 & 0 & 0 & 0 & 0 & 0 & 0 & 0 & 0 & 0 & 0 \\ 
0 & 0 & 0 & 0 & 0 & 0 & 0 & 0 & 0 & 0 & 0 & 0 \\ 
0 & 0 & 1 & 0 & 0 & 0 & 0 & 0 & 0 & 0 & 0 & 0 \\ 
0 & 0 & 0 & 0 & 0 & 0 & 0 & 0 & 0 & 0 & 0 & 0 \\ 
0 & 0 & 0 & 0 & 0 & 0 & 0 & 0 & 0 & 0 & 0 & 0 \\ 
0 & 0 & 0 & 0 & 0 & 0 & 0 & 0 & 0 & 0 & 0 & 0 \\ 
0 & 0 & 0 & 0 & 0 & 0 & 0 & 0 & 0 & 0 & 0 & 0 \\ 
0 & 0 & 0 & 0 & 0 & 0 & 0 & 0 & 0 & 0 & 0 & 0 \\ 
0 & 0 & 0 & 0 & 0 & 0 & 0 & 0 & 1 & 0 & 0 & 0 \\ 
0 & 0 & 0 & 0 & 0 & 0 & 0 & 0 & 0 & 0 & 0 & 0 \\ 
0 & 0 & 0 & 0 & 0 & 0 & 0 & 0 & 0 & 0 & 0 & 0 
\end{array}
\right) 
\end{equation*}

\noindent [6] $\sQ_{\alpha_{2}}$: 
\begin{equation*}
\left( 
\begin{array}{cccccccccccc}
0 & 0 & Q_{2} & 0 & 0 & 0 & 0 & 0 & 0 & 0 & 0 & 0 \\ 
0 & 0 & 0 & Q_{2} & 0 & 0 & 0 & 0 & 0 & 0 & 0 & 0 \\ 
1 & 0 & 0 & 0 & 0 & 0 & 0 & 0 & 0 & 0 & 0 & 0 \\ 
0 & 1 & 0 & 0 & 0 & 0 & 0 & 0 & 0 & 0 & 0 & 0 \\ 
0 & 0 & 0 & 0 & 0 & 0 & Q_{2} & 0 & 0 & 0 & 0 & 0 \\ 
0 & 0 & 0 & 0 & 0 & 0 & 0 & Q_{2} & 0 & 0 & 0 & 0 \\ 
0 & 0 & 0 & 0 & 1 & 0 & 0 & 0 & 0 & 0 & 0 & 0 \\ 
0 & 0 & 0 & 0 & 0 & 1 & 0 & 0 & 0 & 0 & 0 & 0 \\ 
0 & 0 & 0 & 0 & 0 & 0 & 0 & 0 & 0 & 0 & Q_{2} & 0 \\ 
0 & 0 & 0 & 0 & 0 & 0 & 0 & 0 & 0 & 0 & 0 & Q_{2} \\ 
0 & 0 & 0 & 0 & 0 & 0 & 0 & 0 & 1 & 0 & 0 & 0 \\ 
0 & 0 & 0 & 0 & 0 & 0 & 0 & 0 & 0 & 1 & 0 & 0 
\end{array}
\right) 
\end{equation*}

\vspace{5mm}

Hence the matrices of operators (1)--(12) in Section~\ref{subsec:typeG2} can be calculated (by using, e.g., SageMath) as follows: 

\vspace{5mm}

\noindent (1) $\sR_{2\alpha_{1}+\alpha_{2}} \sR_{3\alpha_{1}+2\alpha_{2}} \sR_{\alpha_{1}+\alpha_{2}} \sR_{\alpha_{2}} \sR_{-\alpha_{1}} \sR_{-3\alpha_{1}-\alpha_{2}}$: 
\begin{equation*}
{\fontsize{5pt}{0pt}
\left( 
\renewcommand{\arraystretch}{0.7}
\setlength{\arraycolsep}{2pt}
\begin{array}{cccccccccccc}
1 & Q_{1}Q_{2}-Q_{1} & Q_{2} & 0 & -Q_{1}Q_{2} & -Q_{1}Q_{2} & -Q_{1}Q_{2}^{2} & -Q_{1}Q_{2}^{2} & 0 & 0 & Q_{1}Q_{2}^{2} & Q_{1}^{2}Q_{2}^{3}-Q_{1}^{2}Q_{2}^{2} \\ 
-1 & 1 & 0 & Q_{2} & 0 & 0 & 0 & 0 & 0 & 0 & -Q_{1}Q_{2}^{2} & Q_{1}Q_{2}^{2} \\ 
1 & 0 & 1 & 0 & -Q_{1} & -Q_{1}Q_{2} & 0 & 0 & -Q_{1}Q_{2} & 0 & 0 & 0 \\ 
0 & 1 & 1 & 1 & -Q_{1} & -Q_{1} & 0 & 0 & -Q_{1}Q_{2} & -Q_{1}Q_{2} & 0 & 0 \\ 
0 & -1 & -1 & 0 & 1 & 0 & Q_{2} & 0 & 0 & 0 & 0 & -Q_{1}Q_{2}^{2} \\ 
0 & -1 & -1 & -1 & 1 & 1 & Q_{2} & Q_{2} & 0 & 0 & 0 & -Q_{1}Q_{2} \\ 
0 & 0 & 0 & 1 & -Q_{1}+1 & -Q_{1} & 1 & 0 & -Q_{1} & -Q_{1}Q_{2} & 0 & -Q_{1}Q_{2} \\ 
0 & 0 & 0 & 0 & -Q_{1}+1 & -Q_{1}+1 & 1 & 1 & -Q_{1} & -Q_{1} & 0 & -Q_{1}Q_{2} \\ 
0 & 0 & 0 & -1 & 0 & 1 & -1 & 0 & 1 & 0 & Q_{2} & 0 \\ 
0 & 0 & 0 & 0 & 0 & 0 & -1 & -1 &1 & 1 & 0 & Q_{2} \\ 
0 & 0 & 0 & 0 & 0 & 0 & -1 & -1 & 1 & 0 & 1 & Q_{1}Q_{2}-Q_{1} \\ 
0 & 0 & 0 & 0 & 0 & 0 & 0 & 0 & 0 & 1 & -1 & 1 
\end{array}
\right)}
\end{equation*}

\noindent (2) $\sR_{3\alpha_{1}+2\alpha_{2}} \sR_{\alpha_{1}+\alpha_{2}} \sR_{\alpha_{2}} \sR_{-\alpha_{1}} \sR_{-3\alpha_{1}-\alpha_{2}} \sR_{-2\alpha_{1}-\alpha_{2}}$: 
\begin{equation*}
{\fontsize{5pt}{0pt}
\left( 
\renewcommand{\arraystretch}{0.7}
\setlength{\arraycolsep}{2pt}
\begin{array}{cccccccccccc}
1 & Q_{1}Q_{2}-Q_{1} & Q_{2} & 0 & 0 & -Q_{1}Q_{2} & 0 & -Q_{1}Q_{2}^{2} & 0 & 0 & Q_{1}Q_{2}^{2} & Q_{1}^{2}Q_{2}^{3}-Q_{1}^{2}Q_{2}^{2} \\ 
-1 & 1 & 0 & Q_{2} & 0 & 0 & 0 & 0 & 0 & 0 & -Q_{1}Q_{2}^{2} & Q_{1}Q_{2}^{2} \\ 
1 & 0 & 1 & 0 & Q_{1}Q_{2}-Q_{1} & -Q_{1}Q_{2} & 0 & 0 & -Q_{1}Q_{2} & 0 & 0 & 0 \\ 
0 & 1 & 1 & 1 & 0 & -Q_{1} & 0 & 0 & -Q_{1}Q_{2} & -Q_{1}Q_{2} & 0 & 0 \\  
0 & -1 & -1 & 0 & 1 & 0 & Q_{2} & 0 & 0 & 0 & 0 & -Q_{1}Q_{2}^{2} \\ 
0 & 0 & 0 & -1 & -1 & 1 & -Q_{2} & Q_{2} & 0 & 0 & 0 & 0 \\ 
0 & 0 & 0 & 1 & 1 & -Q_{1} & 1 & 0 & -Q_{1} & -Q_{1}Q_{2} & 0 & -Q_{1}Q_{2} \\ 
0 & 0 & 0 & -1 & -1 & 1 & -1 & 1 & 0 & Q_{1}Q_{2}-Q_{1} & 0 & 0 \\ 
0 & 0 & 0 & -1 & -1 & 1 & -1 & 0 & 1 & 0 & Q_{2} & 0 \\ 
0 & 0 & 0 & 0 & 0 & 0 & 0 & -1 & 1 & 1 & 0 & Q_{2} \\ 
0 & 0 & 0 & 0 & 0 & 0 & 0 & -1 & 1 & 0 & 1 & Q_{1}Q_{2}-Q_{1} \\ 
0 & 0 & 0 & 0 & 0 & 0 & 0 & 0 & 0 & 1 & -1 & 1 
\end{array}
\right)}
\end{equation*}

\noindent (3) $\sR_{\alpha_{1}+\alpha_{2}} \sR_{\alpha_{2}} \sR_{-\alpha_{1}} \sR_{-3\alpha_{1}-\alpha_{2}} \sR_{-2\alpha_{1}-\alpha_{2}} \sR_{-3\alpha_{1}-2\alpha_{2}}$: 
\begin{equation*}
{\fontsize{5pt}{0pt}
\left( 
\renewcommand{\arraystretch}{0.7}
\setlength{\arraycolsep}{2pt}
\begin{array}{cccccccccccc}
1 & Q_{1}Q_{2}-Q_{1} & Q_{2} & 0 & 0 & Q_{1}Q_{2}^{2}-Q_{1}Q_{2} & 0 & 0 & -Q_{1}Q_{2}^{2} & 0 & -Q_{1}Q_{2}^{2} & -Q_{1}^{2}Q_{2}^{3}+Q_{1}^{2}Q_{2}^{2} \\ 
-1 & 1 & 0 & Q_{2} & 0 & 0 & 0 & 0 & 0 & -Q_{1}Q_{2}^{2} & Q_{1}Q_{2}^{2} & -Q_{1}Q_{2}^{2} \\ 
1 & 0 & 1 & 0 & Q_{1}Q_{2}-Q_{1} & 0 & 0 & 0 & -Q_{1}Q_{2} & 0 & -Q_{1}Q_{2}^{2} & 0 \\ 
0 & 1 & 1 & 1 & 0 & Q_{1}Q_{2}-Q_{1} & 0 & 0 & -Q_{1}Q_{2} & -Q_{1}Q_{2} & 0 & -Q_{1}Q_{2}^{2} \\ 
0 & -1 & -1 & -Q_{2} & 1 & 0 & Q_{2} & 0 & 0 & 0 & 0 & 0 \\ 
0 & 0 & 0 & Q_{2}-1 & -1 & 1 & -Q_{2} & Q_{2} & 0 & 0 & 0 & 0 \\ 
0 & -1 & -1 & -1 & 1 & -Q_{1}Q_{2}+Q_{1} & 1 & 0 & Q_{1}Q_{2}-Q_{1} & 0 & 0 & Q_{1}Q_{2}^{2}-Q_{1}Q_{2} \\ 
0 & 0 & 0 & 0 & -1 & 1 & -1 & 1 & 0 & Q_{1}Q_{2}-Q_{1} & 0 & 0 \\ 
0 & 0 & 0 & -Q_{2}+1 & 0 & -1 & Q_{2}-1 & -Q_{2} & 1 & 0 & Q_{2} & 0 \\ 
0 & 0 & 0 & 0 & 0 & -1 & 0 & -1 & 1 & 1 & 0 & Q_{2} \\ 
0 & 0 & 0 & 0 & 0 & -1 & 0 & -1 & 1 & 0 & 1 & Q_{1}Q_{2}-Q_{1} \\ 
0 & 0 & 0 & 0 & 0 & 0 & 0 & 0 & 0 & 1 & -1 & 1 
\end{array}
\right)}
\end{equation*}

\noindent (4) $\sR_{3\alpha_{1}+2\alpha_{2}} \sR_{2\alpha_{1}+\alpha_{2}} \sR_{3\alpha_{1}+\alpha_{2}} \sR_{\alpha_{1}} \sR_{-\alpha_{2}} \sR_{-\alpha_{1}-\alpha_{2}}$: 
\begin{equation*}
{\fontsize{5pt}{0pt}
\left( 
\renewcommand{\arraystretch}{0.7}
\setlength{\arraycolsep}{2pt}
\begin{array}{cccccccccccc}
1 & -Q_{1}Q_{2}+Q_{1} & -Q_{2} & 0 & 0 & -Q_{1}Q_{2}^{2}+Q_{1}Q_{2} & 0 & 0 & -Q_{1}Q_{2}^{2} & 0 & Q_{1}Q_{2}^{2} & -Q_{1}^{2}Q_{2}^{3} + Q_{1}^{2}Q_{2}^{2} \\ 
1 & 1 & 0 & -Q_{2} & 0 & 0 & 0 & 0 & 0 & -Q_{1}Q_{2}^{2} & Q_{1}Q_{2}^{2} & Q_{1}Q_{2}^{2} \\ 
-1 & 0 & 1 & 0 & -Q_{1}Q_{2}+Q_{1} & 0 & 0 & 0 & Q_{1}Q_{2} & 0 & -Q_{1}Q_{2}^{2} & 0 \\ 
0 & -1 & -1 & 1 & 0 & -Q_{1}Q_{2}+Q_{1} & 0 & 0 & -Q_{1}Q_{2} & Q_{1}Q_{2} & 0 & -Q_{1}Q_{2}^{2} \\ 
0 & 1 & 1 & -Q_{2} & 1 & 0 & -Q_{2} & 0 & 0 & 0 & 0 & 0 \\ 
0 & 0 & 0 & -Q_{2}+1 & 1 & 1 & -Q_{2} & -Q_{2} & 0 & 0 & 0 & 0 \\ 
0 & -1 & -1 & 1 & -1 & -Q_{1}Q_{2}+Q_{1} & 1 & 0 & -Q_{1}Q_{2}+Q_{1} & 0 & 0 & -Q_{1}Q_{2}^{2}+Q_{1}Q_{2} \\ 
0 & 0 & 0 & 0 & -1 & -1 & 1 & 1 & 0 & -Q_{1}Q_{2}+Q_{1} & 0 & 0 \\ 
0 & 0 & 0 & -Q_{2}+1 & 0 & 1 & -Q_{2}+1 & -Q_{2} & 1 & 0 & -Q_{2} & 0 \\ 
0 & 0 & 0 & 0 & 0 & -1 & 0 & 1 & -1 & 1 & 0 & -Q_{2} \\ 
0 & 0 & 0 & 0 & 0 & -1 & 0 & 1 & -1 & 0 & 1 & -Q_{1}Q_{2}+Q_{1} \\ 
0 & 0 & 0 & 0 & 0 & 0 & 0 & 0 & 0 & -1 & 1 & 1 
\end{array}
\right)}
\end{equation*}

\noindent (5) $\sR_{2\alpha_{1}+\alpha_{2}} \sR_{3\alpha_{1}+\alpha_{2}} \sR_{\alpha_{1}} \sR_{-\alpha_{2}} \sR_{-\alpha_{1}-\alpha_{2}} \sR_{-3\alpha_{1}-2\alpha_{2}}$: 
\begin{equation*}
{\fontsize{5pt}{0pt}
\left( 
\renewcommand{\arraystretch}{0.7}
\setlength{\arraycolsep}{2pt}
\begin{array}{cccccccccccc}
1 & -Q_{1}Q_{2}+Q_{1} & -Q_{2} & 0 & 0 & Q_{1}Q_{2} & 0 & -Q_{1}Q_{2}^{2} & 0 & 0 & -Q_{1}Q_{2}^{2} & Q_{1}^{2}Q_{2}^{3}-Q_{1}^{2}Q_{2}^{2} \\ 
1 & 1 & 0 & -Q_{2} & 0 & 0 & 0 & 0 & 0 & 0 & -Q_{1}Q_{2}^{2} & -Q_{1}Q_{2}^{2} \\ 
-1 & 0 & 1 & 0 & -Q_{1}Q_{2}+Q_{1} & -Q_{1}Q_{2} & 0 & 0 & Q_{1}Q_{2} & 0 & 0 & 0 \\ 
0 & -1 & -1 & 1 & 0 & Q_{1} & 0 & 0 & -Q_{1}Q_{2} & Q_{1}Q_{2} & 0 & 0 \\ 
0 & 1 & 1 & 0 & 1 & 0 & -Q_{2} & 0 & 0 & 0 & 0 & -Q_{1}Q_{2}^{2} \\ 
0 & 0 & 0 & 1 & 1 & 1 & -Q_{2} & -Q_{2} & 0 & 0 & 0 & 0 \\ 
0 & 0 & 0 & -1 & -1 & -Q_{1} & 1 & 0 & Q_{1} & -Q_{1}Q_{2} & 0 & Q_{1}Q_{2} \\ 
0 & 0 & 0 & -1 & -1 & -1 & 1 & 1 & 0 & -Q_{1}Q_{2}+Q_{1} & 0 & 0 \\ 
0 & 0 & 0 & -1 & -1 & -1 & 1 & 0 & 1 & 0 & -Q_{2} & 0 \\ 
0 & 0 & 0 & 0 & 0 & 0 & 0 & 1 & -1 & 1 & 0 & -Q_{2} \\ 
0 & 0 & 0 & 0 & 0 & 0 & 0 & 1 & -1 & 0 & 1 & -Q_{1}Q_{2}+Q_{1} \\ 
0 & 0 & 0 & 0 & 0 & 0 & 0 & 0 & 0 & -1 & 1 & 1 
\end{array}
\right)}
\end{equation*}

\noindent (6) $\sR_{3\alpha_{1}+\alpha_{2}} \sR_{\alpha_{1}} \sR_{-\alpha_{2}} \sR_{-\alpha_{1}-\alpha_{2}} \sR_{-3\alpha_{1}-2\alpha_{2}} \sR_{-2\alpha_{1}-\alpha_{2}}$: 
\begin{equation*}
{\fontsize{5pt}{0pt}
\left( 
\renewcommand{\arraystretch}{0.7}
\setlength{\arraycolsep}{2pt}
\begin{array}{cccccccccccc}
1 & -Q_{1}Q_{2}+Q_{1} & -Q_{2} & 0 & -Q_{1}Q_{2} & Q_{1}Q_{2} & Q_{1}Q_{2}^{2} & -Q_{1}Q_{2}^{2} & 0 & 0 & -Q_{1}Q_{2}^{2} & Q_{1}^{2}Q_{2}^{3}-Q_{1}^{2}Q_{2}^{2} \\ 
1 & 1 & 0 & -Q_{2} & 0 & 0 & 0 & 0 & 0 & 0 & -Q_{1}Q_{2}^{2} & -Q_{1}Q_{2}^{2} \\ 
-1 & 0 & 1 & 0 & Q_{1} & -Q_{1}Q_{2} & 0 & 0 & Q_{1}Q_{2} & 0 & 0 & 0 \\ 
0 & -1 & -1 & 1 & -Q_{1} & Q_{1} & 0 & 0 & -Q_{1}Q_{2} & Q_{1}Q_{2} & 0 & 0 \\ 
0 & 1 & 1 & 0 & 1 & 0 & -Q_{2} & 0 & 0 & 0 & 0 & -Q_{1}Q_{2}^{2} \\ 
0 & -1 & -1 & 1 & -1 & 1 & Q_{2} & -Q_{2} & 0 & 0 & 0 & Q_{1}Q_{2}^{2} \\ 
0 & 0 & 0 & -1 & Q_{1}-1 & -Q_{1} & 1 & 0 & Q_{1} & -Q_{1}Q_{2} & 0 & Q_{1}Q_{2} \\ 
0 & 0 & 0 & 0 & -Q_{1}+1 & Q_{1}-1 & -1 & 1 & -Q_{1} & Q_{1} & 0 & -Q_{1}Q_{2} \\ 
0 & 0 & 0 & -1 & 0 & -1 & 1 & 0 & 1 & 0 & -Q_{2} & 0 \\ 
0 & 0 & 0 & 0 & 0 & 0 & -1 & 1 & -1 & 1 & 0 & -Q_{2} \\ 
0 & 0 & 0 & 0 & 0 & 0 & -1 & 1 & -1 & 0 & 1 & -Q_{1}Q_{2}+Q_{1} \\ 
0 & 0 & 0 & 0 & 0 & 0 & 0 & 0 & 0 & -1 & 1 & 1 
\end{array}
\right)}
\end{equation*}

\noindent (7) $\sR_{2\alpha_{1}+\alpha_{2}} \sR_{3\alpha_{1}+2\alpha_{2}} \sR_{\alpha_{1}+\alpha_{2}} \sR_{\alpha_{2}} \sR_{\alpha_{1}} \sR_{3\alpha_{1}+\alpha_{2}}$: 
\begin{equation*}
{\fontsize{5pt}{0pt}
\left( 
\renewcommand{\arraystretch}{0.7}
\setlength{\arraycolsep}{2pt}
\begin{array}{cccccccccccc}
1 & Q_{1}Q_{2}+Q_{1} & Q_{2} & 0 & Q_{1}Q_{2} & Q_{1}Q_{2} & Q_{1}Q_{2}^{2} & Q_{1}Q_{2}^{2} & 2Q_{1}Q_{2}^{2} & 0 & Q_{1}Q_{2}^{2} & Q_{1}^{2}Q_{2}^{3}+Q_{1}^{2}Q_{2}^{2} \\ 
1 & 1 & 0 & Q_{2} & 0 & 2Q_{1}Q_{2} & 0 & 2Q_{1}Q_{2}^{2} & 0 & 2Q_{1}Q_{2}^{2} & Q_{1}Q_{2}^{2} & Q_{1}Q_{2}^{2} \\ 
1 & 2Q_{1} & 1 & 0 & Q_{1} & Q_{1}Q_{2} & 0 & 0 & Q_{1}Q_{2} & 0 & 0 & 0 \\ 
2 & 2Q_{1}+1 & 1 & 1 & Q_{1} & 2Q_{1}Q_{2}+Q_{1} & 0 & 0 & Q_{1}Q_{2} & Q_{1}Q_{2} & 0 & 0 \\ 
0 & 1 & 1 & 0 & 1 & 0 & Q_{2} & 0 & 2Q_{1}Q_{2} & 0 & 0 & Q_{1}Q_{2}^{2} \\ 
0 & 1 & 1 & 1 & 1 & 1 & Q_{2} & Q_{2} & 2Q_{1}Q_{2} & 2 Q_{1}Q_{2} & 0 & Q_{1}Q_{2}^{2} \\ 
2 & 2Q_{1}+2 & 2 & 1 & Q_{1}+1 & 2Q_{1}Q_{2}+Q_{1} & 1 & 0 & 2Q_{1}Q_{2}+Q_{1} & Q_{1}Q_{2} & 0 & Q_{1}Q_{2} \\ 
2 & 2Q_{1}+2 & 2 & 2 & Q_{1}+1 &  2Q_{1}Q_{2}+Q_{1}+1 & 1 & 1 & 2Q_{1}Q_{2}+Q_{1} & 2Q_{1}Q_{2}+Q_{1} & 0 & Q_{1}Q_{2} \\ 
0 & 0 & 0 & 1 & 0 & 1 & 1 & 2Q_{2} & 1 & 2Q_{1}Q_{2} & Q_{2} & 2Q_{1}Q_{2} \\ 
0 & 0 & 0 & 0 & 0 & 0 & 1 & 2Q_{2}+1 & 1 & 1 & 2Q_{2} & 2Q_{1}Q_{2}+Q_{2} \\ 
0 & 0 & 0 & 0 & 0 & 0 & 1 & 1 & 1 & 0 & 1 & Q_{1}Q_{2}+Q_{1} \\ 
0 & 0 & 0 & 0 & 0 & 0 & 0 & 2 & 0 & 1 & 1 & 1 
\end{array}
\right)}
\end{equation*}

\noindent (8) $\sR_{3\alpha_{1}+2\alpha_{2}} \sR_{\alpha_{1}+\alpha_{2}} \sR_{\alpha_{2}} \sR_{\alpha_{1}} \sR_{3\alpha_{1}+\alpha_{2}} \sR_{2\alpha_{1}+\alpha_{2}}$: 
\begin{equation*}
{\fontsize{5pt}{0pt}
\left( 
\renewcommand{\arraystretch}{0.7}
\setlength{\arraycolsep}{2pt}
\begin{array}{cccccccccccc}
1 & Q_{1}Q_{2}+Q_{1} & Q_{2} & 0 & 2Q_{1}Q_{2} & Q_{1}Q_{2} & 2Q_{1}Q_{2}^{2} & Q_{1}Q_{2}^{2} & 2Q_{1}Q_{2}^{2} & 0 & Q_{1}Q_{2}^{2} & Q_{1}^{2}Q_{2}^{3}+Q_{1}^{2}Q_{2}^{2} \\ 
1 & 1 & 0 & Q_{2} & 2Q_{1}Q_{2} & 2Q_{1}Q_{2} & 2Q_{1}Q_{2}^{2} & 2Q_{1}Q_{2}^{2} & 0 & 2Q_{1}Q_{2}^{2} & Q_{1}Q_{2}^{2} & Q_{1}Q_{2}^{2} \\ 
1 & 2Q_{1} & 1 & 0 & Q_{1}Q_{2}+Q_{1} & Q_{1}Q_{2} & 0 & 0 & Q_{1}Q_{2} & 0 & 0 & 0 \\ 
2 & 2Q_{1}+1 & 1 & 1 & 2Q_{1}Q_{2}+2Q_{1} & 2Q_{1}Q_{2}+Q_{1} & 0 & 0 & Q_{1}Q_{2} & Q_{1}Q_{2} & 0 & 0 \\ 
0 & 1 & 1 & 0 & 1 & 0 & Q_{2} & 0 & 2Q_{1}Q_{2} & 0 & 0 & Q_{1}Q_{2}^{2} \\ 
0 & 0 & 0 & 1 & 1 & 1 & Q_{2} & Q_{2} & 0 & 2Q_{1}Q_{2} & 0 & 0 \\ 
2 & 2Q_{1}+2 & 2 & 1 & 2Q_{1}Q_{2}+2Q_{1}+1 & 2Q_{1}Q_{2}+Q_{1} & 1 & 0 & 2Q_{1}Q_{2}+Q_{1} & Q_{1}Q_{2} & 0 & Q_{1}Q_{2} \\ 
0 & 0 & 0 & 1 & 1 & 1 & 1 & 1 & 0 & Q_{1}Q_{2}+Q_{1} & 0 & 0 \\ 
0 & 0 & 0 & 1 & 1 & 1 & 2Q_{2}+1 & 2Q_{2} & 1 & 2Q_{1}Q_{2} & Q_{2} & 2Q_{1}Q_{2} \\ 
0 & 0 & 0 & 0 & 0 & 0 & 2Q_{2}+2 & 2Q_{2}+1 & 1 & 1 & 2Q_{2} & 2Q_{1}Q_{2}+Q_{2} \\ 
0 & 0 & 0 & 0 & 0 & 0 & 2 & 1 & 1 & 0 & 1 & Q_{1}Q_{2}+Q_{1} \\ 
0 & 0 & 0 & 0 & 0 & 0 & 2 & 2 & 0 & 1 & 1 & 1 
\end{array}
\right)}
\end{equation*}

\noindent (9) $\sR_{\alpha_{1}+\alpha_{2}} \sR_{\alpha_{2}} \sR_{\alpha_{1}} \sR_{3\alpha_{1}+\alpha_{2}} \sR_{2\alpha_{1}+\alpha_{2}} \sR_{3\alpha_{1}+2\alpha_{2}}$: 
\begin{equation*}
{\fontsize{5pt}{0pt}
\left( 
\renewcommand{\arraystretch}{0.7}
\setlength{\arraycolsep}{2pt}
\begin{array}{cccccccccccc}
1 & Q_{1}Q_{2}+Q_{1} & Q_{2} & 0 & 2Q_{1}Q_{2} & Q_{1}Q_{2}^{2}+Q_{1}Q_{2} & 0 & 0 & Q_{1}Q_{2}^{2} & 0 & Q_{1}Q_{2}^{2} & Q_{1}^{2}Q_{2}^{3}+Q_{1}^{2}Q_{2}^{2} \\ 
1 & 1 & 0 & Q_{2} & 2Q_{1}Q_{2} & 2Q_{1}Q_{2} & 0 & 0 & 0 & Q_{1}Q_{2}^{2} & Q_{1}Q_{2}^{2} & Q_{1}Q_{2}^{2} \\ 
1 & 2Q_{1} & 1 & 0 & Q_{1}Q_{2}+Q_{1} & 2Q_{1}Q_{2} & 0 & 0 & Q_{1}Q_{2} & 0 & Q_{1}Q_{2}^{2} & 2Q_{1}^{2}Q_{2}^{2} \\ 
2 & 2Q_{1}+1 & 1 & 1 & 2Q_{1}Q_{2}+2Q_{1} & 3Q_{1}Q_{2}+Q_{1} & 0 & 0 & Q_{1}Q_{2} & Q_{1}Q_{2} & 2Q_{1}Q_{2}^{2} & 2Q_{1}^{2}Q_{2}^{2}+Q_{1}Q_{2}^{2} \\ 
0 & 1 & 1 & Q_{2} & 1 & 2Q_{1}Q_{2} & Q_{2} & 0 & 2Q_{1}Q_{2} & 0 & 0 & 2Q_{1}Q_{2}^{2} \\ 
0 & 0 & 0 & Q_{2}+1 & 1 & 1 & Q_{2} & Q_{2} & 0 & 2Q_{1}Q_{2} & 0 & 0 \\ 
0 & 1 & 1 & 1 & 1 & Q_{1}Q_{2}+Q_{1} & 1 & 0 & Q_{1}Q_{2}+Q_{1} & 0 & 0 & Q_{1}Q_{2}^{2}+Q_{1}Q_{2} \\ 
0 & 0 & 0 & 2 & 1 & 1 & 1 & 1 & 0 & Q_{1}Q_{2}+Q_{1} & 0 & 0 \\ 
0 & 0 & 0 & Q_{2}+1 & 0 & 1 & Q_{2}+1 & Q_{2} & 1 & 0 & Q_{2} & 2Q_{1}Q_{2} \\ 
0 & 0 & 0 & 2Q_{2}+2 & 0 & 1 & 2Q_{2}+2 & 2Q_{2}+1 & 1 & 1 & 2Q_{2} & 2Q_{1}Q_{2}+Q_{2} \\ 
0 & 0 & 0 & 2 & 0 & 1 & 2 & 1 & 1 & 0 & 1 & Q_{1}Q_{2}+Q_{1} \\ 
0 & 0 & 0 & 2 & 0 & 0 & 2 & 2 & 0 & 1 & 1 & 1 
\end{array}
\right)}
\end{equation*}

\noindent (10) $\sR_{3\alpha_{1}+2\alpha_{2}} \sR_{2\alpha_{1}+\alpha_{2}} \sR_{3\alpha_{1}+\alpha_{2}} \sR_{\alpha_{1}} \sR_{\alpha_{2}} \sR_{\alpha_{1}+\alpha_{2}}$: 
\begin{equation*}
{\fontsize{5pt}{0pt}
\left( 
\renewcommand{\arraystretch}{0.7}
\setlength{\arraycolsep}{2pt}
\begin{array}{cccccccccccc}
1 & Q_{1}Q_{2}+Q_{1} & 2Q_{1}Q_{2}+Q_{2} & 2Q_{1}Q_{2} & 0 & Q_{1}Q_{2}^{2}+Q_{1}Q_{2} & 0 & 2Q_{1}Q_{2}^{2} & 2Q_{1}^{2}Q_{2}^{2}+Q_{1}Q_{2}^{2} & 2Q_{1}^{2}Q_{2}^{2} & Q_{1}Q_{2}^{2} & Q_{1}^{2}Q_{2}^{3}+Q_{1}^{2}Q_{2}^{2} \\ 
1 & 1 & 2Q_{2} & Q_{2} & 0 & 0 & 0 & 0 & 2Q_{1}Q_{2}^{2} & Q_{1}Q_{2}^{2} & Q_{1}Q_{2}^{2} & Q_{1}Q_{2}^{2} \\ 
1 & 0 & 1 & 0 & Q_{1}Q_{2}+Q_{1} & 0 & 2Q_{1}Q_{2} & 0 & Q_{1}Q_{2} & 0 & Q_{1}Q_{2}^{2} & 0 \\ 
0 & 1 & 1 & 1 & 0 & Q_{1}Q_{2}+Q_{1} & 0 & 2Q_{1}Q_{2} & Q_{1}Q_{2} & Q_{1}Q_{2} & 0 & Q_{1}Q_{2}^{2} \\ 
2 & 1 & 2Q_{2}+1 & Q_{2} & 1 & 0 & Q_{2} & 0 & 0 & 0 & 0 & 0 \\ 
2 & 2 & 2Q_{2}+2 & Q_{2}+1 & 1 & 1 & Q_{2} & Q_{2} & 0 & 0 & 0 & 0 \\ 
0 & 1 & 1 & 1 & 1 & Q_{1}Q_{2}+Q_{1} & 1 & 2Q_{1}Q_{2} & 3Q_{1}Q_{2}+Q_{1} & 2Q_{1}Q_{2} & 2Q_{1}Q_{2} & Q_{1}Q_{2}^{2}+Q_{1}Q_{2} \\ 
0 & 0 & 0 & 0 & 1 & 1 & 1 & 1 & 2Q_{1}Q_{2}+2Q_{1} & Q_{1}Q_{2}+Q_{1} & 2Q_{1}Q_{2} & 2Q_{1}Q_{2} \\ 
2 & 2 & 2Q_{2}+2 & Q_{2}+1 & 2 & 1 & Q_{2}+1 & Q_{2} & 1 & 0 & Q_{2} & 0 \\ 
0 & 0 & 0 & 0 & 0 & 1 & 0 & 1 & 1 & 1 & 0 & Q_{2} \\ 
0 & 0 & 0 & 0 & 0 & 1 & 0 & 1 & 2Q_{1}+1 & 2Q_{1} & 1 & Q_{1}Q_{2}+Q_{1} \\ 
0 & 0 & 0 & 0 & 0 & 0 & 0 & 0 & 2 & 1 & 1 & 1 
\end{array}
\right)}
\end{equation*}

\noindent (11) $\sR_{2\alpha_{1}+\alpha_{2}} \sR_{3\alpha_{1}+\alpha_{2}} \sR_{\alpha_{1}} \sR_{\alpha_{2}} \sR_{\alpha_{1}+\alpha_{2}} \sR_{3\alpha_{1}+2\alpha_{2}}$: 
\begin{equation*}
{\fontsize{5pt}{0pt}
\left( 
\renewcommand{\arraystretch}{0.7}
\setlength{\arraycolsep}{2pt}
\begin{array}{cccccccccccc}
1 & Q_{1}Q_{2}+Q_{1} & 2Q_{1}Q_{2}+Q_{2} & 2Q_{1}Q_{2} & 0 & Q_{1}Q_{2} & 0 & Q_{1}Q_{2}^{2} & 0 & 0 & Q_{1}Q_{2}^{2} & Q_{1}^{2}Q_{2}^{3}+Q_{1}^{2}Q_{2}^{2} \\ 
1 & 1 & 2Q_{2} & Q_{2} & 0 & 0 & 0 & 0 & 0 & 0 & Q_{1}Q_{2}^{2} & Q_{1}Q_{2}^{2} \\ 
1 & 0 & 1 & 2Q_{1}Q_{2} & Q_{1}Q_{2}+Q_{1} & Q_{1}Q_{2} & 2Q_{1}Q_{2} & 0 & Q_{1}Q_{2} & 0 & 2Q_{1}Q_{2}^{2} & 0 \\ 
0 & 1 & 1 & 1 & 0 & 2Q_{1}Q_{2}+Q_{1} & 0 & 2Q_{1}Q_{2} & Q_{1}Q_{2} & Q_{1}Q_{2} & 0 & 2Q_{1}Q_{2}^{2} \\ 
2 & 1 & 2Q_{2}+1  &2Q_{2} & 1 & 0 & Q_{2} & 0 & 0 & 0 & 2Q_{1}Q_{2}^{2} & Q_{1}Q_{2}^{2} \\ 
2 & 2 & 2Q_{2}+2 & 2Q_{2}+1 & 1 & 1 & Q_{2} & Q_{2} & 0 & 0 & 2Q_{1}Q_{2}^{2} & 2Q_{1}Q_{2}^{2} \\ 
0 & 0 & 0 & 1 & 1 & 2Q_{1}Q_{2}+Q_{1} & 1 & 0 & 2Q_{1}Q_{2}+Q_{1} & Q_{1}Q_{2} & 2Q_{1}Q_{2} & Q_{1}Q_{2} \\ 
0 & 0 & 0 & 1 & 1 & 2Q_{1}Q_{2}+2Q_{1}+1 & 1 & 1 & 2Q_{1}Q_{2}+2Q_{1} & Q_{1}Q_{2}+Q_{1} & 2Q_{1}Q_{2} & 2Q_{1}Q_{2} \\ 
0 & 0 & 0 & 1 & 1 & 1 & 1 & 0 & 1 & 0 & Q_{2} & 0 \\ 
0 & 0 & 0 & 0 & 0 & 2 & 0 & 1 & 1 & 1 & 0 & Q_{2} \\ 
0 & 0 & 0 & 0 & 0 & 2Q_{1}+2 & 0 & 1 & 2Q_{1}+1 & 2Q_{1} & 1 & Q_{1}Q_{2}+Q_{1} \\ 
0 & 0 & 0 & 0 & 0 & 2 & 0 & 0 & 2 & 1 & 1 & 1 
\end{array}
\right)}
\end{equation*}

\noindent (12) $\sR_{3\alpha_{1}+\alpha_{2}} \sR_{\alpha_{1}} \sR_{\alpha_{2}} \sR_{\alpha_{1}+\alpha_{2}} \sR_{3\alpha_{1}+2\alpha_{2}} \sR_{2\alpha_{1}+\alpha_{2}}$: 
\begin{equation*}
{\fontsize{5pt}{0pt}
\left( 
\renewcommand{\arraystretch}{0.7}
\setlength{\arraycolsep}{2pt}
\begin{array}{cccccccccccc}
1 & Q_{1}Q_{2}+Q_{1} & 2Q_{1}Q_{2}+Q_{2} & 2Q_{1}Q_{2} & Q_{1}Q_{2} & Q_{1}Q_{2} & Q_{1}Q_{2}^{2} & Q_{1}Q_{2}^{2} & 0 & 0 & Q_{1}Q_{2}^{2} & Q_{1}^{2}Q_{2}^{3}+Q_{1}^{2}Q_{2}^{2} \\ 
1 & 1 & 2Q_{2} & Q_{2} & 0 & 0 & 0 & 0 & 0 & 0 & Q_{1}Q_{2}^{2} & Q_{1}Q_{2}^{2} \\ 
1 & 0 & 1 & 2Q_{1}Q_{2} & 2Q_{1}Q_{2}+Q_{1} & Q_{1}Q_{2} & 2Q_{1}Q_{2} & 0 & Q_{1}Q_{2} & 0 & 2Q_{1}Q_{2}^{2} & 0 \\ 
0 & 1 & 1 & 1 & 2Q_{1}Q_{2}+Q_{1} & 2Q_{1}Q_{2}+Q_{1} & 2Q_{1}Q_{2} & 2Q_{1}Q_{2} & Q_{1}Q_{2} & Q_{1}Q_{2} & 0 & 2Q_{1}Q_{2}^{2} \\ 
2 & 1 & 2Q_{2}+1 & 2Q_{2} & 1 & 0 & Q_{2} & 0 & 0 & 0 & 2Q_{1}Q_{2}^{2} & Q_{1}Q_{2}^{2} \\ 
0 & 1 & 1 & 1 & 1 & 1 & 	Q_{2} & Q_{2} & 0 & 0 & 0 & Q_{1}Q_{2}^{2} \\ 
0 & 0 & 0 & 1 & 2Q_{1}Q_{2}+Q_{1}+1 & 2Q_{1}Q_{2}+Q_{1} & 1 & 0 & 2Q_{1}Q_{2}+Q_{1} & Q_{1}Q_{2} & 2Q_{1}Q_{2} & Q_{1}Q_{2} \\ 
0 & 0 & 0 & 0 & Q_{1}+1 & Q_{1}+1 & 1 & 1 & Q_{1} & Q_{1} & 0 & Q_{1}Q_{2} \\ 
0 & 0 & 0 & 1 & 2 & 1 & 1 & 0 & 1 & 0 & Q_{2} & 0 \\ 
0 & 0 & 0 & 0 & 2 & 2 & 1 & 1 & 1 & 1 & 0 & Q_{2} \\ 
0 & 0 & 0 & 0 & 2Q_{1}+2 & 2Q_{1}+2 & 1 & 1 & 2Q_{1}+1 & 2Q_{1} & 1 & Q_{1}Q_{2}+Q_{1} \\ 
0 & 0 & 0 & 0 & 2 & 2 & 0 & 0 & 2 & 1 & 1 & 1 
\end{array}
\right)}
\end{equation*}

%===================%
% START SECTION 0B %
%===================%

%%%%%%%%%%%%
% sec:App_B %
%%%%%%%%%%%%

\section{Example of quantum Yang-Baxter moves in type \texorpdfstring{$C_{2}$}{C2}.}\label{sec:App_B}

Based on Proposition~\ref{prop:rank2_shellability}, we explain how to construct quantum Yang-Baxter moves explicitly in a specific case. 
We assume that $\Fg$ is of type $C_{2}$. 
Let $\Pi$, $\Pi'$ be sequences of roots introduced in Section~\ref{subsec:rank2_shellability}. 
We consider the case that $v = s_{2}$ and $\Pi = (-2\alpha_{1}-\alpha_{2}, -\alpha_{1}, \alpha_{2}, \alpha_{1}+\alpha_{2})$. 
Note that $\Pi' = (\alpha_{1}+\alpha_{2}, \alpha_{2}, -\alpha_{1}, -2\alpha_{1}-\alpha_{2})$. 
Let us construct an explicit matching between a certain subset of $\CP(v, \Pi)$ and that of $\CP(v, \Pi')$, 
and also sign-reversing involutions outside of those subsets. 

Recall the matrices of the operators $\sR_{\alpha_{1}+\alpha_{2}}\sR_{\alpha_{2}}\sR_{\alpha_{1}}\sR_{2\alpha_{1}+\alpha_{2}}$ 
and $\sR_{2\alpha_{1}+\alpha_{2}}\sR_{\alpha_{1}}\sR_{\alpha_{2}}\sR_{\alpha_{1}+\alpha_{2}}$ calculated in Section~\ref{subsec:matrix_typeC2}. 
In particular, the $v$-column of the matrix of the operator $\sR_{\alpha_{1}+\alpha_{2}}\sR_{\alpha_{2}}\sR_{\alpha_{1}}\sR_{2\alpha_{1}+\alpha_{2}}$ 
(resp., $\sR_{2\alpha_{1}+\alpha_{2}}\sR_{\alpha_{1}}\sR_{\alpha_{2}}\sR_{\alpha_{1}+\alpha_{2}}$) is ${}^{t}(Q_{2}, 0, 1, 1, 1, 1, 1, 0)$ 
(resp., ${}^{t}(2Q_{1}Q_{2}+Q_{2}, 2Q_{2}, 1, 1, 2Q_{2}+1, 1, 1, 0)$). 
For example, the $(e, v)$-entry of the matrix of the operator $\sR_{2\alpha_{1}+\alpha_{2}}\sR_{\alpha_{1}}\sR_{\alpha_{2}}\sR_{\alpha_{1}+\alpha_{2}}$ 
is $2Q_{1}Q_{2}+Q_{2}$. 
Therefore, we deduce from equation \eqref{eq:prod_quantum_Bruhat_op_abs} that there exist exactly three $\Pi'$-compatible directed paths $\br^{(1)}, \br^{(2)}, \br^{(3)}$ 
such that 
\begin{itemize}
\item $\br^{(j)}$ starts at $v = s_{2}$ for $j = 1, 2, 3$, 
\item $\ed(\br^{(j)}) = e$, $j = 1, 2, 3$, 
\item $\down(\br^{(j)}) = \alpha_{1}^{\vee} + \alpha_{2}^{\vee}$, $j = 1, 2$, and 
\item $\down(\br^{(3)}) = \alpha_{2}^{\vee}$; 
\end{itemize}
remark that $Q^{\alpha_{1}^{\vee}+\alpha_{2}^{\vee}} = Q_{1}Q_{2}$ and $Q^{\alpha_{2}^{\vee}} = Q_{2}$. 
Similarly, we see that there exist six $\Pi$-compatible directed paths 
$\bp_{1}, \ldots, \bp_{6}$ such that $\CP(v, \Pi) = \{ \bp_{1}, \ldots, \bp_{6} \}$. 
Also, there exist twelve $\Pi'$-compatible directed paths 
$\bq_{1}, \ldots, \bq_{12}$ such that $\CP(v, \Pi') = \{ \bq_{1}, \ldots, \bq_{12} \}$. 
For $\bp \in \CP(v, \Pi)$ (resp., $\bq \in \CP(v, \Pi')$), the statistics $\ed(\bp)$, $\down(\bp)$ (resp., $\ed(\bq)$, $\down(\bq)$) 
are given in Tables~\ref{tab:end+down1}, \ref{tab:end+down2}. 

\begin{table}[htbp]
\begin{minipage}[t]{0.5\hsize}
\centering
\caption{Statistics of elements $\bp \in \CP(v, \Pi)$}
\label{tab:end+down1}
\begin{tabular}{|c||cc|} \hline 
$\bp \in \CP(v, \Pi)$ & $\ed(\bp)$ & $\down(\bp)$ \\ \hline
$\bp_{1}$ & $e$ & $\alpha_{2}^{\vee}$ \\ 
$\bp_{2}$ & $s_{2}$ & $0$ \\ 
$\bp_{3}$ & $s_{1}s_{2}$ & $0$ \\ 
$\bp_{4}$ & $s_{2}s_{1}$ & $0$ \\ 
$\bp_{5}$ & $s_{1}s_{2}s_{1}$ & $0$ \\ 
$\bp_{6}$ & $s_{2}s_{1}s_{2}$ & $0$ \\ \hline
\end{tabular}
\end{minipage}
\begin{minipage}[t]{0.5\hsize}
\centering
\caption{Statistics of elements $\bq \in \CP(v, \Pi')$}
\label{tab:end+down2}
\begin{tabular}{|c||cc|} \hline 
$\bq \in \CP(v, \Pi')$ & $\ed(\bq)$ & $\down(\bq)$ \\ \hline 
$\bq_{1}$ & $e$ & $\alpha_{1}^{\vee} + \alpha_{2}^{\vee}$ \\ 
$\bq_{2}$ & $e$ & $\alpha_{1}^{\vee} + \alpha_{2}^{\vee}$ \\ 
$\bq_{3}$ & $e$ & $\alpha_{2}^{\vee}$ \\ 
$\bq_{4}$ & $s_{1}$ & $\alpha_{2}^{\vee}$ \\ 
$\bq_{5}$ & $s_{1}$ & $\alpha_{2}^{\vee}$ \\ 
$\bq_{6}$ & $s_{2}$ & $0$ \\ 
$\bq_{7}$ & $s_{1}s_{2}$ & $0$ \\ 
$\bq_{8}$ & $s_{2}s_{1}$ & $\alpha_{2}^{\vee}$ \\ 
$\bq_{9}$ & $s_{2}s_{1}$ & $\alpha_{2}^{\vee}$ \\ 
$\bq_{10}$ & $s_{2}s_{1}$ & $0$ \\ 
$\bq_{11}$ & $s_{1}s_{2}s_{1}$ & $0$ \\ 
$\bq_{12}$ & $s_{2}s_{1}s_{2}$ & $0$ \\ \hline
\end{tabular}
\end{minipage}
\end{table}

Note that $\bp \in \CP(v, \Pi)$ and $\bq \in \CP(v, \Pi')$ are explicitly written as follows: 
\begin{alignat*}{2}
\bp_{1} &: s_{2} \xrightarrow{\alpha_{2}} e; \quad & \bp_{2} &: s_{2} \quad \text{(the trivial directed path)}; \\ 
\bp_{3} &: s_{2} \xrightarrow{\alpha_{1}+\alpha_{2}} s_{1}s_{2}; \quad & \bp_{4} &: s_{2} \xrightarrow{\alpha_{1}} s_{2}s_{1}; \displaybreak[1] \\
\bp_{5} &: s_{2} \xrightarrow{\alpha_{1}} s_{2}s_{1} \xrightarrow{\alpha_{1}+\alpha_{2}} s_{1}s_{2}s_{1}; \quad & \bp_{6} &: s_{2} \xrightarrow{\alpha_{1}} s_{2}s_{1} \xrightarrow{\alpha_{2}} s_{2}s_{1}s_{2}; \displaybreak[1] \\ 
\bq_{1} &: s_{2} \xrightarrow{\alpha_{1}+\alpha_{2}} s_{1}s_{2} \xrightarrow{\alpha_{2}} s_{1} \xrightarrow{\alpha_{1}} e; \quad & \bq_{2} &: s_{2} \xrightarrow{\alpha_{1}+\alpha_{2}} s_{1}s_{2} \xrightarrow{\alpha_{1}} s_{1}s_{2}s_{1} \xrightarrow{2\alpha_{1}+\alpha_{2}} e; \displaybreak[1] \\ 
\bq_{3} &: s_{2} \xrightarrow{\alpha_{2}} e; \quad & \bq_{4} &: s_{2} \xrightarrow{\alpha_{1}+\alpha_{2}} s_{1}s_{2} \xrightarrow{\alpha_{2}} s_{1}; \displaybreak[1] \\ 
\bq_{5} &: s_{2} \xrightarrow{\alpha_{2}} e \xrightarrow{\alpha_{1}} s_{1}; \quad & \bq_{6} &: s_{2} \quad \text{(the trivial directed path)}; \displaybreak[1] \\ 
\bq_{7} &: s_{2} \xrightarrow{\alpha_{1}+\alpha_{2}} s_{1}s_{2}; \quad & \bq_{8} &: s_{2} \xrightarrow{\alpha_{1}+\alpha_{2}} s_{1}s_{2} \xrightarrow{\alpha_{2}} s_{1} \xrightarrow{2\alpha_{1}+\alpha_{2}} s_{2}s_{1}; \displaybreak[1] \\ 
\bq_{9} &: s_{2} \xrightarrow{\alpha_{2}} e \xrightarrow{\alpha_{1}} s_{1} \xrightarrow{2\alpha_{1}+\alpha_{2}} s_{2}s_{1}; \quad & \bq_{10} &: s_{2} \xrightarrow{\alpha_{1}} s_{2}s_{1}; \displaybreak[1] \\ 
\bq_{11} &: s_{2} \xrightarrow{\alpha_{1}+\alpha_{2}} s_{1}s_{2} \xrightarrow{\alpha_{1}} s_{1}s_{2}s_{1}; \quad & \bq_{12} &: s_{2} \xrightarrow{\alpha_{1}+\alpha_{2}} s_{1}s_{2} \xrightarrow{2\alpha_{1}+\alpha_{2}} s_{2}s_{1}s_{2}. 
\end{alignat*}
Thus, if we set
\begin{align*}
\CP_{0}(v, \Pi) &:= \CP(v, \Pi), \\ 
\CP_{0}(v, \Pi') &:= \{ \bq_{3}, \bq_{6}, \bq_{7}, \bq_{10}, \bq_{11}, \bq_{12} \} \subset \CP(v, \Pi'), \\ 
\CP_{0}^{C}(v, \Pi') &:= \{ \bq_{1}, \bq_{2}, \bq_{4}, \bq_{5}, \bq_{8}, \bq_{9} \} = \CP(v, \Pi') \setminus \CP_{0}(v, \Pi'),
\end{align*}
then we obtain the following bijection $Y^{v, \Pi}: \CP_{0}(v, \Pi) \rightarrow \CP_{0}(v, \Pi')$ and involution $I_{2}^{v, \Pi'}$ on $\CP_{0}^{C}(v, \Pi')$, 
which preserve $\ed(\cdot)$ and $\down(\cdot)$: 
\begin{alignat*}{6}
Y^{v, \Pi} &: \bp_{1} \mapsto \bq_{3}, \quad & &\bp_{2} \mapsto \bq_{6}, \quad & &\bp_{3} \mapsto \bq_{7}, \quad & &\bp_{4} \mapsto \bq_{10}, \quad & &\bp_{5} \mapsto \bq_{11}, \quad & &\bp_{6} \mapsto \bq_{12}; \\ 
I_{2}^{v, \Pi'} &: \bq_{1} \mapsto \bq_{2}, \quad & &\bq_{2} \mapsto \bq_{1}, \quad & &\bq_{4} \mapsto \bq_{5}, \quad & &\bq_{5} \mapsto \bq_{4}, \quad & &\bq_{8} \mapsto \bq_{9}, \quad & &\bq_{9} \mapsto \bq_{8}. 
\end{alignat*}
These maps give the correspondence $\bp \mapsto \bp'$ in Proposition~\ref{prop:rank2_shellability}. 

Now, let us give an example of generalized quantum Yang-Baxter moves. 
Let $\lambda \in P$. Take $\lambda$-chains $\Gamma_1$ and $\Gamma_2$ 
such that $\Gamma_{2}$ is obtained from $\Gamma_{1}$ by the Yang-Baxter transformation (YB). 
Let $w \in W$. 
As in equations \eqref{eq:division_Gamma1} and \eqref{eq:division_Gamma2}, we take $\Gamma_{1}^{(k)}$, $\Gamma_{2}^{(k)}$, $k = 1, 2, 3$. 
Also, as in equations \eqref{eq:division_A} and \eqref{eq:division_B}, 
we take $A^{(k)}$ (resp., $B^{(k)}$), $k = 1, 2, 3$, for $A \in \CA(w, \Gamma_{1})$ (resp., $B \in \CA(w, \Gamma_{2})$). 
In this example, we consider the case that $\Gamma_{1}^{(2)} = \Pi$ and $\Gamma_{2}^{(2)} = \Pi'$. 
By the consideration above, we can give an explicit description of quantum Yang-Baxter moves 
for $A \in \CA(w, \Gamma_{1})$ (resp., $B \in \CA(w, \Gamma_{2})$) such that $\ed(A^{(1)}) = s_{2}$ (resp., $\ed(B^{(1)}) = s_{2}$), 
as given in Tables~\ref{tab:explicit_YB-move1}, \ref{tab:explicit_YB-move2}. 

\begin{table}[htbp]
\centering
\caption{List of $Y(A)$ for $A \in \CA_{0}(w, \Gamma_{1})$ such that $\ed(A^{(1)}) = s_{2}$}
\label{tab:explicit_YB-move1}
\begin{tabular}{|c||c|c||c|} \hline 
$A^{(2)}$ & $\bp(A^{(2)})$ & $\bp(Y(A)^{(2)}) = Y^{v, \Pi}(\bp(A^{(2)}))$ & $Y(A)^{(2)}$ \\ \hline
$\emptyset$ & $\bp_{2}$ & $\bq_{6}$ & $\emptyset$ \\ 
$\{t+2\}$ & $\bp_{4}$ & $\bq_{10}$ & $\{t+3\}$ \\
$\{t+3\}$ & $\bp_{1}$ & $\bq_{3}$ & $\{t+2\}$ \\ 
$\{t+4\}$ & $\bp_{3}$ & $\bq_{7}$ & $\{t+1\}$ \\ 
$\{t+2, t+3\}$ & $\bp_{6}$ & $\bq_{12}$ & $\{t+1, t+4\}$ \\ 
$\{t+2, t+4\}$ & $\bp_{5}$ & $\bq_{11}$ & $\{t+1, t+3\}$ \\ \hline
\end{tabular}
\end{table}

\begin{table}[htbp]
\centering
\caption{List of $I_{2}(B)$ for $B \in \CA_{0}^{C}(w, \Gamma_{2})$ such that $\ed(B^{(1)}) = s_{2}$}
\label{tab:explicit_YB-move2}
\begin{tabular}{|c||c|c||c|} \hline
$B^{(2)}$ & $\bp(B^{(2)})$ & $\bp(I_{2}(B)^{(2)}) = I_{2}^{v, \Pi'}(\bp(B^{(2)}))$ & $I_{2}(B)^{(2)}$ \\ \hline 
$\{t+1, t+2\}$ & $\bq_{4}$ & $\bq_{5}$ & $\{t+2, t+3\}$ \\ 
$\{t+2, t+3\}$ & $\bq_{5}$ & $\bq_{4}$ & $\{t+1, t+2\}$ \\ 
$\{t+1, t+2, t+3\}$ & $\bq_{1}$ & $\bq_{2}$ & $\{t+1, t+3, t+4\}$ \\ 
$\{t+1, t+2, t+4\}$ & $\bq_{8}$ & $\bq_{9}$ & $\{t+2, t+3, t+4\}$ \\ 
$\{t+1, t+3, t+4\}$ & $\bq_{2}$ & $\bq_{1}$ & $\{t+1, t+2, t+3\}$ \\ 
$\{t+2, t+3, t+4\}$ & $\bq_{9}$ & $\bq_{8}$ & $\{t+1, t+2, t+4\}$ \\ \hline 
\end{tabular}
\end{table}

%===================%
% START SECTION 0C %
%===================%

%%%%%%%%%%%%
% sec:App_C %
%%%%%%%%%%%%

\section{The right-hand side of the identity of Chevalley type for graded characters.} \label{sec:App_C}
We show that the right-hand side of \eqref{eq:PC-type_formula} is identical to zero if $\mu + \lambda \notin P^{+}$. 
\begin{prop}\label{prop:sum_equals_0}
Let $\mu \in P^{+}$, and $x = wt_{\xi} \in W_{\af}$ with $w \in W$ and $\xi \in Q^{\vee}$. 
Take $\lambda \in P$ such that $\mu + \lambda \notin P^{+}$, and let $\Gamma$ be an arbitrary reduced $\lambda$-chain. 
Then we have
\begin{equation*}
\sum_{A \in \CA(w, \Gamma)} \sum_{\bchi \in \bPar(\lambda)} (-1)^{n(A)} q^{-\height(A) - \pair{\lambda}{\xi} - |\bchi|} e^{\wt(A)} \gch V_{\ed(A)t_{\xi + \down(A) + \iota(\bchi)}}^{-}(\mu) = 0. 
\end{equation*}
\end{prop}

In the proof of Proposition~\ref{prop:sum_equals_0}, we make use of the following equalities for graded characters. 
\begin{prop}[{\cite[Proposition~D.1]{KNS}}]\label{prop:character_translation}
For each $x \in W_{\af}$, $\xi \in Q^{\vee}$, and $\lambda \in P^{+}$, one has 
\begin{equation*}
\gch V_{xt_{\xi}}^{-}(\lambda) = q^{-\pair{\lambda}{\xi}} \gch V_{x}^{-}(\lambda). 
\end{equation*}
\end{prop}

\begin{prop}[{cf. \cite[Appendix~B]{NOS}}]\label{prop:sum_equals_0_anti-dominant}
Let $\mu \in P^{+}$ and $x \in W$. 
Take $\lambda \in -P^{+}$ such that $\mu + \lambda \notin P^{+}$, and let $\Gamma$ be the lex $\lambda$-chain. 
Then we have 
\begin{equation*}
\sum_{A \in \CA(x, \Gamma)} (-1)^{|A|} q^{-\height(A)} e^{\wt(A)} \gch V_{\ed(A) t_{\down(A)}}^{-}(\mu) = 0. 
\end{equation*}
\end{prop}

\begin{rem}
\begin{enu}
\item In \cite{NOS}, Proposition~\ref{prop:sum_equals_0_anti-dominant} is stated and proved in terms of semi-infinite Lakshmibai-Seshadri paths. 
\item By Proposition~\ref{prop:character_translation}, we see that Proposition~\ref{prop:sum_equals_0_anti-dominant} also holds for $x \in W_{\af}$. 
\end{enu}
\end{rem}

\begin{proof}[Proof of Proposition~\ref{prop:sum_equals_0}]
By considering $\lambda^{\pm}$, $\Gamma^{\pm}$, $\Gamma_{0}$, and by using Theorems~\ref{comm_mu_nu_par} and \ref{thm:generating_function_reduced} as in the proof of Theorem~\ref{thm:PC-type_formula} (cf.~\eqref{proof_PC}), we obtain:
\begin{align}
& \sum_{A \in \CA(w, \Gamma)} \sum_{\bchi \in \bPar(\lambda)} (-1)^{n(A)} q^{-\height(A) - \pair{\lambda}{\xi} - |\bchi|} e^{\wt(A)} \gch V_{\ed(A)t_{\xi + \down(A) + \iota(\bchi)}}^{-}(\mu) \nonumber \\[3mm] 
&= \sum_{A \in \CA(w, \Gamma_{0})} \sum_{\bchi \in \bPar(\lambda)} (-1)^{n(A)} q^{-\height(A) - \pair{\lambda}{\xi} - |\bchi|} e^{\wt(A)} \gch V_{\ed(A)t_{\xi + \down(A) + \iota(\bchi)}}^{-}(\mu) \nonumber \\[3mm] 
\begin{split}
&= \sum_{A \in \CA(w, \Gamma^{+})} \sum_{B \in \CA(\ed(A), \Gamma^{-})} \sum_{\bchi \in \bPar(\lambda^{+})} (-1)^{|B|} q^{-\height(A)-\height(B)-\pair{\lambda^{-}}{\down(A)+\iota(\bchi)}-\pair{\lambda}{\xi}-|\bchi|} \\ 
& \hspace*{64mm} \times e^{\wt(A) + \wt(B)} \gch V_{\ed(B) t_{\xi + \down(A) + \down(B) + \iota(\bchi)}}^{-}(\mu)
\end{split} \nonumber \\[3mm] 
\begin{split}
&= \sum_{A \in \CA(w, \Gamma^{+})} \sum_{B \in \CA(\ed(A), \Gamma^{-})} \sum_{\bchi \in \bPar(\lambda^{+})} (-1)^{|B|} q^{-\height(A)-\height(B)-\pair{\lambda^{-}}{\down(A)+\iota(\bchi)}-\pair{\lambda}{\xi}-|\bchi|} \\ 
& \hspace*{64mm} \times q^{-\pair{\mu}{\xi + \down(A) + \iota(\bchi)}} e^{\wt(A) + \wt(B)} \gch V_{\ed(B) t_{\down(B)}}^{-}(\mu)
\end{split} \nonumber \\[3mm] 
\begin{split}
&= q^{-\pair{\mu + \lambda}{\xi}} \sum_{A \in \CA(w, \Gamma^{+})} \sum_{\bchi \in \bPar(\lambda^{+})} q^{-\height(A) - \pair{\lambda^{-}+\mu}{\down(A) + \iota(\bchi)} - |\bchi|} e^{\wt(A)} \\ 
& \hspace*{45mm} \times \sum_{B \in \CA(\ed(A), \Gamma^{-})} (-1)^{|B|} q^{-\height(B)}  e^{\wt(B)} \gch V_{\ed(B) t_{\down(B)}}^{-}(\mu);  
\end{split} \label{eq:sum_equals_0_middle}
\end{align}
here the third equality follows by Proposition~\ref{prop:character_translation}. 
Since $\mu + \lambda \notin P^{+}$, it follows that $\mu + \lambda^{-} \notin P^{+}$. 
Therefore, we deduce by Proposition~\ref{prop:sum_equals_0_anti-dominant} that 
\begin{equation*}
\sum_{B \in \CA(\ed(A), \Gamma^{-})} (-1)^{|B|} q^{-\height(B)}  e^{\wt(B)} \gch V_{\ed(B) t_{\down(B)}}^{-}(\mu) = 0
\end{equation*}
for each $A \in \CA(w, \Gamma^{+})$, and hence that \eqref{eq:sum_equals_0_middle} is identical to zero, as needed. 
\end{proof}

%====================%
%     BIBLIOGRAPHY    %
%====================%


\begin{thebibliography}{LNS$^{3}$1}
\bibitem[BFP]{BFP}
F. Brenti, S. Fomin, and A. Postnikov, 
Mixed Bruhat operators and Yang-Baxter equations for Weyl groups, 
\textit{Int. Math. Res. Not.} \textbf{1999} (2001), 419--441. 

\bibitem[D]{D}
M. J. Dyer, 
Hecke algebras and shellings of Bruhat intervals, 
\textit{Compos. Math.} \textbf{89} no. 1 (1993), 91--115. 

\bibitem[FK]{FK}
I. Fischer and M. Konvalinka, 
A bijective proof of the ASM theorem, Part I: The operator formula, 
\textit{Electron. J. Combin.} \textbf{27} (2020), Paper P3.35. 

\bibitem[FM]{FM}
E. Feigin and I. Makedonskyi, 
Generalized Weyl modules, alcove paths and Macdonald polynomials, 
Selecta Math. (New Series) \textbf{23} (2017), 2863--2897. 

\bibitem[HK]{HK}
J.~Hong and S.-J. Kang,
\textit{Introduction to {Q}uantum {G}roups and {C}rystal {B}ases},
  volume~42 of \textit{Graduate Studies in Mathematics}, 
Amer. Math. Soc., Providence, RI, 2002.

\bibitem[INS]{INS}
M.~Ishii, S.~Naito, and D.~Sagaki,
Semi-infinite {L}akshmibai-{S}eshadri path model for level-zero
  extremal weight modules over quantum affine algebras,
\textit{Adv. Math.} \textbf{290} (2016), 967--1009.

\bibitem[Kac]{Kac}
V. C. Kac, 
\textit{Infinite Dimensional Lie Algebras, 3rd Edition}, 
Cambridge University Press, Cambridge, UK, 1990. 

\bibitem[Kas]{Kas}
M. Kashiwara, 
Crystal bases of modified quantized enveloping algebra, 
\textit{Duke Math. J.} \textbf{73} (1994), 383--413. 

\bibitem[KNS]{KNS}
S. Kato, S. Naito, and D. Sagaki, 
Equivariant $K$-theory of semi-infinite flag manifolds and the Pieri-Chevalley formula, 
\textit{Duke Math. J.} \textbf{169} (2020), 2421--2500. 

\bibitem[L]{L}
C. Lenart, 
On the combinatorics of crystal graphs, I. Lusztig's involution, 
\textit{Adv. Math.} \textbf{211} (2007), 204--243.

\bibitem[LL1]{LL1}
C. Lenart and A. Lubovsky, 
A generalization of the alcove model and its applications, 
\textit{J. Algebr. Comb.} \textbf{41} (2015), 751--783. 

\bibitem[LL2]{LL2}
C. Lenart and A. Lubovsky, 
A uniform realization of the combinatorial $R$-matrix for column shape Kirillov-Reshetikhin crystals, 
\textit{Adv. Math.} \textbf{334} (2018), 151--183. 

\bibitem[LNS]{LNS}
C. Lenart, S. Naito, and D. Sagaki, 
A Chevalley formula for semi-infinite flag manifolds and quantum $K$-theory, 
arXiv:2010.06143. 

\bibitem[LNS$^{3}$1]{LNSSS1}
C. Lenart, S. Naito, D. Sagaki, A. Schilling, and M. Shimozono, 
A uniform model for Kirillov-Reshetikhin crystals I: Lifting the parabolic quantum Bruhat graph, 
\textit{Int. Math. Res. Not.} \textbf{2015} (2015), 1848--1901. 

\bibitem[LNS$^{3}$2]{LNSSS2}
C. Lenart, S. Naito, D. Sagaki, A. Schilling, and M. Shimozono, 
A uniform model for Kirillov-Reshetikhin crystals II: Alcove model, path model, and $P=X$, 
\textit{Int. Math. Res. Not.} \textbf{2017} (2017), 4259--4319. 

\bibitem[LNS$^{3}$3]{LNSSS3}
C. Lenart, S. Naito, D. Sagaki, A. Schilling, and M. Shimozono, 
A uniform model for Kirillov-Reshetikhin crystals III: Nonsymmetric Macdonald polynomials at $t = 0$ and Demazure characters, 
Transform. Groups \textbf{22} (2017), 1041--1079. 

\bibitem[LP1]{LP1}
C. Lenart and A. Postnikov, 
Affine Weyl groups in $K$-theory and representation theory, 
\textit{Int. Math. Res. Not.} \textbf{12} (2007), Art. ID rnm038, 65. 

\bibitem[LP2]{LP2}
C. Lenart and A. Postnikov, 
A combinatorial model for crystals of Kac-Moody algebras, 
\textit{Trans. Amer. Math. Soc.} \textbf{360} (2008), 4349--4381. 

\bibitem[N]{N}
F. Nomoto, 
Generalized Weyl modules and Demazure submodules of level-zero extremal weight modules, 
arXiv:1701.08377. 

\bibitem[NOS]{NOS}
S. Naito, D. Orr, and D. Sagaki, 
Chevalley formula for anti-dominant weights in the equivariant $K$-theory of semi-infinite flag manifolds, 
to appear in Adv. Math., arXiv:1808.01468. 

\bibitem[P]{P}
A. Postnikov, 
Quantum Bruhat graph and Schubert polynomials, 
\textit{Proc. Amer. Math. Soc.} \textbf{133} no. 3 (2005), 699--709. 

\end{thebibliography}
\end{document}